\documentclass{article}
\usepackage{graphicx}
\usepackage{authblk}
\usepackage{color}
\usepackage{amsmath, amsthm, amssymb}
\usepackage{bm}
\usepackage{enumitem}
\usepackage{bbm}
\usepackage{blindtext}
\usepackage{titlesec, blindtext, color}
\usepackage{amsthm}
\usepackage[a4paper, total={5.8in, 8in}]{geometry}
\usepackage{mathtools}
\usepackage{tikz}
\usepackage{caption}
\usepackage{subcaption}
\usepackage{tikz-cd}
\usepackage{varwidth}
\usepackage{tasks}
 \usepackage [latin1]{inputenc}
\usepackage{hyperref}
\hypersetup{
colorlinks=true,
citecolor=black,
linkcolor=black,
linktoc=all
}
\usepackage{cleveref}

\newtheorem{theorem}{Theorem}[section]

\newtheorem{lemma}[theorem]{Lemma}

\newtheorem{corollary}[theorem]{Corollary}

\newtheorem{claim}[theorem]{Claim}
\theoremstyle{definition}
\newtheorem{definition}[theorem]{Definition}
\newtheorem{remark}[theorem]{Remark}

\newtheorem{question}[theorem]{Question}
\newtheorem{fact}[theorem]{Fact}

\newcommand{\Z}{\mathbb{Z}}
\begin{document}

\title{Non-equilibrium multi-scale analysis and coexistence in competing first passage percolation}
\author{Thomas Finn\footnote{t.j.finn@bath.ac.uk, University of Bath, Deptartment of Mathematical Sciences, supported by a scholarship from the EPSRC Centre for Doctoral Training in Statistical Applied Mathematics at Bath (SAMBa), under the project EP/L015684/1 and EB/MA1206A.} \quad Alexandre Stauffer\footnote{
astauffer@mat.uniroma3.it, Universit\`a Roma Tre, Dipartimento di Matematica e Fisica; University of Bath, Department of Mathematical Sciences, supported by EPSRC Fellowship EP/N004566/1.}}
\date{}

\maketitle

\begin{abstract}
   The main contribution of this paper is the development of a novel approach to multi-scale analysis that we believe can be used to analyse processes with non-equilibrium dynamics. 
   Our approach will be referred to as \emph{multi-scale analysis with non-equilibrium feedback} and will be used to analyse a natural random growth process with competition 
   on $\mathbb{Z}^d$ called \emph{first passage percolation in a hostile environment} (FPPHE)
   that consists of two first passage percolation processes $FPP_1$ and $FPP_{\lambda}$ that compete for the occupancy of sites. 
%    Two fundamental challenges of this model 
%    are the absence of monotonicity and its intrinsic non-equilibrium dynamics, which our novel approach is able to handle.
   Initially, $FPP_1$ occupies the origin and spreads through the edges of $\mathbb{Z}^d$ at rate 1, while $FPP_{\lambda}$ is initialised at sites called \emph{seeds} that are distributed according 
   to a product of Bernoulli measures of parameter $p\in(0,1)$, where a seed remains dormant until $FPP_1$ or $FPP_{\lambda}$ attempts to occupy it 
   before then spreading through the edges of $\mathbb{Z}^d$ at rate $\lambda>0$. 
   Two fundamental challenges of FPPHE that our approach is able to handle 
   are the absence of monotonicity (for instance, adding seeds could be benefitial to $FPP_1$ instead of $FPP_\lambda$) 
   and its non-equilibrium dynamics; such characteristics, for example,   
   prevent the application of a more standard multi-scale analysis.
   As a consequence of our main result for FPPHE, we establish a coexistence phase for $d\geq3$, answering an open question in \cite{sidoravicius2019multi}. 
   This exhibits a rare situation where a natural random competition model on $\mathbb{Z}^d$ observes coexistence for processes with \emph{different} speeds. 
   Moreover, we are able to establish the stronger result that $FPP_1$ and $FPP_{\lambda}$ can both occupy a \emph{positive density} of sites with positive probability, 
   which is in stark contrast with other competition processes.
\end{abstract}

\section{Introduction}

In this paper we consider competing random spatial growth on $\mathbb{Z}^d$ and study the phenomena of \emph{coexistence}, in which the two types simultaneously grow without bound.
Our focus is competing first passage percolation which has been studied for over twenty years and many intriguing questions remain open (see Section~\ref{sec:related_models} for more detail).
Understanding for what choice of parameters coexistence can occur is a fundamental question and is often difficult due to intrinsic lack of monotonicity (both processes have to grow indefinitely without blocking 
the other one) and long-range dependencies in such models.

To overcome difficulties inherent in competing first passage percolation, we introduce a novel approach to a widely used technique known as multi-scale analysis (or multi-scale renormalisation). 
We refer to our approach as \emph{multi-scale analysis with non-equilibrium feedback} and apply it 
to analyse a challenging growth process known as \emph{first passage percolation in a hostile environment} (FPPHE)~\cite{sidoravicius2019multi}, in which two 
types concurrently grow inside $\Z^d$, occupying the sites of $\Z^d$ while competing with each other for space.
As a consequence of our main result, we establish the regime of coexistence for FPPHE on $\mathbb{Z}^d$ that unlike similar models occurs when the two types grow at \emph{different} rates.
% We analyse a growth process known as \emph{first passage percolation in a hostile environment} (FPPHE)~\cite{sidoravicius2019multi}, in which two 
% types concurrently grow inside $\Z^d$, occupying the sites of $\Z^d$ while competing with each other for space.
% We establish the regime of \emph{coexistence} for FPPHE on $\mathbb{Z}^d$, which unlike similar growth models occurs when the two types grow at \emph{different} rates.
% In order to establish this result, we introduce a novel approach to a widely used technique known as multi-scale analysis (or multi-scale renormalisation).
% We regard this as a central contribution of our work and refer to it as \emph{multi-scale analysis with non-equilibrium feedback}.

Multi-scale analysis is a powerful technique that has been used to analyse a wide range of difficult processes, especially those with \emph{slow decay of correlations}; recent examples include 
random interlacements, dependent percolation processes, and interacting particle systems~\cite{kesten2003branching,kesten2005spread,kesten2006,sznitman2010,peres2013,candellero2018,gracar2019}.
Despite its power, multi-scale analyses usually rely strongly on certain crucial properties, such as \emph{stationarity} (that is, the process is \emph{in
equilibrium}) and \emph{monotonicity}. Such properties are present in all the aforementioned works, but are not satisfied by FPPHE and several other important processes, 
one additional example being 
a non-equilibrium model for the spread of an infection among random walk particles, 
whose analysis is highlighted as a fundamental challenge in~\cite{kesten2005spread,kesten2006}; a similar question was recently raised in~\cite{grimmett2020brownian}.
We believe our idea could be applicable to analyse this and other processes with non-equilibrium dynamics. 

Below we first introduce FPPHE and discuss our results for that model, 
and then turn to presenting the idea behind our multi-scale analysis with non-equilibrium feedback.

\subsection{First passage percolation in a hostile environment (FPPHE)}
\label{sec:FPPHE}

FPPHE is a natural random growth process with competition introduced by Sidoravicius and Stauffer~\cite{sidoravicius2019multi} to analyse a notoriously challenging random aggregation model called \emph{multiparticle diffusion limited aggregation} (MDLA).
Nonetheless, FPPHE is an interesting model to analyse in its own right, showing several phase transitions and a wide range of behaviour. 
In fact, most of the work in \cite{sidoravicius2019multi} is devoted to studying FPPHE, from which results about MDLA can be deduced. 

In FPPHE,
there are two first passage percolation processes $FPP_1$ and $FPP_{\lambda}$ that compete for the occupancy of sites on $\mathbb{Z}^d$. 
Initially the $FPP_1$ process only occupies the origin, while, following a product of Bernoulli measures of parameter $p\in(0,1)$ on all other sites of $\Z^d$, 
we place \emph{seeds} from which the $FPP_{\lambda}$ will eventually attempt to grow. 
From these initial conditions, FPPHE evolves in time as follows. 
The $FPP_1$ process spreads through the edges of $\mathbb{Z}^d$ at rate $1$. 
This means that, for any given edge $e$, when the $FPP_1$ process occupies for the first time one of the endpoints of $e$, 
then after a random amount of time (that is distributed as an exponential random variable of rate 1), 
the $FPP_1$ process attempts to occupy the other endpoint of $e$, with the attempt being successful if and only if that endpoint is not already occupied by either process. 
The $FPP_{\lambda}$ process, in turn, does not initially spread but remains dormant. 
When either the $FPP_1$ or $FPP_{\lambda}$ process attempts to occupy a site that has a seed, the attempt fails and the seed is \emph{activated}, 
meaning that the $FPP_{\lambda}$ process then spreads from that seed through the edges of $\mathbb{Z}^d$ at rate $\lambda>0$. 
The other seeds remain dormant until they are activated. 
Once a site is occupied by either process, 
it is occupied by that process henceforth and thus $FPP_1$ and $FPP_{\lambda}$ are in competition for the occupancy of sites of $\mathbb{Z}^d$. 

Initially $FPP_1$ has the advantage as it is able to spread from time $0$ while $FPP_{\lambda}$ remains dormant in seeds until activated by either process. 
On the other hand, $FPP_{\lambda}$ has the advantage of occupying infinitely many seeds initially
and so the spread of $FPP_{\lambda}$ from any of these seeds could allow $FPP_{\lambda}$ to block the spread of $FPP_1$. Understanding how this delicate balance 
is exhibited under different choices of $\lambda$ and $p$ is fundamental in the analysis of FPPHE.

Figure~\ref{fig:simul} illustrates a simulation of FPPHE in two dimensions for two different values of $p$ and $\lambda$.
We highlight that in Figure~\ref{fig:psmall} the small value of $p$ and relatively large value of $\lambda$ makes $FPP_1$ grow `almost `radially'' from the origin, 
but even if $\lambda$ is not very close to $1$, $FPP_\lambda$ manages to 
conquer a quite large region, highlighting the strong dependences in the process. 
On the other hand, when $p$ is close to $1-p_c^\text{site}$ (where $p_c^{\text{site}}=p_c^{\text{site}}(d)$ is the critical probability for independent site percolation on 
$\mathbb{Z}^d$), as in Figure~\ref{fig:plarg}, 
$FPP_1$ gives rise to a more irregular shape as it needs to deviate from possible blocks that $FPP_\lambda$ may create as it eventually occupies part of the thin set of non-seed sites. 
(We refer the reader to \cite{grimmett1999grundlehren} for a background on percolation.)
This makes regions of $\Z^d$ that are close to one another be occupied by the process at times that are far apart.
For example, in Figure~\ref{fig:plarg}, $FPP_\lambda$ blocked the growth of $FPP_1$ to the left of the origin (the origin is close to the middle of the picture, where the black region is), 
forcing $FPP_1$ to make a long detour to reach regions to the left of the origin. 

\begin{figure}
\centering
\begin{subfigure}{.5\textwidth}
  \centering
  \includegraphics[width=.95\linewidth]{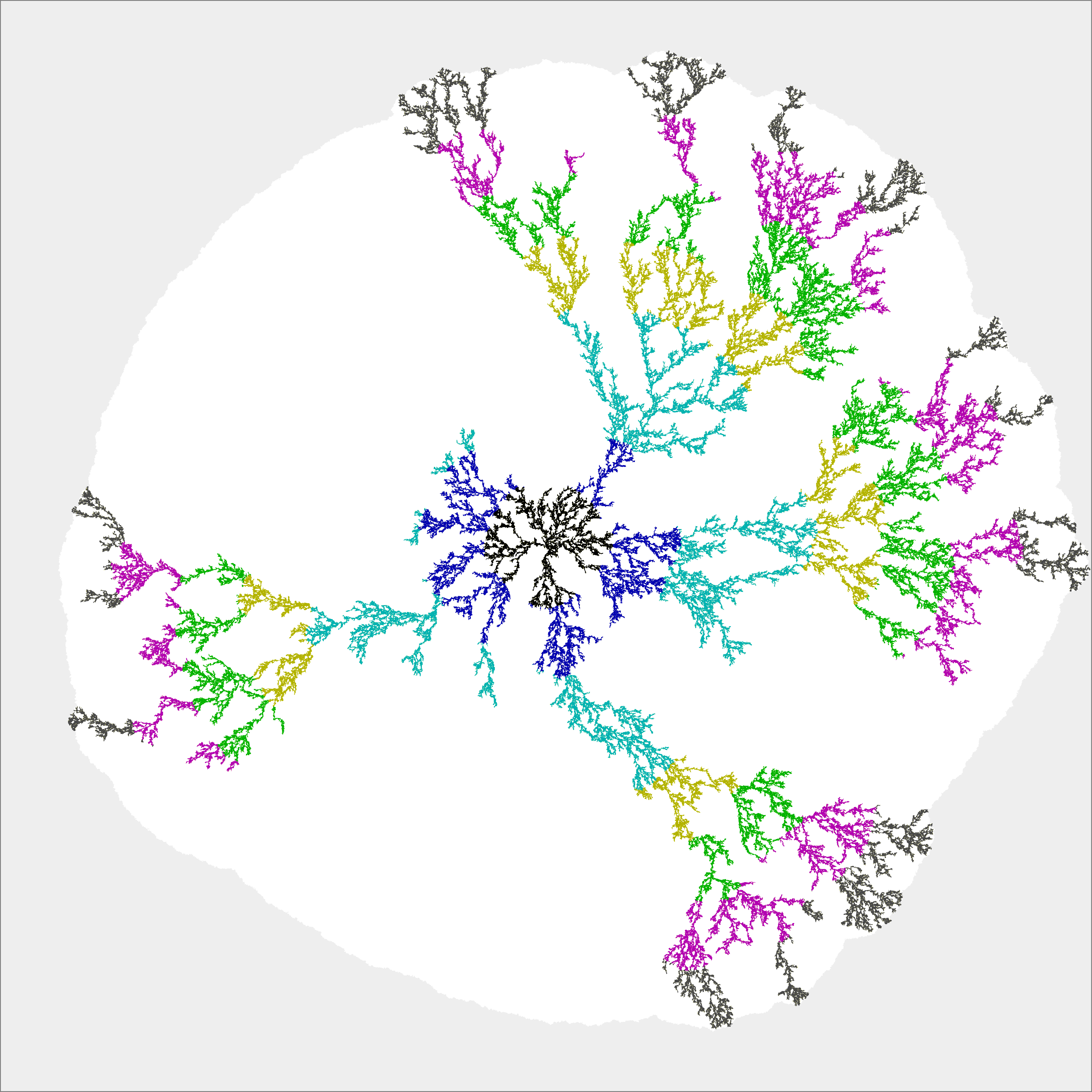}
  \caption{$p=0.03$, $\lambda=0.7$}
  \label{fig:psmall}
\end{subfigure}%
\begin{subfigure}{.5\textwidth}
  \centering
  \includegraphics[width=.95\linewidth]{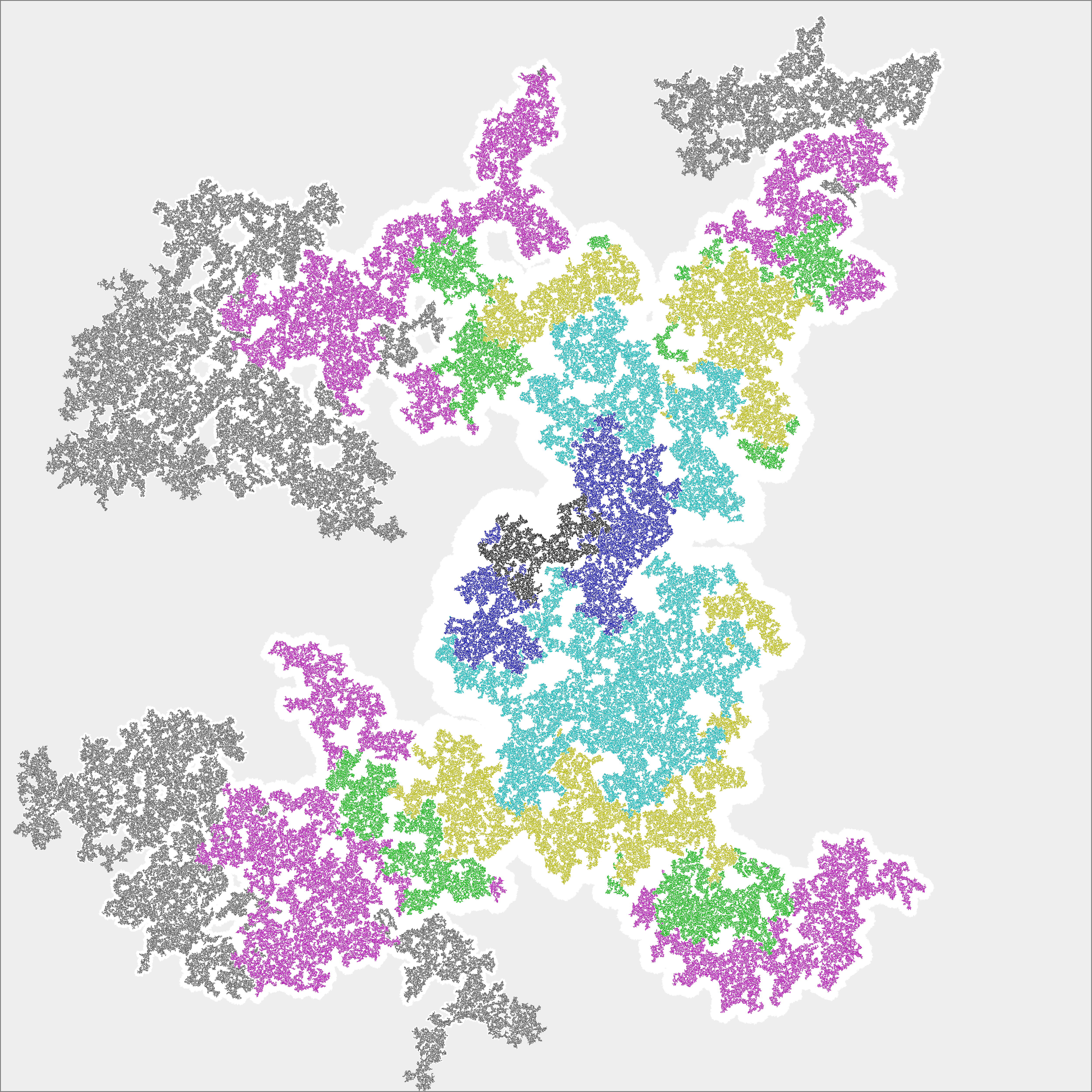}
  \caption{$p=0.4$, $\lambda=0.008$}
  \label{fig:plarg}
\end{subfigure}
\caption{Simulation of FPPHE in two dimensions. The colored areas represent the sites occupied by $FPP_1$, with different colors representing different epochs of the evolution of $FPP_1$. 
   The white area represents the sites occupied by $FPP_\lambda$, while the light gray boundary area represents unoccupied sites, including inactive seeds. }
\label{fig:simul}
\end{figure}

We say that $FPP_1$ (resp.\ $FPP_{\lambda}$) \emph{survives} if in the limit as time goes to infinity $FPP_1$ (resp.\ $FPP_{\lambda}$) occupies an infinite \emph{connected} 
component of $\mathbb{Z}^d$. 
Note that $FPP_{\lambda}$ initially has infinitely many seeds, so the requirement of a connected component is natural. 
If a process does not survive, we say that it \emph{dies out}. 
Hence, if $FPP_{\lambda}$ dies out, then $FPP_{\lambda}$ is an infinite collection of connected regions, each of which is almost surely of finite size. 
When both processes survive, we say that \emph{coexistence} occurs or that $FPP_1$ and $FPP_{\lambda}$ \emph{coexist}. 

% According to the survival of either process and the fact that at least one process must survive, 
As at least one process must survive, 
there are three possible phases that FPPHE can observe. The first is the \emph{extinction phase} where $FPP_1$ dies out almost surely (and thus $FPP_{\lambda}$ survives almost surely). 
Note that there is always a positive probability of $FPP_1$ dying out\footnote{For example, with probability $p^{2d}>0$, 
every site neighbouring the origin hosts a seed and, consequently, $FPP_1$ dies out.}, thus the almost sure criterion is essential in the definition of the extinction phase. 
Moreover, if $p>1-p_c^{\text{site}}$, then for any $\lambda>0$, 
FPPHE is in the extinction phase as the origin is almost surely contained in a finite cluster of non-seeds. Consequently, FPPHE is most interesting when $p < 1-p_c^{\text{site}}$. 
Similarly, $FPP_1$ dies out almost surely if $\lambda \geq 1$ (see Section~\ref{sec:known_results} for more details) and so we focus on the case $\lambda < 1$.
The second phase is when, with positive probability, $FPP_1$ survives and $FPP_{\lambda}$ dies out, which we refer to as the \emph{strong survival phase}. 
Ultimately, if with positive probability both $FPP_1$ and $FPP_{\lambda}$ survive, we say that FPPHE is in the \emph{coexistence phase}. 
For $d=1$ it is immediate that neither the strong survival or coexistence phase can exist. 
% Classifying the behaviour of FPPHE in higher dimensions is the primary focus of our analysis. 

\begin{figure}
\centering
\begin{subfigure}{.33\textwidth}
  \centering
  \includegraphics[width=.8\linewidth]{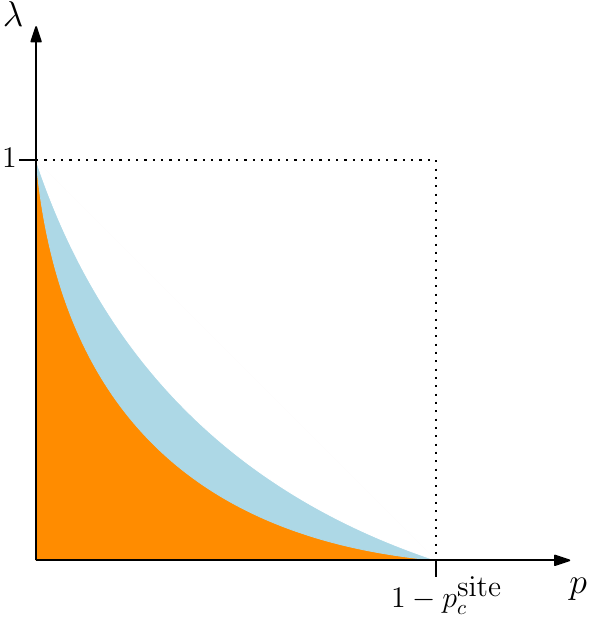}
  \caption{$d=2$}
  \label{fig:phase1a}
\end{subfigure}%
\begin{subfigure}{.33\textwidth}
  \centering
  \includegraphics[width=.8\linewidth]{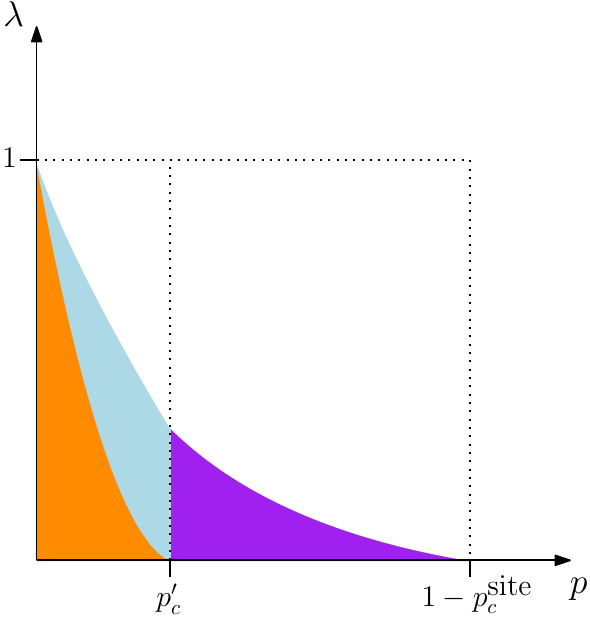}
  \caption{$d\geq3$}
  \label{fig:phase1b}
\end{subfigure}
\begin{subfigure}{.33\textwidth}
  \centering
  \includegraphics[width=.8\linewidth]{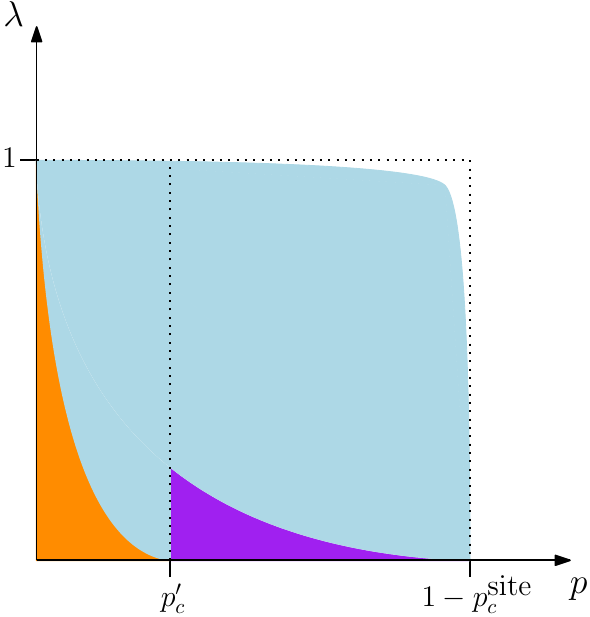}
  \caption{$d\geq3$ (known)}
  \label{fig:phaseknown}
\end{subfigure}
\caption{Possible phase diagrams for the behaviour of FPPHE, as observed via simulations, split into the cases $d=2$ and $d\geq3$, and the known phase diagram for $d\geq3$. In Figure~\ref{fig:phase1a} and Figure~\ref{fig:phase1b}, the orange regions represent where we expect there to be strong survival 
   while the white regions represent where we expect extinction. The purple region in Figure \ref{fig:phase1b} represents the regime of coexistence.
   The blue regions represent the cases where the transition between the phases is not clear even from simulations. In particular, for $d=2$, it is unclear whether there is a regime of coexistence at all, 
   while for $d\geq 3$ it is expected that the purple region extends to the left of $p'_c$, though simulations are not at all conclusive. 
   In Theorem~\ref{maintheorem}, we prove for $d\geq3$ and $p<1-p_c^{\text{site}}$, if $\lambda$ is small enough, then $FPP_1$ survives.
   A finer description of the phases is illustrated in Figure~\ref{fig:phaseknown}; the orange region represents the strong survival phase established in Theorem~\ref{thm:strongsurvival} and \cite{sidoravicius2019multi}, 
   while the purple region represents the coexistence regime given by Corollary~\ref{maincorollary}. 
   The white region is the extinction regimes discussed in Section~\ref{sec:FPPHE} and Section~\ref{sec:known_results}, while the phases in the blue region are not known.} 
\label{fig:phase1}
\end{figure}

It seems natural to believe that, as $p$ or $\lambda$ is increased, 
we observe a monotone transition from a regime of strong survival to a regime of extinction (possibly passing, on the way, through a regime of coexistence).
The phase diagram of FPPHE, as observed via simulations, is illustrated in Figure~\ref{fig:phase1a} and Figure~\ref{fig:phase1b}.
However, as counter-intuitive as it may appear, FPPHE is not a monotone process at all.
To explain this, note that once the locations of the seeds and the values of the passage 
times\footnote{For an edge $(u,v)$ first occupied by $FPP_1$ (resp., $FPP_\lambda$) at $u$, the passage time is the time that $FPP_1$ (resp., $FPP_\lambda$) then takes to attempt to occupy 
$v$ from $u$.} 
have been sampled, 
the evolution of FPPHE becomes deterministic. 
The lack of monotonicity in FPPHE corresponds to the fact that it is possible to choose the locations of the seeds and the values of the passage times such that adding a single seed can be benefitial to $FPP_1$ (instead of 
$FPP_\lambda$). 
This may sound 
counter-intuitive at first, but the idea is that the addition of a seed may slow down $FPP_1$ \emph{locally}, when that seed is activated, 
but this slow down could have the additional effect of delaying the activation of other seeds that are further away along that direction.
It is this delay on activating some seeds that could be helpful for $FPP_1$, aiding its growth along nearby directions. 
By properly choosing the passage times, the addition of a single seed could even 
make $FPP_1$ switch from non-survival to survival. We further note that in a very recent work~\cite{candellero2020}, a stronger notion of non-monotonicity is established;
it is proved that there are graphs for which \emph{the probability that $FPP_1$ survives is not 
a monotone function of $p$ and $\lambda$} (with several phase transitions taking place). 
Even though it is not known whether FPPHE on $\Z^d$ is monotone according to the latter notion, 
the lack of the former notion already creates several problems for the development of a standard multi-scale analysis. 
The approach we develop is able to circumvent this issue, analyzing FPPHE without relying on monotonicity (in particular, we do not use sprinkling ideas).

\subsection{Our results for FPPHE}

Our main result for FPPHE is that if $p<1-p_c^{\text{site}}$ and $\lambda$ is small enough, then with positive probability $FPP_1$ occupies a positive density of sites, immediately giving that $FPP_1$ survives. 
We note that this result is also true in dimension 2. This result is similar in spirit to that of Sidoravicius-Stauffer~\cite{sidoravicius2019multi}, but the latter establishes survival of $FPP_1$ by first fixing $\lambda$ and then 
setting $p$ small enough. The difference in the order at which the parameters are chosen does make a difference in the overall behavior of FPPHE, as illustrated in Figure~\ref{fig:simul} and more thoroughly explained 
in Section~\ref{sec:known_results}.

%Our first result for FPPHE is to establish 
%% The main result of this paper is establishing 
%a coexistence phase for $d\geq3$. We believe that this result is the first example of a random competition model on $\mathbb{Z}^d$ where processes of different speeds 
%coexist\footnote{We remark that Sidoravicius and Stauffer did prove coexistence for a simplified version of FPPHE, where passage times are 
%deterministic~\cite{sidoravicius2019multi}. However, this adaptation renders the model simple enough so that coexistence can be derived from basic percolation results.}. 
%In fact, we prove the stronger result that two processes of different speeds can both occupy a \emph{positive density} of sites with positive probability. 
%This is in stark contrast to other natural random competition models on $\mathbb{Z}^d$, such as the widely studied two-type Richardson model, and highlights the rich behaviour of FPPHE. 
%We discuss results for other models more thoroughly in Section \ref{sec:related_models}. %we recall the behaviour in other random competition models on $\mathbb{Z}^d$.
%
%The main technical result in establishing this coexistence phase is proving that if $p<1-p_c^{\text{site}}$ and $\lambda$ is small enough, then with positive probability $FPP_1$ occupies a positive density of sites, immediately giving that $FPP_1$ survives. 
%We note that this result is true also in dimension 2.

\begin{theorem}
   \label{maintheorem}
   Suppose $d\geq2$ and $p<1-p_c^{\emph{site}}$. There exists a constant $\tilde \lambda=\tilde \lambda(d,p)>0$ such that if $\lambda\in(0,\tilde \lambda)$, then with positive probability $FPP_1$ occupies a positive density of sites.
\end{theorem}
From Theorem~\ref{maintheorem}, coexistence in dimension $d\geq3$ follows almost immediately. It is well known that $p_c^{\text{site}}<1/2$ for $d\geq3$
(cf.\ Campanino and Russo \cite[Theorem 4.1]{campanino1985upper}), and so the inequality $p_c^{\text{site}}<1-p_c^{\text{site}}$ holds. 
In this scenario, if $p\in(p_c^{\text{site}},1-p_c^{\text{site}})$, then there is almost surely an infinite connected component of seeds as well as an infinite connected component of non-seeds. 
An infinite connected component of seeds guarantees survival for $FPP_{\lambda}$ regardless of the value of $\lambda$, while $FPP_{\lambda}$ has a positive density of sites for all $p>0$. 
With this in mind, for $p\in(p_c^{\text{site}},1-p_c^{\text{site}})$, one can use Theorem~\ref{maintheorem} to set $\lambda$ small enough so that $FPP_1$ and $FPP_{\lambda}$ coexist and both occupy a positive density of sites.

In fact, the situation is more subtle. Through an enhancement argument~\cite{aizenman1991strict,balister2014essential}, 
we can deduce a regime of coexistence just below $p_c^{\text{site}}$. 
More precisely, after placing seeds with density $p$, 
we label as \emph{filled seeds} the set of seeds and sites that are disconnected from infinity by seeds. Let $\{o\xrightarrow{f}\infty\}$ be the event that a neighbour of the origin is contained in an infinite connected component of filled seeds. Define the critical probability of percolation for the process given by filled seeds as
$$
p'_c = \sup\left\{p\in[0,1]:\mathbb{P}_{1-p}(o\xrightarrow{f}\infty)=0\right\},
$$
where $\mathbb{P}_{1-p}$ is the probability measure\footnote{The choice of subscript $1-p$ is to remain consistent with later notation, where the emphasis is placed on non-seeds.} induced by the placing seeds with density $p$. Note that such a critical probability exists as the event of filled seeds percolating is monotone in $p$. Clearly $p'_c\leq p_c^{\text{site}}$ and if $p>p'_c$, then $FPP_{\lambda}$ occupies an infinite connected component of sites almost surely. It is known that $p'_c< p_c^{\text{site}}$ for $d=2,3$ by~\cite{aizenman1991strict,balister2014essential} and this strict inequality is conjectured to hold for higher dimensions.

\begin{corollary}
   \label{maincorollary}
   Suppose $d\geq3$, $p\in(p'_c,1-p_c^{\emph{site}})$ and $\tilde \lambda$ is as given in Theorem~\ref{maintheorem}. If $\lambda\in(0,\tilde \lambda)$, then with positive probability $FPP_1$ and $FPP_{\lambda}$ coexist and both occupy a positive density of sites.
\end{corollary}

We believe the above result is the first example of a random competition model on $\mathbb{Z}^d$ where processes of different speeds 
coexist\footnote{We remark that Sidoravicius and Stauffer did prove coexistence for a simplified version of FPPHE, where passage times are 
deterministic~\cite{sidoravicius2019multi}. However, this adaptation renders the model simple enough so that coexistence can be derived from basic percolation results.}. 
In fact, we prove the stronger result that two processes of different speeds can both occupy a \emph{positive density} of sites with positive probability. 
This is in stark contrast to other natural random competition models on $\mathbb{Z}^d$, such as the widely studied two-type Richardson model, and highlights the rich behaviour of FPPHE. 
(We discuss results for other models more thoroughly in Section \ref{sec:related_models}.) %we recall the behaviour in other random competition models on $\mathbb{Z}^d$.
We remark that the coexistence regime we establish is a bit particular, in the sense that $FPP_\lambda$ is guaranteed to survive already at time $0$. 
An interesting open problem is to establish the regime of coexistence for some $p<p'_c$; we discuss further open questions in Section~\ref{sec:conclusion}.

Corollary~\ref{maincorollary} leads one to consider the behaviour of FPPHE when $p<p'_c$ and our second result is a regime of strong survival in this direction. For $M\geq0$, let $\Lambda_M=[-M/2,M/2]^d$ and $\partial\Lambda_M$ denote the set of sites in $\Lambda_M$ that have a neighbour not contained in $\Lambda_M$. Let $\{o\xrightarrow{f}\partial\Lambda_M\}$ be the event that there is a path of filled seeds from a neighbour of the origin to $\partial\Lambda_M$. Define the critical probability
\begin{equation}
\label{p''_c}
p''_c = \sup\left\{p\in[0,1]:\exists a>0 \mbox{ s.t.\ for all large enough }M, \mathbb{P}_{1-p}(o\xrightarrow{f}\partial\Lambda_M)\leq \exp(-M^{a})\right\}.
\end{equation}
It is immediate that $p''_c\leq p'_c$ and we believe there is equality, but proving this is beyond the scope of this article. We prove that if $p<\min\{p''_c,1-p_c^{\text{site}}\}$, then for a sufficiently small choice of $\lambda$, there is strong survival. 
We note that we can extend this result to $d=2$. 
% , but since $p_c^{\text{site}}(2)>1/2$ we need to assume that 
% $p<\min\{p_c^{\text{site}},1-p_c^{\text{site}}\}$. 

\begin{theorem}
	\label{thm:strongsurvival}
Suppose $d\geq2$ and $p<\min\{p''_c,1-p_c^{\emph{site}}\}$. There exists a constant $\tilde \lambda_1=\tilde \lambda_1(d,p)>0$ such that if $\lambda\in(0,\tilde \lambda_1)$, then there is strong survival.
\end{theorem}

We relate our results to known results for FPPHE in Section~\ref{sec:known_results}, after we discuss our novel approach to multi-scale analysis.

\subsection{Multi-scale analysis with non-equilibrium feedback}
\label{sec:MAwNEF}

The development of multi-scale analysis has arisen as a means of understanding the macroscopic behaviour of complex systems and has proven to be a powerful tool across many mathematical fields. 
A standard multi-scale construction usually goes as follows. 
The underlying space is covered by boxes of a fixed size that are labelled as either \emph{good} or \emph{bad} according to (local) events that are measurable with respect to that box. 
One would choose such events so that large clusters of good boxes 
imply some desired behaviour at a mesoscopic or macroscopic scale. 
In some cases this single-scale renormalization procedure will suffice but often one will need to define boxes at infinitely many scales of increasing size 
and define good and bad appropriately at higher scales. 
% with respect to the geometry of bad boxes at lower scales contained within that box. 
The main challenge of multi-scale constructions is finding a sensible way of defining good and bad boxes so that the desired properties of the system can be captured while good boxes occur 
with sufficiently high probability at all scales. If one is able to do this, then this robust framework is able to handle systems with strong correlations that otherwise are not tractable to analysis. 
Examples of such an approach range from random interlacements and dependent percolation~\cite{sznitman2010,peres2013,candellero2018} to interacting particle systems~\cite{kesten2003branching,kesten2005spread,kesten2006,gracar2019}. 
% For example, in a series of papers \cite{kesten2003branching} \cite{kesten2005spread} \cite{kesten2008shape}, Kesten and Sidoravicius use multi-scale analysis to study interacting particle systems that arise from reaction-diffusion systems and are able to prove strong results such as a phase transition in the spread of an infection and a shape limit for a system without using subadditivity. 

A common feature of systems that are amenable to a standard multi-scale analysis as described above is that they are in \emph{equilibrium}, 
thus for example it is possible to assess how good or bad a configuration is at any given moment in time and space without having to observe global information about the system. 
Another important feature is \emph{monotonicity}, which allows for sprinkling arguments to be developed in order to ``decouple'' space-time boxes that are sufficiently far from one another. 
% meaning that all the information about how the process will behave in a box is determined by if the box is good or bad. 
Hence if one can find a suitable definition of good and box boxes, and can control how they are distributed, 
then a standard multi-scale analysis can be potent in understanding the macroscopic structure of the system in question. 

The main contribution of this paper is to develop a framework, which we refer to as \emph{multi-scale analysis with non-equilibirum feedback}, 
that we believe can be used to perform multi-scale analysis on some non-equilibrium systems. 
As with standard multi-scale analysis, our framework is not a ``black box'' that is ready to be applied, but rather an approach (or a strategy) to be followed and tailored to each particular problem.
Here we provide a heuristic and overall explanation on the main ideas behind our framework, and we give a more detailed account later in Section~\ref{sec_new_MSA}, where we develop it to analyze FPPHE.

For concreteness, henceforth we consider the case of FPPHE, but the reasoning below could be adapted to a more general non-equilibrium process. 
We would like to characterise some desired macroscopic behaviour of FPPHE such as survival of $FPP_1$. 
We could perform a standard multi-scale analysis by constructing good and bad boxes at infinitely many scales so that a good box implies that the passage times and the distribution of the seeds inside the 
box is such that if there were only $FPP_1$ outside the box then $FPP_1$ could ``spread well'' throughout the box. 
Such events could be defined in a local way, but the problem lies in the part ``if there were only $FPP_1$ outside the box''. What if it is $FPP_\lambda$ who enters the box first? 
Even worse, what if $FPP_\lambda$ growing from outside the box ends up taking the whole boundary of the box? 

Our approach is to distinguish between two classes of good boxes, which we refer to as being of \emph{positive feedback} or \emph{negative feedback}. For this discussion, we could 
simplistically say that 
a box of positive feedback is a box for which $FPP_1$ arrives to it first, before $FPP_\lambda$ does\footnote{We will need to develop a much more subtle notion of positive feedback, but this simple version is enough to give the 
overall idea behind our approach.}. Note that whether a box is of positive or negative feedback is an event that is \emph{not local}, 
and depends on the whole evolution of the process up to the time it reaches the box. 
Considering also the fact that FPPHE is not a monotone process, 
controlling the probability that a box has positive or negative feedback seems to be rather difficult. 
But the way we proceed is to analyze positive and negative feedback boxes in a completely different manner; in particular, 
we will \emph{not check} whether boxes have positive or negative feedback, 
instead we will use this information and assess its consequences. This motivated the choice of the 
word \emph{feedback}, as we regard this information as a feedback that the box receives from the whole system (that is, it is regarded as an information that is given to us, 
instead of an information whose validity we check or whose probability we estimate).

The overall idea is to tune the definition of positive and negative feedback so that the following two properties are obtained: 
\begingroup
\renewcommand{\theenumi}{\roman{enumi}}
\renewcommand{\labelenumi}{\theenumi.}
\begin{enumerate}
   \item\label{it:ppos} a positive feedback box implies the well spread of $FPP_1$ inside that box (that is, it implies the desired macroscopic behavior of the process),
   \item\label{it:pneg} a box that has negative feedback can be associated to a bad box (called its progenitor) in a way that we can control how far away the progenitor of a negative feedback box is.
\end{enumerate}
\endgroup
% (i) a positive feedback box implies the well spread of $FPP_1$ inside that box (that is, it implies the desired macroscopic behavior of the process), 
% and (ii) a box that has negative feedback can be associated to a bad box (called its progenitor) in a way that we can control how far away the progenitor of a negative feedback box is. 
Thus since bad boxes (via a standard multi-scale analysis) can be shown to be rare,~\eqref{it:pneg} allows us to control the influence and location of boxes with negative feedback. 
Note that~\eqref{it:pneg} allows us to show that negative feedback boxes are rare by controlling bad boxes only, without having to estimate the probability that a box is of negative feedback. 

To give a better idea regarding~\eqref{it:pneg}, the progenitor of a negative feedback box $Q$ is the box that \emph{is responsible} for $Q$ ending up having a negative feedback; in the case of FPPHE, one could imagine that 
the progenitor is the bad box from which a seed of $FPP_\lambda$ is activated and grew all the way until entering $Q$, doing so before $FPP_1$.
We could generalize~\eqref{it:pneg} so that a negative feedback box is associated to more than one progenitors, as long as we can control the region inside which such progenitors may be found, but in the present paper it is 
enough to have a single progenitor for each negative feedback box.
We also remark that bad boxes do not need to be classified as a positive feedback or negative feedback. So whenever we refer to a positive or negative
feedback box, we assume that the box is good. % without losing any generality.

Our approach is then naturally split into two parts. 
The first part consists of a standard multi-scale analysis where we control how the process will behave inside a box given that the behavior of the process outside the box is nice enough. 
Then, in the second part we define positive and negative feedback, and derive properties~\eqref{it:ppos} and~\eqref{it:pneg} above. 
We divided the bulk of our proof for FPPHE in two sections exactly to better highlight the role of each of these parts, 
with the first part (the standard multi-scale analysis) being developed in Section~\ref{sec_standard_MSA}
and the second part (the actual multi-scale analysis with non-equilibrium feedback) being developed in Section~\ref{sec_new_MSA}.

We believe that the idea behind the multi-scale analysis with positive feedback is robust enough to be applied to other systems with non-equilibrium dynamics. 
In particular, systems that could be suitable to such type of approach are the ones which 
either contains a stationary component (like the passage times of FPPHE, which are i.i.d., 
or an additional example are systems with a suitable graphical representation involving i.i.d.\ random variables) 
or can be approximated by a stationary process. Such a feature could be used in a two-part proof, as described above, 
with a standard multi-scale analysis being constructed on the 
stationary component of the process and a positive/negative feedback mechanism being developed to leverage the multi-scale analysis to control the aspects of the process that is 
not in equilibrium.

\subsection{Known results for FPPHE}
\label{sec:known_results}

%In \cite{sidoravicius2019multi} FPPHE was introduced as a means for analysing a notoriously challenging random aggregation model called \emph{multiparticle diffusion limited aggregation} (MDLA) but emerged as an interesting model to analyse in its own right. 
%In fact, most of the work in \cite{sidoravicius2019multi} is devoted to studying FPPHE, from which results about MDLA can be deduced. 
%In this regard, $FPP_1$ can be viewed as a randomly growing aggregate and $FPP_{\lambda}$ represents regions of space where there is a lack of particles to allow for aggregation. With this interpretation in mind and a delicate coupling argument, positive speed for MDLA can be deduced from strong survival of FPPHE.

The first result for FPPHE was proving the existence of a strong survival phase for $d\geq2$ by Sidoravicius and Stauffer \cite[Theorem 1.3]{sidoravicius2019multi}. 
More precisely, for $d\geq2$ they proved that for any $\lambda\in(0,1)$, there exists $p_0=p_0(\lambda,d)>0$ such that if $p\in(0,p_0)$, then there is strong survival. 
The difference between Theorem~\ref{thm:strongsurvival} and the strong survival phase proven in \cite{sidoravicius2019multi} is that we fix $p$ first, and then take $\lambda$ small enough, 
whereas in \cite{sidoravicius2019multi} it is $\lambda$ that is fixed first and $p$ that is taken small enough. As noticed in Figure~\ref{fig:simul}, such a change makes FPPHE behave rather differently.

The proof in \cite{sidoravicius2019multi} relied on a multi-scale argument that handled the long-range dependencies in FPPHE in the following manner. 
One can restrict the spread of $FPP_{\lambda}$ from a seed to a (possibly arbitrarily large) region of $\mathbb{Z}^d$ by showing that $FPP_1$ is able to grow around (encapsulating) $FPP_\lambda$ inside this region. 
Outside of this ``bad region'', $FPP_{\lambda}$ originating from this seed no longer affects how FPPHE evolves. 
If such bad region does not intersect bad regions from other seeds (a fact that ends up being obtained by setting $p$ small enough), 
then a multi-scale argument that accounts for all bad regions implies that 
the sites that are not contained in any bad region percolate in $\Z^d$. This implies the strong survival phase. 

In the coexistence phase, however, $FPP_1$ cannot encapsulate $FPP_{\lambda}$, 
and in particular, there is a seed that gives rise to an infinite component of $FPP_{\lambda}$ that $FPP_1$ must survive against. 
Hence we had to find a different way to control how much effect a seed may have in the growth of $FPP_\lambda$.
% As we mentioned above, 
% to handle these strong correlations we introduce a novel approach to multi-scale technique called \emph{multi-scale analysis with non-equilibrium feedback}, 
% whose idea we believe could have a wider applicability. 
% We regard this as one of the main contributions of this work and discuss it in more detail in Section~\ref{sec:MAwNEF}.
Drawing a parallel between the two works, we notice that~\cite{sidoravicius2019multi} develops a multi-scale analysis 
that is particularly tailored to 
the analysis of that specific question and also crucially relies on the fact that outside of bad regions one only finds $FPP_1$. 
This last property allowed for an essentially standard multi-scale analysis to be developed, as local events alone were enough to guarantee the encapsulation of 
$FPP_\lambda$. In particular, the technique in~\cite{sidoravicius2019multi} is ineffective to handle the more 
challenging coexistence phase. 
On the other hand, not only our proof can handle the coexistence regime, but it also consists on the development of a novel approach to multi-scale analysis, 
which gives a mechanism that we believe 
is suitable to be adapted to other questions and models. 
In particular, in Section~\ref{sec:conclusion}, we briefly explain how our approach can be tuned to give a cleaner proof of the result in~\cite{sidoravicius2019multi}.

We end this section with further remarks about FPPHE.
It is intuitive that if $\lambda\geq1$, then $FPP_1$ dies out almost surely. 
In particular, adapting a result by van den Berg and Kesten \cite{van1993inequalities} regarding strict inequalities for first passage percolation (in particular, the extension to the limit shape given in 
\cite[Proposition 6.4]{sidoravicius2019multi}), 
we obtain that there exists an $\varepsilon>0$ which does not depend on $\lambda$ (just on $p$ and $d$) 
so that $FPP_1$ spreads in a way that is strictly slower than first passage percolation of rate $1-\varepsilon$. 
This is simply because $FPP_1$ must at least deviate from a density of seeds. Then if $\lambda >1-\varepsilon$, some seed of $FPP_{\lambda}$ will manage to grow fast enough to eventually encapsulate $FPP_1$ inside a finite region. 
We refer to \cite[Theorem 4.1]{finnthesis} for details.
The known behaviour for FPPHE on $\mathbb{Z}^d$ for $d\geq3$ can be seen in Figure~\ref{fig:phaseknown}.

FPPHE can be readily generalised to any infinite connected graph. %and for the remainder of this section we highlight some known results in the more general setting. 
For example, if $G$ is a rooted tree and $p<1-p_c^{\text{site}}(G)$, then coexistence is equivalent to the root being contained in an infinite cluster of non-seed sites, 
where the root has the same role as the origin on $\mathbb{Z}^d$. 
Such cases are considered trivial as coexistence is an artefact of a more classically understood process. 
The first non-trivial coexistence phase was established in Candellero and Stauffer \cite[Corollary 1.3]{candellero2018coexistence} when $G$ is a vertex-transitive, hyperbolic and non-amenable graph. 
First in \cite[Theorem 1.1]{candellero2018coexistence}, they prove that for \emph{any} $\lambda>0$, there exists a $p_0=p_0(G,\lambda)$ such that if $p\in(0,p_0)$, then the $FPP_1$ process survives with positive probability. 
Though this is analogous to the result in \cite{sidoravicius2019multi} for $\mathbb{Z}^d$, there is a difference in behaviour as in hyperbolic and non-amenable graphs $FPP_1$ can survive even if $\lambda\geq1$. 
Secondly in \cite[Theorem 1.2]{candellero2018coexistence}, they prove that for \emph{any} $p\in(0,1)$, and \emph{any} $\lambda>0$, the $FPP_{\lambda}$ process survives almost surely. 
In other words, the strong survival phase does not occur for any choice of parameters. 
We remark that a result analogous to our theorem on strong survival (Theorem~\ref{thm:strongsurvival}), where $p$ is fixed first and then $\lambda$ is set small enough, 
has not been established in~\cite{candellero2018coexistence} for hyperbolic, nonamenable
graphs. The main issue being that such a case requires a more delicate analysis which would involve a detailed control of the geometry of non-seeds.

\subsection{Related growth models and the search for coexistence}
\label{sec:related_models}

The study of random growth processes with competition is a classical area of research with particular interest in the circumstances in which both processes can coexist with one another, 
especially on $\mathbb{Z}^d$. 
A notable example is the \emph{two-type Richardson model}, in which two types (which we call $FPP_1$ and $FPP_\lambda$) grow throughout $\mathbb{Z}^d$ as follows: $FPP_1$ starts from the origin and spreads throughtout the 
edges of $\mathbb{Z}^d$ at rate $1$ (as a first passage percolation process), whereas $FPP_\lambda$ starts from a neighbour of the origin and spreads in the same way as $FPP_1$, but at a rate $\lambda$. 
Whenever a site is occupied by one of the processes, it remains so forever. 
Unlike the model FPPHE we study in this paper, in the two-type Richardson model both $FPP_1$ and $FPP_\lambda$ start spreading already from time $0$. 

It is conjectured that when $d\geq2$ coexistence in the two-type Richardson model can occur with positive probability if and only if $\lambda=1$ (that is, both processes spread at the same rate). 
Coexistence for $\lambda=1$ was established on $\mathbb{Z}^2$ by 
H\"aggstr\"om and Pemantle \cite{haggstrom1998first} before being generalised to higher dimensions by Garet and Marchand \cite{garet2005coexistence} and Hoffman \cite{hoffman2005coexistence}. 
The other side of the conjecture remains open, but H\"aggstr\"om and Pemantle \cite{haggstrom2000absence} proved that 
the set of values for $\lambda$ that coexistence can occur with positive probability is at most a countable subset of $(0,\infty)$. 
At first, it may sound incredible that this does not imply that coexistence can only occur for $\lambda=1$, but the issue is that the event of coexistence is not monotone in $\lambda$. 
Nonetheless, the two-type Richardson model is a monotone process in the sense that there is a straightforward coupling under which if $\lambda$ is increased, then the set of sites occupied by $FPP_{\lambda}$ can only increase. 
But the event of coexistence is more subtle since increasing (resp., decreasing) $\lambda$ could cause the transition from a regime of coexistence to a regime where $FPP_\lambda$ (resp., $FPP_1$) is the only one to survive. 
Recently, the conjecture for the two-type Richardson model was solved for the half-plane by Ahlberg, Deijfen and Hoffman \cite{ahlberg2018two}, exploiting the planarity of $\mathbb{Z}^2$ and the boundary of the half-plane.
We remark that, as we explained previously, FPPHE is \emph{not} monotone at all, which 
% Not only is it possible to engineer configurations of seeds and passage times so that by adding a seed we help $FPP_1$ survive, 
% it has been shown that, at least for some choices of graphs, increasing $p$ or $\lambda$ may increase the probability that $FPP_1$ survives. 
% This 
highlights the challenges that the intrinsic strong dependences in FPPHE can give rise to. 

Garet and Marchand \cite{garet2008first} proved that in the two-type Richardson model with $\lambda\neq1$, both processes cannot simultaneously occupy a positive density of sites almost surely. 
In fact, in \cite{garet2008first} they extend this result to a broad class of passage times beyond the context of the two-type Richardson model. This is a stark contrast to the behaviour of FPPHE, 
as in Corollary~\ref{maincorollary} we have that FPPHE can have both processes occupying a positive density of sites and coexisting with positive probability.

% Instead of competition between the occupancy of sites on a graph, one could consider models where occupancy of interacting particles is what drives the competition. 
In a somewhat different context, Deijfen, Hirscher and Lopes \cite{deijfen2019competing} studied a two-type variant of a particle system known as the \emph{frog model}. 
In this model, we start two \emph{active} particles from the origin of $\mathbb{Z}^d$, one particle of type $1$ and the other of type $\lambda$. In addition, at each site that is not 
the origin, we place an independent random number of \emph{sleeping} particles. Then, from time 0 active particles move as independent simple random walks. with type $1$ particles jumping at rate $1$ and type $\lambda$ particles 
jumping at rate $\lambda$. 
Sleeping particles, in turn, do not jump. However, they become active whenever a particle jumps on their site, and at that time acquire the type of the particle that jumped onto them. 
A central question is to understand which type manages to conquer an infinite number of particles, and in particular whether coexistence is possible (that is, whether both types can simultaneously conquer an infinite number 
of particles).

In \cite{deijfen2019competing}, it is shown (for a discrete-time variant of the aforementioned process) that coexistence is possible if the two types have equal jump rates (i.e., $\lambda=1$). 
They conjecture that when the jump probabilities are different, coexistence is possible only if the initial distribution of the sleeping particles is especially heavy-tailed.
We note that the two-type frog model is also a monotone process.
Many other one-type interacting particle systems could also be generalized to two-type competition versions, and references for these can be found in \cite{deijfen2019competing}.

\subsection{Outline of paper}

In Section \ref{sec_preliminaries} we provide a rigorous construction of FPPHE and fix some useful notation in regards to first passage percolation. 
In Section \ref{sec_standard_MSA} we perform a standard multi-scale analysis on FPPHE, as we described in Section~\ref{sec:MAwNEF}.
% This standard multi-scale analysis will introduce good and bad boxes at all scales and proves that good boxes occur with sufficiently high probability at all scales. 
In Section \ref{sec_new_MSA} we introduce multi-scale analysis with non-equilibrium feedback and derive all the desired properties.
In Section \ref{sec:Proof_Thm1} we establish Theorem~\ref{maintheorem} and Theorem~\ref{thm:strongsurvival} using the results of the previous sections. In Section~\ref{sec:conclusion} we conclude with some open questions and final remarks.

\section{Preliminaries}
\label{sec_preliminaries}

In this section we give a formal construction of FPPHE on $\mathbb{Z}^d$. Initially we place seeds according to i.i.d.\ Bernoulli random variables of parameter $p\in(0,1)$ on all sites except the origin. Fix $\lambda>0$. To define the passage times for $FPP_1$ and $FPP_{\lambda}$ we construct two sets of passage times $\{t_e\}_{e\in E}$ and $\{t^{\lambda}_{e}\}_{e\in E}$, where $E$ is the edge set of $\mathbb{Z}^d$, such that $\{t_e\}_{e\in E}$ (resp.\ $\{t^{\lambda}_e\}_{e\in E}$) is a collection of i.i.d.\ exponentially distributed random variables of rate 1 (resp.\ $\lambda$).

At time $t=0$ all seeds are inactive and $FPP_1$ only occupies the origin. The dynamics for $t\geq0$ are as follows. If $x$ is first occupied by the $FPP_1$ process at time $t$ and $y$ is a neighbour of $x$ connected by an edge $e_{x,y}$, then $FPP_1$ attempts to occupy $y$ at time $t+t_{e_{x,y}}$. The occupation is successful if $y$ is not occupied by $FPP_1$ or $FPP_{\lambda}$ before time $t+t_{e_{x,y}}$ and $y$ is not a seed. If $y$ is occupied by $FPP_1$ or $FPP_{\lambda}$ by time $t+t_{e_{x,y}}$, then the occupation from $x$ is unsuccessful. If $y$ is an inactive seed for all time before time $t+t_{e_{x,y}}$, then $y$ becomes an active seed at time $t+t_{e_{x,y}}$, meaning that it is occupied by the $FPP_{\lambda}$ process at time $t+t_{e_{x,y}}$. The dynamics if $x$ was instead occupied by the $FPP_{\lambda}$ process are precisely the same except with passage times given by $\{t_e^{\lambda}\}_{e\in E}$.

\begin{remark}
In our analysis we never appeal to the memoryless property of exponentials and so more general passage times with reasonable conditions could be considered instead. For example, distributions with exponential tails would satisfy our analysis. We choose to restrict our attention to exponential passage times for ease of exposition and to avoid cumbersome notation.
\end{remark}

To remove possible ambiguities when speaking about the occupancy of sites by $FPP_1$ and $FPP_{\lambda}$, we introduce the following notation. We let $\eta_t$ be the \emph{configuration} for FPPHE on $\mathbb{Z}^d$ at time $t$, meaning that $\eta_t:\mathbb{Z}^d\to\{-1,0,1,2\}$ where $\{-1,0,1,2\}$ is a set of labels giving the state of occupancy of a site in the following manner. For a fixed site $x\in\mathbb{Z}^d$ and time $t\geq0$, we have:
\begin{itemize}
\item if $\eta_t(x)=-1$, then $x$ is a non-activated seed at time $t$,
\item if $\eta_t(x)=0$, then $x$ is not a seed and is not occupied by $FPP_1$ or $FPP_{\lambda}$ at time $t$,
\item if $\eta_t(x)=1$, then $x$ is occupied by $FPP_1$ at time $t$,
\item if $\eta_t(x)=2$, then $x$ is either an activated seed or a non-seed occupied by $FPP_{\lambda}$ at time $t$.
\end{itemize}

Using the above notation $\eta_0$ is defined as follows. We set $\eta_0(0)=1$ and for $x\in Z^d\setminus\{0\}$, with probability $p$ we have $\eta_0(x)=-1$, otherwise $\eta_0(x)=0$.

From the set of passage times $\{t_e\}_{e\in E}$ and $\{t^{\lambda}_e\}_{e\in E}$ given in the construction of FPPHE, it will be useful to recall some standard definitions in first passage percolation. The interested reader can find an introduction to first passage percolation in the monograph \cite{auffinger201750} by Auffinger, Damron and Hanson.

Let $\gamma=(v_0,v_1,\ldots,v_{n-1},v_n)$ be a finite simple path of sites on $\mathbb{Z}^d$, meaning that $v_i\neq v_j$ if $i\neq j$ and $\|v_i-v_{i-1}\|_1=1$ for all $i\in\{1,2,\ldots,n\}$. We define the \emph{passage time of} $\gamma$ \emph{with respect to} $\{t_e\}_{e\in E}$ as the random variable
\begin{align*}
T_1(\gamma) = \sum_{i=1}^nt_{e_i}
\end{align*}
where $e_i$ is the edge between $v_{i-1}$ and $v_i$. Similarly, we define the \emph{passage time of} $\gamma$ \emph{with respect to} $\{t^{\lambda}_e\}_{e\in E}$ as the random variable
\begin{align*}
T_{\lambda}(\gamma) = \sum_{i=1}^nt^{\lambda}_{e_i}.
\end{align*}
For $x,y\in\mathbb{Z}^d$, let $\Gamma(x\to y)$ be the set of all finite simple paths from $x$ to $y$. We define the \emph{passage time between }$x$ \emph{and} $y$ as the random variable
\begin{align*}
    T_{\#}\left(x\to y\right)=\inf_{\gamma\in\Gamma(x\to y)}T_{\#}\left(\gamma\right),
\end{align*}
where the label $\#\in\{1,\lambda\}$ indicates which set of passage times are being referred to. If $U$ is a subgraph of $\mathbb{Z}^d$, we let $T_{\#}(x\to y; U)$ be the infimum of all passage times over paths that do not exit $U$, so that 
\begin{align*}
    T_{\#}\left(x\to y; U\right)=\inf_{\gamma\in\Gamma(x\to y; U)}T_{\#}\left(\gamma\right),
\end{align*}
where $\Gamma(x\to y; U)$ is the set of all finite simple paths from $x$ to $y$ that do not exit $U$. 

\section{Standard multi-scale analysis}
\label{sec_standard_MSA}

In order to better explain the non-equilibrium feedback idea we introduce, we split our multi-scale analysis in the next two sections. In this section, we describe the part of our analysis that follows a more standard (``in equilibrium'') multi-scale analysis approach, which will be based only on \emph{local} events on the sets of passage times $\{t_e\}_{e\in E}$ and $\{t^{\lambda}_e\}_{e\in E}$. 

Let $L_1$ be some large integer to be set later. For $k\geq2$, set
\begin{equation}
\label{eq:L_k_def}
L_k=k^2L_{k-1}^d.
\end{equation}
Each $L_k$ gives us the side lengths of the $k$-boxes we define below. We refer to $k$ as the \emph{scale}.

At each scale $k\geq1$ we partition $\mathbb{Z}^d$ into boxes of side-length $L_k/3$, producing a collection of disjoint boxes that we call $k$\emph{-cores}:
$$
\left\{Q_k^{\text{core}}\left(i\right)\right\}_{i\in\mathbb{Z}^d} \text{ with } Q_k^{\text{core}}\left(i\right) = \left(L_k/3\right)i + \left[-L_k/6,L_k/6\right)^d.
$$
For each $k\geq1$, define $k$\emph{-boxes} as a collection of overlapping boxes of length $L_k$ which are centred on the $k$-cores by
$$
    \left\{Q_k\left(i\right)\right\}_{i\in\mathbb{Z}^d} \text{ with } Q_k\left(i\right) = \left(L_k/3\right)i + \left[-L_k/2,L_k/2\right]^d. 
$$

\subsection{Good boxes at scale 1}
\label{sec:good_1_boxes}

The first step in this multi-scale analysis is defining local events to $1$-boxes that will distinguish between \emph{good} and \emph{bad} $1$-boxes. Roughly speaking, we would like a good $1$-box to imply that $FPP_1$ has the opportunity to spread with fast passage times throughout the $1$-box while the spread of $FPP_{\lambda}$ is inhibited. Moreover, we would like to prove that good $1$-boxes occur with high probability for a suitably large choice of $L_1$ and small choice of $\lambda$. In this section we make a rigorous idea of these notions.

One condition for a $1$-box to be good is that the non-seed sites \emph{percolate} well inside the $1$-box, giving $FPP_1$ the opportunity to spread to nearby $1$-boxes. To define this condition rigorously, let $\mathcal{G}$ be the random sub-graph of $\mathbb{Z}^d$ induced by removing all $FPP_{\lambda}$ seeds from $\mathbb{Z}^d$, noting that we drop the dependence on $p$ for ease of notation. Let $d_{\mathcal{G}}(\cdot,\cdot)$ be the natural graph metric induced by $\mathcal{G}$ and $V(\mathcal{G})$ be the vertex set of $\mathcal{G}$. Fix $i\in\mathbb{Z}^d$ and consider the 1-box $Q_1(i)$. Let $\partial Q_1(i)$ be the set of sites of $Q_1(i)$ that have a neighbour not contained in $Q_1(i)$. Enumerate the clusters of $V(\mathcal{G})\cap Q_1(i)$ where we remove edges between sites in $\partial Q_1(i)$ as
$$
\mathcal{C}_1(i),\mathcal{C}_2(i),\ldots,\mathcal{C}_{n_i}(i),
$$
where $n_i$ is the number of clusters and $|\mathcal{C}_1(i)|\geq|\mathcal{C}_2(i)|\geq\ldots\geq|\mathcal{C}_{n_i}(i)|$. Note that this ordering may not be unique, but this fact will not alter our arguments. The removal of edges between sites in $\partial Q_1(i)$ is a non-restrictive condition that will simplify arguments later. 
We let $\mathcal{C}^{-}_1(i)=\mathcal{C}_1(i)\setminus \partial Q_1(i)$ denote the cluster $\mathcal{C}_1(i)$ with sites in $\partial Q_1(i)$ removed.

Recall we assume $p<1-p_c^{\text{site}}$ so that the non-seed sites are supercritical. With this assumption in mind, we have
$$
\theta(1-p)=\mathbb{P}_{1-p}(o\text{ is contained in an infinite cluster of }\mathcal{G})>0,
$$
where $\mathbb{P}_{1-p}$ is the measure given by placements of seeds with parameter $p$ in the construction of FPPHE and we write $o$ for the origin of $\mathbb{Z}^d$. We write $1-p$ instead of $p$ to emphasise that it is the non-seed sites that we are considering and to remain consistent with standard notation in percolation theory. 

For a fixed $\varepsilon>0$, define the event
\begin{align}
E_1(i) = \left\{\begin{array}{r}
\mbox{for all }i'\mbox{ with }Q_1^{\text{core}}(i')\subset Q_1(i)\text{ we have that }\mathcal{C}_1(i)\cap Q_1^{\text{core}}(i')\text{ touches}\\
\text{all faces of }Q_1^{\text{core}}(i') \text{ and }
|\mathcal{C}_1(i)\cap Q_1^{\text{core}}(i')|\geqslant (1-\varepsilon)\theta(1-p)(L_1/3)^d\\
\end{array} \right\} \label{def:E_1}
\end{align}
where the dependence on $\varepsilon$ is left as implicit for ease of notation. Roughly speaking, the event $E_1(i)$ guarantees that within each $1$-core contained in $Q_1(i)$, the large component of non-seeds in $Q_1(i)$ percolates well within each $1$-core. Deuschel and Pisztora \cite[Theorem 1.1]{deuschel1996surface} proved large deviation estimates for such events, and in particular, there exists a constant $b_1>0$ such that
$$
\mathbb{P}(E_1(i))>1-\exp\left(-b_1L_1^{d-1}\right).
$$
It could be the case that the large components of non-seeds within each core are not contained in the same cluster within a $1$-box. To handle this, we consider the event
$$
E_2(i) = \left\{|\mathcal{C}_2(i)|\leq \log^2 |Q_1(i)|\right\},
$$
which is the event that components of non-seeds that are not in the largest component in $Q_1(i)$ are logarithmically small. Penrose and Pisztora \cite[Theorem 5]{penrose1996large} provide large deviation estimates for such events and give us that there exists a constant $b_2>0$ such that
$$ 
\mathbb{P}(E_2(i))>1-\exp\left(-b_2\log^2 L_1\right).
$$
If $Q_1(i)$ and $Q_1(j)$ are neighbours in the sense that $i$ and $j$ are neighbours on $\mathbb{Z}^d$, there is an overlap of side length at least $L_k/3$. If $E_1(i), E_1(j), E_2(i)$ and $E_2(j)$ all hold, then for large enough $L_1$, the intersection of the large components $\mathcal{C}_1(i)\cap\mathcal{C}_1(j)$ of the respective $1$-boxes must be non-empty, as the smaller clusters are only logarithmically small.

The event $E_1(i)$ only tells us that the non-seeds percolate well in a $1$-box but it could be the case that the paths of non-seeds within the box are superlinear in the graph distance of the sites which would greatly inhibit the spread of $FPP_1$. To remedy this, we recall the notion of the \emph{chemical distance} between two sites on a percolated graph, which is their graph distance induced by the graph under percolation. Let $\mathcal{G}(i)$ be $\mathcal{G}$ restricted to $Q_1(i)$ with edges between sites within $\partial Q_1(i)$ not included. Set $d_{\mathcal{G}(i)}(\cdot,\cdot)$ to be the natural graph distance metric on $\mathcal{G}(i)$.

For a constant $c_1\geq1$, define the event
$$
E_3(i) = \left\{\forall x,y\in\mathcal{C}_1(i), d_{\mathcal{G}(i)}(x,y)< c_1\max\left\{\|x-y\|_1,\log^2 L_1\right\}\right\},
$$
that states the chemical distance between two sites in the large component of non-seed sites within $Q_1(i)$ is linear with a logarithmic term. For $E_3(i)$ we can appeal to large deviation estimates for the chemical distance for supercritical Bernoulli percolation by Antal and Pisztora \cite[Corollary 1.3]{antal1996chemical}. So long as $c_1$ is large enough with respect to $p$ and $d$, as given in \cite[Theorem 1.1]{antal1996chemical}, we may deduce that for all large enough choices of $L_1$ that
$$
\mathbb{P}\left(E_3(i)\right)>1-\exp\left(-c_1\log^2 L_1\right). 
$$
	We observe two technical points to justify the above bound.
	First, the main results in Antal and Pisztora \cite{antal1996chemical} are stated for \emph{bond} percolation, but also work for site percolation; see the remark following \cite[Theorem 1.2]{antal1996chemical}.
	Second, the event $E_3(i)$ concerns paths restricted to the $1$-box $Q_1(i)$, while the results in Antal and Pisztora \cite{antal1996chemical} consider unrestricted paths.
	However, the same proof in the truncated case follows by performing the coarse graining scheme of Antal and Pisztora, say with sub-boxes of side-length $N$ partitioning $Q_1(i)$, and then setting $c_1$ and $L_1$ large enough with respect to $N$.
	More precisely, each sub-box is classified as either good or bad according to the presence of a unique large component of non-seeds within the sub-box (similar to our events $E_1$ and $E_2$).
	As the probability that a sub-box is good can be made large, we are in the domain of highly supercritical percolation.
	Thus to construct a path from $x$ to $y$ that are both in the large component of $Q_1(i)$, it suffices to fix a deterministic path of sub-boxes from the sub-box containing $x$ to the sub-box containing $y$, and then go around bad clusters of sub-boxes using an argument of Fontes and Newman \cite{fontes1993first} as done in \cite{antal1996chemical}.

For a large constant $c_2>1$, define the event 
$$
E_4(i) = \left\{\forall x,y\in\mathcal{C}_1(i), T_1(x\to y; \mathcal{G}(i))<c_2\max\left\{\|x-y\|_1,\log^2 L_1\right\}\right\},
$$
that states the passage times between two non-seed sites in the large component of $Q_1(i)$ are linear in their graph distance. By considering Chernoff bounds for sum of exponentials, then one can deduce that $E_4(i)$ holds with high probability for a large enough choice of $c_2$ through similar arguments made for $E_3(i)$. 

Let $\mathcal{E}(i)$ be the set of edges of $\mathbb{Z}^d$ where both endpoints are contained in $Q_1(i)$ and $E_5(i)$ be the event that all edges in  $\mathcal{E}(i)$ have passage time at least $1/\sqrt{\lambda}$ with respect to $\{t_e^{\lambda}\}_{e\in E}$, so that
$$
E_5(i) = \left\{\forall e\in \mathcal{E}(i), t^{\lambda}_{e}\geq\tfrac{1}{\sqrt{\lambda}}\right\}.
$$
Recall that if $X$ is an exponentially distributed random variable of rate $\mu$, then $ \mathbb{P}(X<x)=1-e^{-x \mu}$ for any $x\geq0$. Hence by taking the union bound over each edge $e\in \mathcal{E}(i)$, we have
\begin{equation}
\label{eq:E_5_prob}
\mathbb{P}\left(E_5(i)^\text{c}\right)\leqslant \left(1-e^{-\sqrt{\lambda}}\right)|\mathcal{E}(i)| \leqslant \sqrt{\lambda}|\mathcal{E}(i)| 
\end{equation}
where $|\cdot|$ denotes cardinality and we use the inequality $e^{-x}\geq 1-x$ for all $x\in\mathbb{R}$. The right-hand side of \eqref{eq:E_5_prob} may be made arbitrarily small by setting $\lambda$ small enough.

%If $E_1(i)$ through $E_5(i)$ all hold, then we define $Q_1(i)$ to be \emph{good}, as outlined in the following definition.

\begin{definition}
(Good box at scale 1) Let $i\in\mathbb{Z}^d$ and fix any $\varepsilon>0$ and large enough constants $c_1, c_2>1$. Define the event
$$
G_1(i)= \bigcap_{j=1}^5 E_j(i).
$$
If $G_1(i)$ holds, then we define $Q_1(i)$ to be \emph{good}. Otherwise we define $Q_1(i)$ to be \emph{bad}.
\end{definition}

	Note that a 1-box is measurable with respect to events concerning the sites and edges inside the 1-box,  and thus disjoint 1-boxes being good are independent events.
	The probability of a $1$-box being good is invariant under translations and so we may define $\rho_1$ to be the probability that an arbitrary $1$-box is bad by setting
$$
\rho_1 = 1 - \mathbb{P}(G_1(o)).
$$

\begin{lemma}
\label{lem:prob_good_1_box}
For any $\varepsilon>0$, suitably large choices of $c_1,c_2>1$ and $B>0$, by setting $L_1$ large enough and then $\lambda$ small enough, we have
$$
\rho_1 <  L_1^{-B}.
$$
\end{lemma}

\begin{proof}
The result is an immediate consequence of the bounds derived for the probabilities of the events $E_j(i)$ above.
\end{proof}

\subsection{Good boxes at higher scales}
\label{sec:good_k_boxes}

In this section we generalise the notion of good boxes to higher scales. The idea is that for $k\geq2$, a $k$-box is good if it contains no more that $A$ disjoint bad $(k-1)$-boxes, where $A$ is a fixed constant that we set later.

\begin{definition}
(Good box at scale $k\geq2$) Fix a constant $A>0$. Let $k\geq2$ and consider the $k$-box $Q_k(i)$. Define $G_k(i)$ to be the event that $Q_k(i)$ contains no more than $A$ disjoint bad $(k-1)$-boxes. If $G_k(i)$ holds, then we define $Q_k(i)$ to be \emph{good}. Otherwise we define $Q_k(i)$ to be \emph{bad}.
\end{definition}

	From the definition of good $k$-box for $k\geq2$, we deduce that disjoint $k$-boxes being good are independent events.
	Later in Lemma~\ref{lem:prob_good_k_box} we will see that choosing $A>d$ will suffice so long as $L_1$ is large enough. The intuition behind this definition is that knowing $Q_k(i)$ is good means that there is at most $A$ disjoint bad $(k-1)$-boxes that will simplify how we control the spread of $FPP_{\lambda}$ within $Q_k(i)$. There may be many other bad boxes at lower scales, but our multi-scale framework will allow us to consider bad boxes on a scale-by-scale basis. With at most $A$ disjoint bad $(k-1)$-boxes in a good $k$-box, the total number of bad $(k-1)$-boxes in said $k$-box is bounded above by a constant that only depends on $A$ and $d$. 

Fix $k\geqslant1$. Similar to the scale 1 case, the probability of a $k$-box being good is invariant under translations, so we can define $\rho_k$ to be the probability that an arbitrary $k$-box is bad by setting
$$
    \rho_k = 1-\mathbb{P}\left(G_k(o)\right).
$$

\begin{lemma}
\label{lem:prob_good_k_box}
Let $A>d$. For a sufficiently large choice of $L_1$ and small enough choice of $\lambda$, then for all $k\geq1$, we have
$$
    \rho_k \leqslant \rho_1^{A^{k-1}}.
$$
\end{lemma}

\begin{proof}
We proceed by induction, noting that the case $k=1$ is immediate. Now let $k>1$ and assume Lemma~\ref{lem:prob_good_k_box} holds up to scale $k-1$. We can first bound $\rho_k$ by noticing that
\begin{align*}
    \rho_k \leqslant \mathbb{P}(\Pi_{A+1}),
\end{align*}
where $\Pi_{A+1}$ is the event that there are $A+1$ disjoint bad $(k-1)$-boxes in the $k$-box. We can find an upper bound on the number of $(k-1)$-boxes in a $k$-box by counting the number of $(k-1)$-cores contained in a $k$-box, which equals
\begin{align*}
    \left(\tfrac{L_k}{L_{k-1}/3}\right)^d = 3^d k^{2d}L_{k-1}^{d(d-1)}.
\end{align*}
Hence the number of possible ways to choose $(A+1)$ disjoint bad $(k-1)$-boxes is bounded above by $\left(3^d k^{2d}L_{k-1}^{d(d-1)}\right)^{A+1}$ and by taking a union bound over these configurations we see that
\begin{align}
    \rho_k&\leqslant\left(3^d k^{2d}L_{k-1}^{d(d-1)}\rho_{k-1}\right)^{A+1},\nonumber\\
    &\leqslant \left(3^d k^{2d}L_{k-1}^{d(d-1)}\right)^{A+1}\rho_1^{(A+1)A^{k-2}}, \label{pf:prob_good_k_box1}
\end{align}
wherein the last line we appeal to the inductive hypothesis. The result follow from \eqref{pf:prob_good_k_box1} if we prove that
\begin{align}
\left(3^d k^{2d}L_{k-1}^{d(d+1)}\right)^{A+1}\leq \rho_1^{-A^{k-2}},
\label{eq:tobeproved}
\end{align}
by setting $L_1$ large enough and then $\lambda$ small enough. We first note that by \eqref{eq:L_k_def} we have that
$$
L_{k} = \left(\prod_{j=2}^{k}j^{2d^{k-j}}\right)L_1^{d^{k-1}}.
$$
Hence there exists a constant $c>0$ such that if $L_1$ is sufficiently large, we have that 
$$
3^dk^{2d}L_{k-1}^{d(d+1)} \leq L_1^{d^{k+c}}
$$
for all $k\geq2$. %, as the right-hand side dominates the combinatorial term arising in the left, and 
Consequently,
\begin{align}
\left(3^d k^{2d}L_{k-1}^{d(d+1)}\right)^{A+1} \leq L_1^{(A+1)d^{k+c}} \leq L_1^{A^{k+c+2}},  \label{pf:prob_good_k_box2}
\end{align}
where the last inequality follows as $A>d$. Henceforth fix $c$ such that \eqref{pf:prob_good_k_box2} holds. By Lemma~\ref{lem:prob_good_1_box}, for any constant $B>0$, if $L_1$ is large enough and $\lambda$ is small enough then 
$$
L_1 < \rho_1^{-1/B}.
$$
In this case, from \eqref{pf:prob_good_k_box2} we deduce that 
$$
\left(3^d k^{2d}L_{k-1}^{d(d+1)}\right)^{A+1} < \rho_1^{-\tfrac{A^{k+c+2}}{B}} < \rho_1^{-A^{k-2}},
$$
where the last inequality follows for all $k\geq2$ by setting $B$ large enough. Thus, we established~\eqref{eq:tobeproved} and the result follows by induction.
\end{proof}

It is worth noting that in many applications of such multi-scale analysis constructions, \emph{disjoint} boxes do not necessarily mean \emph{independent} boxes. To deduce an analogue of Lemma~\ref{lem:prob_good_k_box} in such cases, one will typically use a decay of correlation or decoupling inequalities which despite adding some terms (that are proven to be of smaller order) still lead to a similar recursion as in the proof of Lemma~\ref{lem:prob_good_k_box}. In our case, disjointness does imply independence and so our recursion argument simplifies to just controlling a combinatorial term.

\section{Multi-scale analysis with non-equilibrium feedback}
\label{sec_new_MSA}

In this section we enhance the multi-scale construction of Section \ref{sec_standard_MSA} in a way that will allow us to handle the non-equilibrium dynamics of FPPHE; 
we call this approach \emph{multi-scale analysis with non-equilibrium feedback}. 
It will be useful to have in mind the intuitive explanation of our idea, which is given in Section~\ref{sec:MAwNEF}.
We start this section giving some further details to this high-level discussion, and then we move to the rigorous construction.

\subsection{High level description}

As explained in Section~\ref{sec:MAwNEF}, we will introduce a notion of \emph{timeliness} to the entrance of a good box, so that 
each \emph{good} $k$-box will be further classified as having either \emph{positive feedback} or \emph{negative feedback}, whereas \emph{bad} $k$-boxes are not further classified.
Roughly speaking, a $k$-box will be called of \emph{positive feedback} if $FPP_1$ is able to occupy a site away from the boundary of the $k$-box sufficiently fast in comparison to the time the box is first visited 
by either $FPP_1$ or $FPP_{\lambda}$.
Moreover, each $k$-box $Q$ will be associated to another $k$-box, 
called its \emph{parent}. The parent of $Q$ is the box that contains inside its core the site of $Q$ that is first visited by either $FPP_1$ or $FPP_{\lambda}$.
The above definitions will be later given in a fully precise manner, and will be tuned in a way to imply the following properties: 
\begin{itemize}
   \item[\eqref{cask}] A $k$-box with positive feedback is mostly occupied by $(k-1)$-boxes of positive feedback, which is a type of \emph{cascading} property of positive feedback boxes.
   \item[\eqref{progk}] The parent of a negative feedback $k$-box is either of negative feedback or a bad $k$-box, which will give rise of a so-called \emph{progenitor structure} as we discuss below.
   \item[\eqref{F_k}] A good $k$-box is entered quickly if there is a nearby $k$-box of positive feedback, giving that positive feedback propagates \emph{fast} enough. 
   \item[\eqref{D_k^C}] A negative feedback $k$-box has an entrance time that is sufficiently after the entrance of its parent (if its parent is also of negative feedback), 
      giving that negative feedback $k$-boxes have a \emph{delayed} entrance time.
\end{itemize}
The emphasized words appearing above are the terms that we chose to represent each property and that give the property its label. 

One could keep in mind from~\eqref{progk} that the parent of $Q$ \emph{witnesses} the fact that $Q$ is of negative feedback. 
So, if the parent of $Q$ is also (a good box) of negative feedback, 
then by \eqref{progk} one finds that it must have a parent itself that is of negative feedback or bad. 
Applying \eqref{progk} inductively, one obtains that any negative feedback box can be associated to a path (or a \emph{trail}) of negative feedback boxes 
(following each box's parent) that culminates into a \emph{bad} box with a good parent, which will be called the \emph{progenitor} of $Q$.  
(We remark that the trail of a negative feedback box can contain in its interior both negative feedback boxes and bad boxes with a negative feedback parent, and culminates into the progenitor which is the first 
bad box with positive feedback parent found along the path of parents.)
We will refer to this as a progenitor structure, which will be the crucial property that will allow 
us to control where negative feedback boxes may occur.
This is precisely the property~\eqref{it:pneg} described in 
Section~\ref{sec:MAwNEF}.
We emphasise~\eqref{progk} is not immediate from the definitions and requires $\lambda$ to be sufficiently small (cf.\ Lemma~\ref{lem:prog_1} and Lemma~\ref{lem:induction1}).
 
Properties~\eqref{F_k} and~\eqref{D_k^C} guarantee that FPPHE gets delayed when traversing a trail of negative feedback boxes, while $FPP_1$ can quickly traverse through positive feedback boxes nearby. 
But note that no delay can be guaranteed when FPPHE passes from a bad box to a negative feedback box (or vice-versa).
Nonetheless, the conclusion we can obtain is that there cannot be long trails composed of large sequences of negative feedback boxes, which together with a control showing that bad boxes are rare
will allow us to confine the trail of a negative feedback box to a small region around its progenitor. 

Putting the properties together will allow us to conclude the following positive property from positive feedback boxes:
\begin{align}
\mbox{positive feedback implies most non-seed sites in said box are occupied by } FPP_1.
\label{cask_roughly}
\end{align}
This is precisely property~\eqref{it:ppos} described in Section~\ref{sec:MAwNEF}.
We note that we cannot guarantee that $FPP_1$ occupies the \emph{whole} of the box for several reasons. 
For example, in scale 1, the box will typically contain a density of seeds, and in higher scale it may contain bad boxes of smaller scale 
inside which we cannot give any guarantee for the spread of $FPP_1$. 
Moreover, just by assuming a box is of positive feedback does not give us much information to control the spread of $FPP_1$ on sites that are close to the boundary of the box. 
This last issue will be solved, however, using the fact that boxes have a large overlap. 
So, the sites that are near the boundary of a box will be contained in the inner part of other boxes, 
so the spread of $FPP_1$ on those sites will be controlled by asserting whether those other boxes are of positive feedback. 

The proof will then proceed as follows. We will start in Section~\ref{sec:setL_1} with a brief clarification on how $L_1$ and $\lambda$ will be set, and then 
define positive feedback for scale 1 boxes in Section~\ref{sec_GE_1}, and introduce the parent structure and the progenitor of a box in
Section~\ref{sec:Parent_progenitor_def}. Then our proof continues through three major steps: 
first we show that~\eqref{cas1} holds, then we show that, for all $k$, property~\eqref{cask} implies~\eqref{progk},~\eqref{F_k} and~\eqref{D_k^C}, 
and ultimately we show that if all properties hold up to scale $k-1$ then~\eqref{cask} holds. The three steps together inductively guarantee that all properties hold for all $k$.
The details of the proofs will be carried out in a few sections, as we need to distinguish the case of scale $1$ from higher scales. 
The proof for scale 1 is derived in Section~\ref{sec_GE_1} for~\eqref{cas1}, Section~\ref{sec:Parent_progenitor_def} for~$\left(\text{Prog}_1\right)$ 
and Section~\ref{sec_control} for~\eqref{F_1_1} and~\eqref{D_1_C}. The other sections are reserved to the definitions and the proofs for boxes of higher scales.
%\section{}

\subsection{Setting $L_1$ and $\lambda$}
\label{sec:setL_1}

Recall that the results of Section~\ref{sec_standard_MSA} hold (in particular, Lemma~\ref{lem:prob_good_k_box}) for all $L_1$ large enough and $\lambda$ small enough. 
That is, there exists a value $\bar L$, which depends on $d$, $p$ and on the constant $A$ from the definition of good boxes, 
and $\bar \lambda=\bar \lambda(L_1, d, p, A)>0$ such that for all $L_1\geq \bar L$, if $\lambda\in(0,\bar\lambda)$ then the above results hold.
Henceforth we fix $L_1$ larger than $\bar L$ and such that $L_1 \geq 5000\log^2 L_1$; the value $5000$ being an arbitrary choice, which just allows us to guarantee that $L_1$ is sufficiently 
larger than $\log^2 L_1$. 
We will also require $L_1$ to be large enough with respect to some conditions on $A$, $d$ and $p$ that we only specify in later sections for ease of exposition.
Now that $L_1$ has been fixed, we shall not write any condition on $L_1$ in the lemmas from this section. 

Regarding $\lambda$, we will require that $\lambda$ is not only smaller than $\bar \lambda$ but also smaller than another constant. To encapsulate the condition on $\lambda$ in a simpler form, 
for any constant $x>0$, we define
\begin{align}
   \lambda_x = \min\left\{\bar\lambda,\frac{1}{\left(x+r_1\right)^2L_1^2}\right\}, \label{def:lambda_x}
\end{align}
where $r_1$ is a constant that only depends on $d$ and $p$ that we set later in Fact~\ref{fact}. Then by properly choosing $x$ in the lemmas below we will enforce proper conditions on $\lambda$. Note that $\lambda_x$ depends on $L_1$, but we omit this from the notation since $L_1$ is regarded as fixed from now onwards.

\subsection{Positive feedback at scale 1}
\label{sec_GE_1}

In this section we define what it means for a good $1$-box to have positive feedback and prove that a natural consequence of a $1$-box having positive feedback is that \emph{most} of the non-seed sites in that box are occupied by $FPP_1$, in a sense that is made rigorous in Lemma~\ref{lem:main_lemma_1}. 

Before we define positive feedback at scale 1, we first introduce some notation that will be useful for the remainder of the paper. Given $v\in\mathbb{Z}^d$, define the \emph{entrance time for $v$}, written as $\tau(v)$, as the earliest time that either $FPP_1$ or $FPP_{\lambda}$ occupies $v$, so that
$$
\tau(v) = \inf\left\{t\geq0: \eta_t(v)\in\{1,2\}\right\}.
$$
For $k\geq1$ and $i\in\mathbb{Z}^d$, define the \emph{entrance time for }$Q_k(i)$, written as $\tau_{k,i}$, as the earliest time any site in $Q_k(i)$ is occupied by either $FPP_1$ or $FPP_{\lambda}$, so that
\begin{equation}
\label{def:entrance_time}
\tau_{k,i}=\inf\left\{\tau(v) \colon v\in Q_k(i)\right\}. 
\end{equation}

\begin{fact}
\label{fact}
Suppose $Q_1(i)$ is a good $1$-box. Recall $\mathcal{C}^{-}_1(i)$ is the component of non-seeds $\mathcal{C}_1(i)$ with sites in $\partial Q_1(i)$ removed. In later arguments it will be useful to have an upper bound on the passage times between any two sites in $\mathcal{C}^{-}_1(i)$. Indeed, as a direct consequence of the definition of a good $1$-box, for any $x,y\in\mathcal{C}^{-}_1(i)$, there exists a path $\gamma$ from $x$ to $y$ such that 
$$
\gamma \subset \mathcal{C}^{-}_1(i) \quad\text{and}\quad T_1(\gamma) < c_2\max\left\{\|x-y\|_1,\log^2 L_1\right\}\leq dc_2 L_1,
$$
where we recall that $c_2$ is a large enough constant set in the definition of a good $1$-box. It will be useful to set $r_1=dc_2$ to be consistent with notation we set later. 
\end{fact}

\begin{definition}[Positive feedback at scale 1]
Let $r_1$ be as given in Fact~\ref{fact} and suppose that $Q_1(i)$ is good. Define $Q_1(i)$ to have \emph{positive feedback} if there exists a site $v\in\mathcal{C}^{-}_1(i)$ occupied by the $FPP_1$ process so that
$$
\tau(v)\leq\tau_{1,i}+r_1L_1.
$$
If $Q_1(i)$ is good but does not have positive feedback, then we define it to have \emph{negative feedback}.
\end{definition}

From positive feedback we now define a property that will be a key component in a cascading argument in the proof of Theorem~\ref{maintheorem}:
\begin{equation}
\label{cas1}
\text{If $Q_1(i)$ has positive feedback and $v\in\mathcal{C}^{-}_1(i)$, then $v$ is occupied by $FPP_1$.} \tag{$\text{CAS}_1$}
\end{equation}
If $Q_1(i)$ is good and $v\in \mathcal{C}^{-}_1(i)$ is occupied by $FPP_{\lambda}$, then 
$$
\tau(v)\geq\tau_{1,i} + \tfrac{1}{\sqrt{\lambda}},
$$
by definition of good $1$-box as $FPP_{\lambda}$ must cross at least one edge whose endpoints are both contained in $Q_1(i)$. 
The intuition is that by setting $\lambda$ to be small enough, we can guarantee that \eqref{cas1} holds, as given in the following lemma. 
Recall the definition of $\lambda_x$ from~\eqref{def:lambda_x}.

\begin{lemma}
\label{lem:main_lemma_1}
If $\lambda\in(0,\lambda_{r_1})$, then \eqref{cas1} holds.
\end{lemma}

\begin{proof}
Suppose $Q_1(i)$ has positive feedback. Let $v\in\mathcal{C}^{-}_1(i)$ and suppose for contradiction that $v$ is occupied by $FPP_{\lambda}$. Without losing generality, we may assume that $v$ is such a site with earliest entrance time in $\mathcal{C}^{-}_1(i)$, so that
$$
\tau(v) = \inf\left\{\tau(w): w\in\mathcal{C}^{-}_1(i)\text{ and }w\text{ is occupied by }FPP_{\lambda}\right\}.
$$
As $v\in\mathcal{C}^{-}_1(i)$, we note that $v$ is not a seed, and so by definition of good $1$-box we have
\begin{equation}
\label{pf:main_lemma_1_1}
\tau(v)\geq\tau_{1,i} + \tfrac{1}{\sqrt{\lambda}}. 
\end{equation}
By definition of positive feedback, there exists a site $u\in\mathcal{C}^{-}_1(i)$ that is occupied by the $FPP_1$ process by time $\tau_{1,i}+r_1L_1$. Recalling Fact~\ref{fact}, there exists a path $\gamma$ from $u$ to $v$ such that $\gamma\subset \mathcal{C}^{-}_1(i)$ and $T_1(\gamma)<r_1L_1$. Thus
\begin{equation}
\label{pf:main_lemma_1_2}
\tau(v) < \tau(u) + r_1L_1 < \tau_{1,i} + 2r_1L_1.
\end{equation}
% Note that by the same reasoning as \eqref{pf:main_lemma_1_1} the earliest time that the $FPP_{\lambda}$ process can occupy a site on $\gamma$ is $\tau_{1,i}+\tfrac{1}{\sqrt{\lambda}}$. 
If $\lambda$ is small enough so that the inequality $2r_1L_1 < \tfrac{1}{\sqrt{\lambda}}$ holds, then \eqref{pf:main_lemma_1_2} contradicts \eqref{pf:main_lemma_1_1}, which establishes the result.
\end{proof}

\subsection{Parent and progenitor structure}
\label{sec:Parent_progenitor_def}

In this section we provide the fundamental structure that will allow us to control entrance times for boxes at all scales. First we introduce the notion of a \emph{parent}. Recall that for all $k\geq1$, $k$-cores give a natural partition of $\mathbb{Z}^d$ and can be used to uniquely identify from where a $k$-box is entered first by either $FPP_1$ or $FPP_{\lambda}$.

\begin{definition}[Parent]
\label{def:parent}
Fix $k\geq1$ and $i\in\mathbb{Z}^d$. Let $e_{k,i}\in\mathbb{Z}^d$ be the almost surely unique site in $Q_k(i)$ that is occupied by either $FPP_1$ or $FPP_{\lambda}$ at time $\tau_{k,i}$, as given in \eqref{def:entrance_time}. We define the \emph{parent} of $Q_k(i)$ as the $k$-box $Q_k(i^{\text{par}})$ satisfying $e_{k,i}\in Q_k^{\text{core}}(i^{\text{par}})$.
\end{definition}

\begin{figure}[!htb]
	\center{\includegraphics[width=0.4\textwidth]{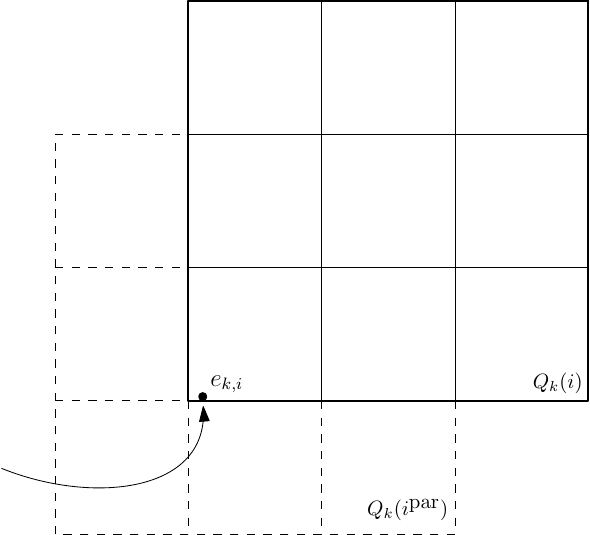}}
   \caption{A $k$-box $Q_k(i)$ and its parent $Q_k(i^{\text{par}})$ where squares represent the cores tessellating $\mathbb{Z}^d$. The arrow represents the spread of $FPP_1$ or $FPP_{\lambda}$ that is first to enter $Q_k(i)$ through $e_{k,i}$.}
	\label{fig:parent} 
\end{figure}

From the definition of parent, there is a \emph{parent structure} that encodes entrance times as a directed tree, where the vertices are $k$-cores, and the root of the tree is the $k$-core containing the origin. See Figure~\ref{fig:parent} for an illustration of the parent structure. This parent structure is useful as it gives us information about where $FPP_1$ or $FPP_{\lambda}$ first enters a box and if $Q_k(i)$ is a $k$-box with parent $Q_k(i^{\text{par}})$, then we know that $\tau_{k,i^{\text{par}}}\leq\tau_{k,i}$.

A key element to our method is deducing that all bad boxes and boxes with negative feedback can be traced to a unique bad box, which in some sense \emph{caused} them to be bad or of negative feedback. 
To make this notion precise, we first define the \emph{progenitor} of a bad box or box with negative feedback at all scales. 
We remark that we have not yet defined positive or negative feedback at higher scales but for ease of exposition we provide the definition of progenitor now.

Before we define progenitor rigorously, it will be useful to fix some notation. Given any $k$-box, we can trace through successive parents as far back as we desire, all the way back to $Q_k(o)$, which is the unique $k$-box that is its own parent. For an arbitrary $k$-box $Q_k(i_1)$, we set
$$
\mathcal{P}_k^{i_1\leftarrow i_N} =\{Q_k(i_1),Q_k(i_2),\ldots,Q_k(i_N)\}
$$
to be a collection of $k$-boxes such that $Q_k(i_n)=Q_k(i_{n-1}^{\text{par}})$ for all $n\in\{2,3,\ldots,N\}$. We define $\mathcal{P}_k^{i_1\leftarrow i_N}$ to be the \emph{path of parents from }$Q_k(i_1)$\emph{ to }$Q_k(i_N)$.
% If $Q_k(i_1)$ is a bad box or a $k$-box with negative feedback, we define its \emph{progenitor} as the first box whose parent has positive feedback obtained by following 
% the path of parents from $Q_k(i_1)$. This is formally written in the definition below.

\begin{definition}[Progenitor]
\label{def:progenitor}
Let $k\geq1$ and suppose $Q_k(i)$ is bad or has negative feedback. Without losing generality, 
we may assume that the $k$-box whose core contains the origin has positive feedback. 
We let $Q_k(i^{\text{prog}})$ to be the unique $k$-box such that the path of parents $\mathcal{P}_k^{i\leftarrow i^{\text{prog}}}$ contains only bad $k$-boxes or $k$-boxes with negative feedback and the parent of $Q_k(i^{\text{prog}})$ has positive feedback. 
We define $Q_k(i^{\text{prog}})$ to be the \emph{progenitor} of $Q_k(i)$.
\end{definition}

For $k\geq1$, if $Q_k(o)$ is good and $o\in\mathcal{C}_1(o)$ then it will follow that $Q_k(o)$ has positive feedback. 
This is clear from the definition of positive feedback at scale $1$ and will follow once we define positive feedback at higher scales. 
Hence if $Q_k(o)$ has positive feedback and $Q_k(i)$ is bad or has negative feedback, then $\mathcal{P}_k^{i\leftarrow o}$ contains at least one $k$-box with positive feedback, and thus the progenitor is well-defined. 
% This is why assuming the origin is flawless is natural in the definition of progenitor. 
% Moreover, the origin being flawless is an event of positive probability by Lemma~\ref{lem:flawless}, under a suitable choice of $L_1$ and $\lambda$.

The definition of progenitor does not on its own provide a useful structure to better understand the structure of boxes with negative feedback. Indeed, it could be the case that a box with negative feedback is its own progenitor. The utility of the definition of progenitor becomes apparent by introducing the following property that we refer to as a \emph{progenitor structure}:
\begin{equation}
\label{progk}
\text{If $Q_k(i)$ has negative feedback, then $Q_k(i^{\text{par}})$ is either bad or has negative feedback.} \tag{$\text{Prog}_k$}
\end{equation}
If there is a progenitor structure at scale $k$, then every $k$-box with negative feedback can be traced back to its progenitor, which must be a bad $k$-box. 
Moreover, all $k$-boxes on the path of parents from a $k$-box with negative feedback to its progenitor must be bad or have negative feedback. These are key properties that will allow us to control the entrance times of boxes with negative feedback. In Figure~\ref{fig:ProgStructure}, we see how to implement the progenitor structure to find the progenitor of a bad box or negative feedback box from the parent structure. The arrows provide the parent structure, and by backwards traversing the arrows from a pink or red box, we arrive at a unique red (bad) box that is the progenitor. We note that it is possible for the parent of a box of positive feedback to have negative feedback or be bad. 

\begin{figure}[!htb]
	\center{\includegraphics[width=0.4\textwidth]{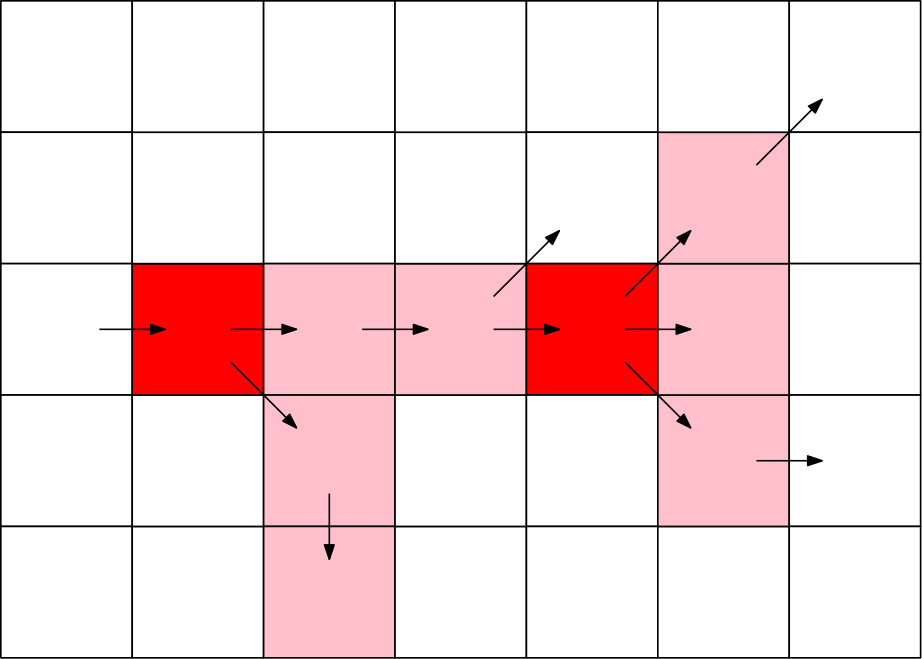}}
   \caption{Above is an illustration of the progenitor structure from a bad box. Red boxes represent cores of bad boxes, pink boxes represent the cores of boxes of negative feedback and
   the white boxes represent cores of boxes of positive feedback. 
   The arrows point to the core of a box from the core of its parent. The unique box with a parent of positive feedback (the leftmost red box) is the progenitor of all the other pink and red boxes.}
	\label{fig:ProgStructure} 
\end{figure}

The following lemma gives that if $\lambda$ is small enough so that \eqref{cas1} holds, then $(\text{Prog}_1)$ holds, providing the desired progenitor structure at scale $1$.

\begin{lemma}
\label{lem:prog_1}
If $\lambda\in(0,\lambda_{r_1})$, then $(\emph{Prog}_1)$ holds.
\end{lemma}

\begin{proof}
Assume $Q_k(i)$ has negative feedback, so that there is no site within $\mathcal{C}^{-}_1(i)$ occupied by $FPP_1$ at time $\tau_{1,i} + r_1L_1$, and $\lambda\in(0,\lambda_{r_1})$. Suppose for contradiction that $Q_1(i^{\text{par}})$ has positive feedback. Recall that $e_{1,i}$ is the almost surely unique site such that $\tau_{1,i}=\tau(e_{1,i})$ and by definition of parent, $e_{1,i}\in Q_1^{\text{core}}(i^{\text{par}})$. Let $u\notin Q_1(i)$ be the neighbour of $e_{1,i}$ that transmitted either $FPP_1$ or $FPP_{\lambda}$ to $e_{1,i}$ at time $\tau_{1,i}$. 

First consider the case where $u$ is occupied by $FPP_{\lambda}$ and so $e_{1,i}$ is also occupied by $FPP_{\lambda}$. It follows immediately from the definition of good box that
\begin{equation}
\label{pf:prog_1_1}
\tau_{1,i}=\tau(e_{1,i}) \geq \tau(u) + \tfrac{1}{\sqrt{\lambda}} > \tau_{1,i^{\text{par}}} +  \tfrac{1}{\sqrt{\lambda}}.
\end{equation}
As $Q_1(i^{\text{par}})$ is assumed to have positive feedback, there exists some site $z\in\mathcal{C}^{-}_1(i^{\text{par}})$ such that 
$$
\tau(z) < \tau_{1,i^{\text{par}}} + r_1L_1.
$$
Fix a site $z_0\in\mathcal{C}^{-}_1(i)\cap \mathcal{C}^{-}_1(i^{\text{par}})$. By Fact~\ref{fact}, there exists a path $\gamma_0$ from $z$ to $z_0$ such that 
$$
\gamma_0\subset \mathcal{C}^{-}_1(i^{\text{par}}) \quad \text{ and }\quad T_1(\gamma_0) < r_1L_1.
$$
As $Q_1(i^{\text{par}})$ has positive feedback we have that every site on $\gamma_0$ is occupied by $FPP_1$ by~\eqref{cas1} (cf.\ Lemma~\ref{lem:main_lemma_1}). Hence
\begin{equation}
\label{pf:prog_1_2}
\tau_{1,i} \leq \tau(z_0) < \tau(z) + r_1L_1 < \tau_{1,i^{\text{par}}} + 2r_1L_1. 
\end{equation}
As $\lambda\in(0,\lambda_{r_1})$, we have that \eqref{pf:prog_1_1} contradicts \eqref{pf:prog_1_2} and so $u$ is instead occupied by $FPP_1$. As $u$ neighbours a site in $Q_1^{\text{core}}(i^{\text{par}})$ then the distance of $u$ from $\partial Q_1(i^{\text{par}})$ is at least $L_1 / 6$. By definition of good $1$-box we have that $|\mathcal{C}_2(i^{\text{par}})|\leq \log^2|Q_1(i^{\text{par}})|$ and so long as $L_1$ is large enough, it must be the case that $u\in\mathcal{C}^{-}_1(i^{\text{par}})$.

Arbitrarily fix a site $v\in\mathcal{C}^{-}_1(i) \cap\mathcal{C}^{-}_1(i^{\text{par}})$. As $Q_1(i^{\text{par}})$ is good, by Fact~\ref{fact} there exists a path $\gamma_1$ from $u$ to $v$ such that
\begin{equation}
\label{pf:prog_1_3}
\gamma_1\subset \mathcal{C}^{-}_1(i^{\text{par}}) \quad \text{and}\quad T_1(\gamma_1) < r_1L_1. 
\end{equation}
By Lemma~\ref{lem:main_lemma_1}, as $Q_1(i^{\text{par}})$ has positive feedback, every site on $\gamma_1$ is occupied by $FPP_1$. Moreover, by \eqref{pf:prog_1_3} we have
$$
\tau(v)<\tau(u)+r_1L_1<\tau(e_{1,i})+r_1L_1=\tau_{1,i}+r_1L_1,
$$
contradicting the negative feedback of $Q_1(i)$. Thus $Q_1(i^{\text{par}})$ is bad or has negative feedback, establishing the result.
\end{proof}

\subsection{Establishing the fundamental properties at scale 1} \label{sec_control}

As a consequence of Lemma~\ref{lem:main_lemma_1}, it is intuitive that a $1$-box having positive feedback induces fast entrance times to nearby good $1$-boxes due to the large overlap between neighbouring $1$-boxes. We make this notion precise in the following property that gives this control for $FPP_1$.
\begin{equation}
   \label{F_1_1}
   \begin{array}{r}
   \text{If $Q_1(i)$ is good and $Q_1(j)$ has positive feedback}\\
    \text{with $\|i-j\|_{\infty}=1$, then $\tau_{1,i} < \tau_{1,j} + 2r_1L_1$.} \tag{$\text{Fast}_1$}\\
   \end{array}
\end{equation}
In the following lemma we prove that if $\lambda$ is small enough so that \eqref{cas1} holds, then \eqref{F_1_1} holds as well.
\begin{lemma}
\label{lem:fast_1}
If $\lambda\in(0,\lambda_{r_1})$, then \eqref{F_1_1} holds.
\end{lemma}

\begin{proof}
Let $i,j$ be such that $Q_1(i)$ is good, $Q_1(j)$ has positive feedback and $\|i-j\|_\infty=1$. By definition of positive feedback, there exists a site $u\in\mathcal{C}^{-}_1(j)$ such that $u$ is occupied by $FPP_1$ and $\tau(u)<\tau_{1,j} + r_1L_1$. Arbitrarily fix a site $v\in\mathcal{C}^{-}_1(j)\cap\mathcal{C}^{-}_1(i)$. By Fact~\ref{fact} there exists a path $\gamma$ from $u$ to $v$ such that $\gamma\subset\mathcal{C}^{-}_1(j)$ and $T_1(\gamma)<r_1L_1$. As $\lambda\in(0,\lambda_{r_1})$, by Lemma~\ref{lem:main_lemma_1}, all sites on $\gamma$ are occupied by $FPP_1$ and thus
$$
\tau_{1,i} \leq \tau(v) < \tau(u) + r_1L_1 < \tau_{1,j} + 2r_1L_1.
$$
\end{proof}

Note that in \eqref{F_1_1}, we cannot deduce that $Q_1(i)$ has positive feedback. It may be the case that it was entered much earlier and actually has negative feedback but Lemma~\ref{lem:fast_1} gives us that it is impossible that $Q_1(i)$ is entered significantly later than any positive feedback neighbor.

For any constant $C>0$, define the property:
\begin{equation}
\label{D_1_C}
\text{ If both $Q_1(i)$ and $Q_1(i^{\text{par}})$ have negative feedback, then $\tau_{1,i}>\tau_{1,i^{\text{par}}}+CL_1.$} \tag{$\text{Del}_1^C$}
\end{equation}
In other words, the property \eqref{D_1_C} guarantees that there is at least a \emph{delay} of time at least $CL_1$ in the entrance times of a negative feedback $1$-box and its parent if the parent also has negative feedback. In the lemma below, we prove that for any constant $C>0$, the property \eqref{D_1_C} holds if $\lambda$ is sufficiently small, 
providing us the control on $FPP_{\lambda}$ at scale 1 that we will ultimately require.

\begin{lemma}
\label{lem:del_1}
   For any $C>0$, if $\lambda \in \left(0,\lambda_C\right)$, then \eqref{D_1_C} holds.
\end{lemma}

\begin{proof}
Suppose that $Q_1(i)$ and $Q_1(i^\text{par})$ both have negative feedback and $\lambda\in(0,\lambda_C)$. Recall that $e_{1,i}$ is the almost surely unique site such that $\tau(e_{1,i})=\tau_{1,i}$. By definition of parent, we have
$$
e_{1,i}\in Q_1^{\text{core}}(i^{\text{par}})\subset Q_1(i).
$$
Let $u\notin Q_1(i)$ be the neighbouring site of $e_{1,i}$ that transmitted either $FPP_1$ or $FPP_{\lambda}$ to $e_{1,i}$ at time $\tau_{1,i}$. First note that by construction $u$ is occupied by either $FPP_1$ or $FPP_{\lambda}$ before $Q_1(i)$ is first entered by either process and so
\begin{equation}
\label{pf:del_1_1}
\tau_{1,i} > \tau(u).
\end{equation}
If $u$ was occupied by $FPP_{\lambda}$, then it follows from the definition of good $1$-box that
$$
\tau_{1,i}-\tau_{1,i^{\text{par}}}\geq \tfrac{1}{\sqrt{\lambda}}>CL_1, 
$$
where the final inequality follows as $\lambda\in(0,\lambda_C)$.

Suppose instead that $u$ was occupied by the $FPP_1$ process and in this case, by construction of good $1$-box and parent, we have that $u\in\mathcal{C}^{-}_1(i^{\text{par}})$. Arbitrarily fix a site $v \in \mathcal{C}^{-}_1(i)\cap\mathcal{C}^{-}_1(i^{\text{par}})$. By Fact~\ref{fact}, there exists a path $\gamma$ from $u$ to $v$ such that
$$
\gamma \subset \mathcal{C}^{-}_1(i^{\text{par}})\quad\text{and}\quad T_1(\gamma) < r_1 L_1.
$$
By definition of $Q_1(i)$ having negative feedback, there is no site in $\mathcal{C}^{-}_1(i)$ occupied by $FPP_1$ at time $\tau_{1,i}+r_1L_1$. However, if every site on $\gamma$ was occupied by $FPP_1$, then Fact~\ref{fact} would imply that $\tau(v)<\tau_{1,i} + r_1L_1$, giving that $Q_1(i)$ has positive feedback, which is a contradiction. Thus there exists a site $v_0\in\gamma$ that is occupied by $FPP_{\lambda}$ with 
\begin{equation}
\label{pf:del_1_2}
\tau(v_0) < \tau(u) + r_1L_1.
\end{equation}
By definition of good $1$-box and as $v_0$ is occupied by $FPP_{\lambda}$, we have
\begin{equation}
\label{pf:del_1_3}
\tau(v_0)\geq\tau_{1,i^{\text{par}}}+\tfrac{1}{\sqrt{\lambda}}.
\end{equation}
Putting together \eqref{pf:del_1_1}, \eqref{pf:del_1_2} and \eqref{pf:del_1_3}, we observe that
$$
\tau_{1,i}+r_1L_1>\tau(u)+r_1L_1 > \tau(v_0)\geq\tau_{1,i^{\text{par}}}+\tfrac{1}{\sqrt{\lambda}}.
$$
By rearranging the above equation, we see that
\begin{align*}
\tau_{1,i} - \tau_{1,i^{\text{par}}} \geq \tfrac{1}{\sqrt{\lambda}} -r_1L_1 > CL_1,
\end{align*}
where the final inequality holds as $\lambda\in(0,\lambda_C)$.
\end{proof}

\subsection{Definition of clusters and positive feedback at higher scales}
\label{sec:positivefeedback}
In this section we define positive feedback at higher scales. For this we must make a precise notion of what it means for a good box to be well-separated from bad boxes. In this direction we introduce the idea of \emph{clusters} of $(k-1)$-boxes in a $k$-box.

\begin{definition}[Clusters]
   Let $k\geq2$ and consider the $k$-box $Q_{k}(i)$. We will define the clusters of $Q_{k}(i)$, 
   which will be subsets of the $(k-1)$-boxes contained in $Q_{k}(i)$. Start by taking the set of all $(k-1)$-boxes that have a non-empty intersection with the \emph{bad} $(k-1)$-boxes from $Q_{k}(i)$. 
   Note that the bad $(k-1)$-boxes are contained in this set. Next, add to this set all the $(k-1)$-boxes that are disconnected from infinity by it. This new set may not be connected, 
   but it can be naturally split into its connected components, which we refer to as the \emph{clusters} of $Q_{k}(i)$. Let $\mathcal{A}$ be such a cluster. We define the \emph{outer-boundary} of $\mathcal{A}$, 
   denoted by $\partial \mathcal{A}$, as the set of $(k-1)$-boxes $Q_{k-1}(\ell)\not\in\mathcal{A}$ for which there exists $Q_{k-1}(\iota)\in\mathcal{A}$ with $\|\ell-\iota\|_{\infty}=1$. 
   We will abuse notation and sometimes will let $\partial\mathcal{A}$ refer to the vertices contained in the $(k-1)$-boxes in $\partial\mathcal{A}$, but the meaning will be clear from the context.  See Figure~\ref{fig:clusters} for an illustration of clusters.
\end{definition}

Fix $k\geq2$ and suppose $Q_k(i)$ is a good $k$-box with cluster $\mathcal{A}$. As there are at most $A$ disjoint bad $(k-1)$-boxes in $Q_k(i)$, with the $d$-dimensional isoperimetric inequality we deduce that there exists a constant $\sigma=\sigma(A,d)\in(0,\infty)$ such that 
\begin{align*}
\text{the number of $(k-1)$-boxes contained in }\mathcal{A}\cup\partial\mathcal{A}\leq \sigma.
\end{align*}
Moreover, since there are at most $A$ clusters in $Q_k(i)$, then the total number of $(k-1)$-boxes that are contained in some cluster of $Q_k(i)$ is bounded above by $A\sigma$, which is a constant that only depends on $A$ and $d$. 

With the definition of clusters of a $k$-box to hand, we have a notion of bad $(k-1)$-boxes in a good $k$-box being \emph{well-separated} from one another; in particular, boxes that belong to different clusters are sufficiently far apart. It will be useful to also have a global notion of well-separated, meaning that a site is sufficiently far from bad boxes at all scales. We make this notion precise in the following definition.

\begin{definition}[Flawless]
   \label{def:flawless}
   Let $x\in\mathbb{Z}^d$, and define $i$ so that $x\in Q_1^\text{core}(i)$.  
   Define $x$ to be \emph{flawless} if it is not contained in a cluster of a box at any scale and is contained in $\mathcal{C}_1(i)$; recall the definition of $\mathcal{C}_1(i)$ from the 
   beginning of Section~\ref{sec:good_1_boxes}.
\end{definition}

Our main results concern certain events occurring with positive probability, namely survival of $FPP_1$ in Theorem~\ref{maintheorem}, and survival of $FPP_1$ and non-survival of $FPP_{\lambda}$ in Theorem~\ref{thm:strongsurvival}. It transpires that the origin being flawless is the specific event we utilise to prove these results. In the following lemma, we prove that such an event occurs with positive probability for a large enough choice of $L_1$ and small enough choice of $\lambda$.

\begin{lemma}
\label{lem:flawless}
Let $x\in\mathbb{Z}^d$ and $A>d$. For a large enough choice of $L_1$ and then small enough choice of $\lambda$, we have that $x$ is flawless with positive probability.
\end{lemma}
\begin{proof}
	Fix $x\in\mathbb{Z}^d$ and for $k\geq1$, let $\mathfrak{Q}_k^x$ be the collection of $k$-boxes that contain $x$. 
	Let $i_0\in\Z^d$ be so that $x\in Q_1^\text{core}(i_0)$. We may write the event that $x$ is not flawless as
	\begin{align}
   	\label{eq:not_flawless}
   	\{x\text{ is not flawless}\} = \bigcup_{k\geq1}\bigcup_{Q_{k}(i)\in\mathfrak{Q}_k^x}H_k(i)\cup\{x\notin\mathcal{C}_1(i_0)\},
	\end{align}
	where $H_k(i)$ is the event that $Q_k(i)$ is in a cluster of some box in $\mathfrak{Q}_{k+1}^x$.
	As each cluster in a good $(k+1)$-box contains no more than $\sigma$ $k$-boxes, we deduce that
	\begin{align}
	\label{eq:H_k}
	\mathbb{P}\left(H_k(i)\right) &\leq \mathbb{P}\left(
\mbox{there exists a bad }k\mbox{-box }Q_k(j)\mbox{ such that }\|i-j\|_{\infty}\leq 10\sigma\right)\leq c\rho_k,
	\end{align}
	where $c=c(A,d)>0$ and $\rho_k$ is the probability that an arbitrary $k$-box is bad. 
	Note that 
	\begin{align}
	\mathbb{P}\left(\bigcap_{k\geq1}\bigcap_{Q_{k}(i)\in\mathfrak{Q}_k^x}H_k^\text{c}(i)\cap \{x\notin\mathcal{C}_1(i_0)\}\right) &\leq \mathbb{P}\left(x\notin\mathcal{G}\right)
	\leq 1-\theta(1-p).
	\end{align}
	
	From \eqref{eq:not_flawless}, \eqref{eq:H_k} and the union bound, we have
	\begin{align}
	\mathbb{P}\left(x \text{ is not flawless}\right) &\leq 1-\theta(1-p)+\sum_{k\geq1}\sum_{Q_k(i)\in\mathfrak{Q}_k^x} c\rho_k,\nonumber\\
	&\leq 1-\theta(1-p) + c'\sum_{k\geq1}\rho_k, 
	\label{pf:flawless1}
	\end{align}
	where $c'=c'(A,d)>0$ and $\theta(1-p)>0$ as $p<1-p_c^{\text{site}}$. By recalling Lemma~\ref{lem:prob_good_k_box}, the sum in \eqref{pf:flawless1} can be made arbitrarily small by setting $\rho_1$ small enough in Lemma~\ref{lem:prob_good_1_box}, establishing the result.
\end{proof}

As the choice of $L_1$ and $\lambda$ in Lemma~\ref{lem:flawless} only depend on $A$, $d$ and $p$, we may add this condition to how we set $L_1$ and $\lambda$ in Section~\ref{sec:setL_1}.

There is an issue of a cluster being too close to the boundary of a $k$-box, as it may be part of a much larger cluster outside of that $k$-box. 
To address this, we introduce the notion of \emph{boundary clusters}.

\begin{definition}[Boundary and inner clusters]
\label{def:boundary_cluster}
Let $k\geq2$, $Q_{k}(i)$ be good and $\mathcal{A}$ be a cluster in $Q_{k}(i)$. We define $\mathcal{A}$ to be a \emph{boundary cluster} in $Q_{k}(i)$ if the outer-boundary of $\mathcal{A}$ has a non-empty intersection with the boundary of $Q_{k}(i)$, so that
$$
\partial\mathcal{A}\cap \partial Q_{k}(i) \neq \emptyset.
$$
Otherwise, we define $\mathcal{A}$ to be an \emph{inner cluster} in $Q_{k}(i)$. 
\end{definition}

In Figure~\ref{fig:clusters}, the two uppermost clusters intersect the boundary clusters and the rest are inner clusters.

We let $\mathcal{W}(k,i)$ be the set of $(k-1)$-boxes in a good $k$-box $Q_{k}(i)$ that are not contained in any clusters and call such $(k-1)$-boxes \emph{wonderful}. In other words, if $Q_{k}(i)$ is good and $\mathcal{A}_1,\mathcal{A}_2,\ldots,\mathcal{A}_{n}$ are all the clusters of bad $(k-1)$-boxes in $Q_{k}(i)$, then
$$
\mathcal{W}(k,i) = \left\{Q_{k-1}(\ell):Q_{k-1}(\ell)\subset Q_{k}(i) \text{ and } Q_{k-1}(\ell)\notin \bigcup_{j=1}^{n}\mathcal{A}_j\right\}.
$$
Before we give the next definition, it will be useful to also define a collection of boxes as the following. For all $k\geq1$, let
\begin{align}
\left\{Q^{\text{inn}}_k\left(i\right)\right\}_{i\in\mathbb{Z}^d} \text{ with } Q^{\text{inn}}_k\left(i\right) = \left(L_k/3\right)i + \left[-499L_k/1000,499L_k/1000\right]^d, \label{def:inner_part}
\end{align}
where we refer to $Q^{\text{inn}}_k\left(i\right)$ as the \emph{inner part} of $Q_k(i)$. In some cases we will abuse notation and instead mean that $Q_k^{\text{inn}}(i)$ refers to all the $(k-1)$-boxes contained in the set given by \eqref{def:inner_part} rather than the sites but the context will always leave this as unambiguous. Intuitively, the inner part of a $k$-box contains all sites that are sufficiently far from the boundary of the $k$-box. 
Indeed, so long as the distance of sites in the inner part is of the order $L_k$ from the boundary, our analysis will work and so the choice of $L_{k}/1000$ is purely for concreteness.

\begin{definition}[$\mathcal{W}^{\text{inn}}(k,i)$]
\label{def:W_inn}
Let $k\geq2$ and $Q_{k}(i)$ be good. Let $\mathfrak{A}(k,i)$ be the collection of all clusters in $Q_{k}(i)$ whose outer-boundary has a non-empty intersection with $Q_{k}^{\text{inn}}(i)$, so that if $\mathcal{A}\in\mathfrak{A}(k,i)$, then 
$$
\partial \mathcal{A} \cap Q_{k}^{\text{inn}}(i) \neq \emptyset.
$$
Define $\mathcal{W}^{\text{inn}}(k,i)$ to be the union of wonderful $(k-1)$-boxes in $Q_{k}^{\text{inn}}(i)$ along with the outer-boundary of clusters in $\mathfrak{A}(k,i)$, so that
$$
\mathcal{W}^{\text{inn}}(k,i) = \left(\mathcal{W}(k,i)\cap Q_k^{\text{inn}}(i)\right)\cup\bigcup_{\mathcal{A}\in\mathfrak{A}(k,i)}\partial \mathcal{A}.
$$
\end{definition}

\begin{figure}[!htb]
	\center{\includegraphics[width=0.5\textwidth]{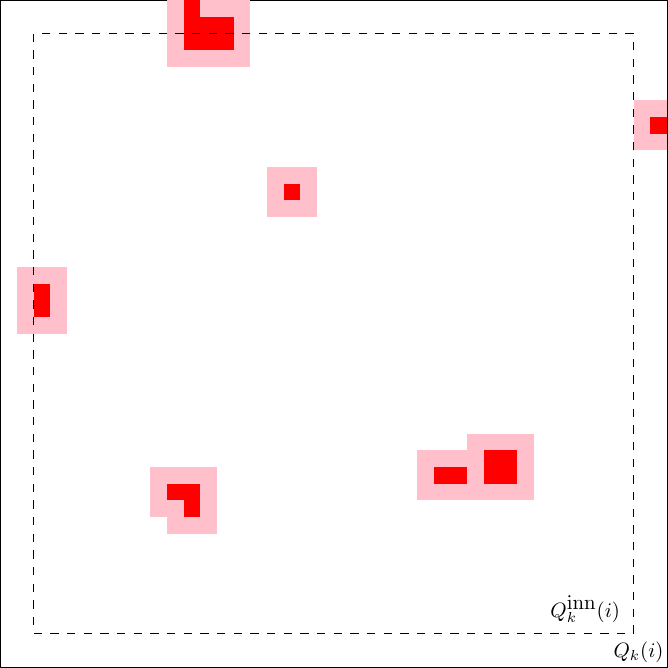}}
   \caption{A $k$-box $Q_k(i)$ with clusters given in red and their respective outer boundary in pink.}
	\label{fig:clusters} 
\end{figure}

It transpires that $\mathcal{W}^{\text{inn}}(k,i)$ is what we choose to replace $C^{-}_1(i)$ at higher scales. Roughly speaking, we will tune the definition of positive feedback at higher scale so that if there is a $(k-1)$-box with positive feedback in $\mathcal{W}^{\text{inn}}(k,i)$ at a relatively early time compared to when either $FPP_1$ or $FPP_{\lambda}$ first enters $Q_k(i)$, then all $(k-1)$-boxes in $\mathcal{W}^{\text{inn}}(k,i)$ will have positive feedback. This is made rigorous later in Lemma~\ref{lem:cascading}.

\begin{remark}
\label{rem:completion}
It is worth examining why we choose to define $\mathcal{W}^{\text{inn}}(k,i)$ as we did in Definition~\ref{def:W_inn} rather than just setting $\mathcal{W}^{\text{inn}}(k,i)$ to be $\mathcal{W}(k,i)\cap Q_k^{\text{inn}}(i)$. The issue is that $\mathcal{W}(k,i)\cap Q_k^{\text{inn}}(i)$ may not be a single connected component of wonderful $(k-1)$-boxes in $Q_k(i)$. It will be essential in later proofs that we can find paths of $(k-1)$-boxes through $\mathcal{W}^{\text{inn}}(k,i)$. To be explicit, for $k\geq1$, we say that $\pi_{k}=\{Q_k(i_1),\ldots,Q_k(i_n)\}$ is a \emph{path of $k$-boxes} if $\|i_{\ell+1}-i_{\ell}\|_{\infty}=1$ for all $1\leq \ell<n$ and write $|\pi_{k}|=n$. If $\mathcal{W}^{\text{inn}}(k,i)$ is not a single connected component, then this may not be possible without exiting $\mathcal{W}^{\text{inn}}(k,i)$. By including the outer boundaries of any clusters that intersect $Q_k^{\text{inn}}(i)$, we guarantee that $\mathcal{W}^{\text{inn}}(k,i)$ is a single connected component. Hence, if 
$$
Q_{k-1}(j_0),Q_{k-1}(j_1)\in\mathcal{W}^{\text{inn}}(k,i),
$$
then there exists a path $\pi_{k-1}$ of $(k-1)$-boxes from $Q_{k-1}(j_0)$ to $Q_{k-1}(j_1)$ such that
$$
\pi_{k-1}\subset\mathcal{W}^{\text{inn}}(k,i)\quad\mbox{and}\quad|\pi_{k-1}| \leq \tfrac{3dL_{k}}{L_{k-1}} + A\sigma.
$$
\end{remark}

\begin{definition}[Positive feedback at higher scales] 
We define positive and negative feedback at higher scales inductively. Let $k\geq2$, suppose $Q_k(i)$ is good, and positive and negative feedback are defined at scale $k-1$. Define
\begin{equation}
\label{def:r_k}
r_{k} = r_1\prod_{j=2}^{k}\left(1+\tfrac{a_1}{j^2L_{j-1}^{d-1}}\right),
\end{equation}
where $r_1$ is as given in Fact~\ref{fact} and $a_1=\tfrac{1000A\sigma}{3}$. Note that $a_1$ depends only on $A$ and $d$. Recall the definition of entrance time for a $k$-box as given in \eqref{def:entrance_time}. If there exists $Q_{k-1}(\ell)\in\mathcal{W}^{\text{inn}}(k,i)$ with positive feedback such that
$$
\tau_{k-1,\ell} < \tau_{k,i} + \tfrac{3}{500}r_kL_k,
$$
then we define $Q_k(i)$ to have \emph{positive feedback}. If $Q_k(i)$ does not have positive feedback, then we define it to have \emph{negative feedback}.
\end{definition}

\begin{remark}
\label{rem:r}
The constant $r_k$ is the analogue of $r_1$ at scale $k$ and \eqref{def:r_k} is justified when we generalise \eqref{F_1_1} to higher scales. From \eqref{def:r_k}, we deduce that $\{r_k\}_{k\geq1}$ is an increasing sequence such that
$$
r = \lim_{k\to\infty}r_k<\infty.
$$
Moreover, so long as $L_1$ is sufficiently large, then $r<2r_1$. 
We assume this is the case, and note that this constraint for $L_1$ depends on $a_1$, which only depends on $A$ and $d$.
Hence, we can add this constraint to how we set $L_1$ in Section~\ref{sec:setL_1}.
\end{remark}

\subsection{Fundamental properties at higher scales}
\label{sec:fprophigh}

In this section we generalise the fundamental properties \eqref{cas1}, \eqref{F_1_1} and \eqref{D_1_C} to higher scales; recall that 
we have already defined the fundamental property \eqref{progk} at all scales.  
We will also introduce a fifth fundamental property that will give information about how $FPP_{\lambda}$ spreads through clusters in a good box.

Recall that \eqref{cas1} states that all sites in the large component of non-seeds in a good $1$-box away from the boundary are occupied by $FPP_1$ and that this property follows from positive feedback by setting $\lambda$ small enough, as stated in Lemma~\ref{lem:main_lemma_1}. 
The higher scale version of \eqref{cas1} is given below. 
For $k\geq2$ define the following property:
\begin{equation}
\label{cask}
\begin{array}{l}
\text{If $Q_{k}(i)$ has positive feedback, then for all $i_0$ such that $Q_{k-1}(i_0)\in \mathcal{W}^{\text{inn}}(k,i)$},\\
\text{we have that $Q_{k-1}(i_0)$ has positive feedback.} \tag{$\text{CAS}_k$}\\
\end{array}
\end{equation}

This section contains two main results: the first one is stated below, which gives that 
by setting $\lambda$ small enough, we have that \eqref{cask} holds for all $k\geq2$. 
The second main result will be given later, in Lemma~\ref{lem:induction1}.
\begin{lemma}
\label{lem:cascading}
There exists a constant $C^*>0$ such that if $\lambda\in(0,\lambda_{C^*})$, then \eqref{cask} holds for all $k\geq1$.
\end{lemma}

% The proof of Theorem~\ref{maintheorem} follows readily from a cascading argument due to Lemma~\ref{lem:cascading} and most of the remainder of this paper is devoted to this latter result. 
By Lemma~\ref{lem:main_lemma_1} we know that we may set $\lambda$ so that \eqref{cas1} holds but the issue is that if $\lambda$ is such that \eqref{cask} holds, it is not immediately clear how one should set $\lambda$ to ensure that $(\text{CAS}_{k+1})$ holds. Indeed if this procedure requires that $\lambda$ tends to $0$ as $k\to\infty$ then this is not sufficient to prove Lemma~\ref{lem:cascading}. Hence the proof of Lemma~\ref{lem:cascading} will require careful consideration of how \eqref{cask} implies $(\text{CAS}_{k+1})$ and choosing $\lambda$ in such a way that it does not depend on the scale $k$.

The first step in proving Lemma~\ref{lem:cascading} is generalising \eqref{F_1_1} and \eqref{D_1_C} to higher scales. %, together with~\ref{cask} they will be referred to as the \emph{fundamental properties}. 
Then we introduce the notion of the \emph{confinement} of bad clusters to account for the existence of bad boxes in good boxes at higher scales. 
Once these ideas have been formalised, we prove that if $\lambda$ is such that \eqref{cask} holds and is below a certain constant we define later that does not depend on $k$, then 
\eqref{progk}, \eqref{F_k} and \eqref{D_k^C}
all hold, this will be the content of the second main result of this section, Lemma~\ref{lem:induction1}. 
Finally we show that if all such properties hold up to scale $k$, then $(\text{CAS}_{k+1})$ holds for a possible smaller value of $\lambda$. 
Noting that the required upper bound on $\lambda$ converges to a value that is bounded away from $0$ 
as $k$ goes to infinity establishes Lemma~\ref{lem:cascading}.

In order to establish the fundamental properties at scale $k\geq2$, we need to pay careful attention to the presence of bad boxes at lower scales, in which we cannot control either $FPP_1$ or $FPP_{\lambda}$. For example, it could happen that $FPP_1$ is \emph{slowed down} in some regard as it has to deviate around bad boxes to spread, whereas $FPP_{\lambda}$ may be \emph{sped up} by being able to spread through bad boxes at no cost. If the slowing down or speeding up of the respective processes can accumulate in a non-trivial way, then we may lose control of both of the processes. The idea is that there are relatively few bad $(k-1)$-boxes in a good $k$-box, so any slowing down or speeding up at scale $(k-1)$ is minuscule relative to scale $k$.

%The first lemma for the fundamental properties at higher scales is the following. If $k\geq2$ and $\lambda$ is such that \eqref{cask} holds, then there is a progenitor structure at scale $k$.
%\begin{lemma}
%\label{lem_parentstructure_k}
%Let $k\geq2$. If $\lambda$ is such that \eqref{cask} holds up to and including scale $k$ then \eqref{progk} also holds.
%\end{lemma}

For $k\geq2$ we define the following property that generalises \eqref{F_1_1} to higher scales:
\begin{equation}
\label{F_k}
\begin{array}{r}
\text{If $Q_k(i)$ is good and $Q_k(j)$ has positive feedback}\\
 \text{with $\|i-j\|_{\infty}=1$, then $\tau_{k,i} < \tau_{k,j} + 2r_kL_k$,} \tag{$\text{Fast}_k$}\\
\end{array}
\end{equation}
where $r_k$ is as defined in \eqref{def:r_k}. \

Next we generalise \eqref{D_1_C} to higher scales. For $k\geq2$ and any constant $C>0$, define the property:
\begin{equation}
\label{D_k^C}
\text{If both $Q_k(i)$ and $Q_k(i^{\text{par}})$ have negative feedback, then $\tau_{1,i}>\tau_{1,i^{\text{par}}}+C\omega_kL_k$,} \tag{$\text{Del}_k^C$}
\end{equation}
where 
\begin{equation}
\label{def:omega_k}
\omega_k = \prod_{m=2}^{k}\left(1 - \tfrac{a_2}{m^2L_{m-1}^{d-1}}\right),
\end{equation}
and $a_2>0$ is a constant that only depends on $A$, $d$ and $p$ that we set later in the proof of Lemma~\ref{lem:induction1}.

\begin{remark}
\label{rem:omega}
From \eqref{def:omega_k}, we deduce that $\{\omega_k\}_{k\geq1}$ is a decreasing sequence with 
$$
\omega = \lim_{k\to\infty}\omega_k >0.
$$
The sequence $\{\omega_k\}_{k\geq1}$ accounts for the possible increase of speed that $FPP_{\lambda}$ gains at high scales due to the existence of bad boxes at lower scales. From \eqref{def:omega_k}, we deduce that if $L_1$ is sufficiently large, then $\omega>1/2$.
As in Remark~\ref{rem:r}, since $a_2$ depends only on $A$, $d$ and $p$, we add this constraint on $L_1$ to how we set it in Section~\ref{sec:setL_1}.
\end{remark}

We now introduce the concept of the \emph{confinement} of a cluster, which will allow us to control the spread of $FPP_{\lambda}$ from clusters in a good box at high scales. In particular, the confinement of clusters in a good box is a crucial ingredient in proving that \eqref{D_k^C} holds for $k\geq2$. For this we will use the notion of inner clusters that was introduced in Definition \ref{def:boundary_cluster}.

\begin{definition}[Successful confinement of inner clusters]
Let $k\geq2$, $Q_k(i)$ be good and $\mathcal{A}$ be an inner cluster in $Q_k(i)$. If there exists a $(k-1)$-box $Q_{k-1}(j)$ with negative feedback such that 
$$
Q_{k-1}(j)\in\partial\mathcal{A} \quad\text{and}\quad Q_{k-1}(j^{\text{prog}})\in\mathcal{A},
$$
then we define $\mathcal{A}$ to be \emph{poorly confined}. If $\mathcal{A}$ is not poorly confined, we define it to be \emph{successfully confined}.
\end{definition}

Note that, because of the progenitor structure (that is, if property ($\text{Prog}_{k-1}$) holds), then we deduce that the existence of any 
$(k-1)$-box $Q_{k-1}(\iota)\not\in \mathcal{A}$ of negative feedback (even with $Q_{k-1}(\iota)\not\in \partial \mathcal{A}$)
for which $Q_{k-1}(\iota^{\text{prog}})\in\mathcal{A}$ implies that $\mathcal{A}$ is poorly confined.

\begin{remark}
   \label{rem:geo_interp}
   Let $k\geq2$ and suppose $Q_k(i)$ has positive feedback. If \eqref{cask} holds then we deduce that if $\mathcal{A}$ is an inner cluster of $Q_k(i)$ that has a non-empty intersection with $Q_k^{\text{inn}}(i)$ 
   then it must be successfully confined as all $(k-1)$-boxes in $\partial\mathcal{A}$ have positive feedback as $\partial\mathcal{A}\subset\mathcal{W}^{\text{inn}}(k,i)$.
   But our proof will proceed in the opposite direction, that is, we will use successful confinement of clusters inside a $k$-box to establish~\eqref{cask}.
\end{remark}

In high-level terms, if a cluster $\mathcal{A}$ is successfully confined, then $\mathcal{A}$ contains all $(k-1)$-boxes of negative feedback that have a progenitor in $\mathcal{A}$. In some sense, $\mathcal{A}$ being successfully confined ``contains'' the effect of the bad boxes in $\mathcal{A}$. 
A key point of our analysis is understanding the conditions under which an inner cluster can be poorly confined.

We now provide a different perspective of poorly confined inner clusters. 
The idea is that any poorly confined inner cluster must have a box of negative feedback in its boundary with a progenitor in \emph{another} cluster. 
This will allow us to show that a poorly confined cluster must cause a sequence of poorly confined clusters; so a poorly confined cluster inside $k$-box 
will have an effect that is in some sense proportional to its distance to the boundary of the $k$-box. 

We define the \emph{source} of a poorly confined inner cluster as the box of negative feedback in the outer-boundary with earliest entrance time (regardless of where its progenitor is located). 
The source serves as a ``witness" to the poor confinement of the cluster. 

\begin{definition}[Source of a poorly confined inner cluster]
\label{def:source}
Let $k\geq1$ and $\mathcal{A}$ be a poorly confined inner cluster of a good $(k+1)$-box. Set $Q_{k}(s)$ be the $k$-box in $\partial\mathcal{A}$ with negative feedback that has the smallest entrance time, so that
$$
\tau_{k,s} = \inf\{\tau_{k,h}: Q_{k}(h)\in\partial\mathcal{A}\text{ and }Q_{k}(h)\text{ has negative feedback}\},
$$
noting that this infimum is over a non-empty set as $\mathcal{A}$ is poorly confined.
\end{definition}

As we will see later, it is desirable that the progenitor of the source of a poorly confined inner cluster is not contained in that cluster and is entered relatively soon after the outer boundary of the cluster is entered for the first time. These properties are given below:
\begin{equation}
\label{confk}
\begin{array}{l}
\text{If $Q_{k+1}(i)$ is good and $\mathcal{A}$ is a poorly confined inner cluster in}\\
\text{$Q_{k+1}(i)$ with source $Q_{k}(s)\in\partial\mathcal{A}$, then $Q_{k}(s^{\text{prog}})\notin\mathcal{A}$}\text{ and} \tag{$\text{Conf}_k$}\\
\tau_{k,s} < \tau_{k,b} + 2\sigma r_kL_k,\text{ where }\tau_{k,b} = \inf\{\tau_{k,b'}:Q_k(b')\in\partial\mathcal{A}\}.
\end{array}
\end{equation}

To prove that the fundamental properties hold at higher scales for an appropriate choice of $\lambda$ it will be useful to introduce the property \eqref{fund_property} at scale $k$, that states for a constant $C>0$ all the fundamental properties hold at all scales up to and including scale $k$, as given below:

\begin{equation}
\label{fund_property}
\begin{array}{l}
\text{For $C>0$, the properties ($\text{CAS}_j$), ($\text{Prog}_j$), ($\text{Fast}_j$), ($\text{Del}_j^C$) } \\
\text{and ($\text{Conf}_j$) hold for all $j\in \{1,2,\ldots,k\}$.} \tag{$\text{P}_k^C$}
\end{array}
\end{equation}

The reason we introduce \eqref{fund_property} is that to prove some fundamental property at higher scales, we will need to use other fundamental properties at that scale and lower. 

Recall that in Lemma~\ref{lem:cascading} we wish to set $C$ large enough so that if $\lambda\in(0,\lambda_C)$, then we may deduce that \eqref{cask} holds for all $k\geq1$. The issue is that in choosing $C$ large enough so that if $\lambda\in(0,\lambda_C)$ then \eqref{cask} holds, it is not immediate that $(\text{CAS}_{k+1})$ holds too. It transpires that some of the fundamental properties require $C$ to be sufficiently large to deduce that they hold at higher scales. 
The key point is that we may choose $C$ in such a way that does not depend on $k$ or $L_1$. In fact, we now define a threshold that $C$ needs to be above to allow for the first half of the inductive argument for Lemma~\ref{lem:cascading}. Define
\begin{equation}
\label{def:C_0}
C_0 = 8\sigma r_1a_3,
\end{equation}
where $a_3>1$ is a constant that only depends on $A$ and $d$ that we set later in the proof of Claim~\ref{claim:delk}. 
	The reason for this choice of $C_0$ will only become apparent in later proofs but the reader should note that $C_0$ only depends on $A$, $d$ and $p$. 
	
	The next lemma is the second main result of this section, essentially giving that \eqref{cask} and a proper choice of $\lambda$ 
imply the fundamental properties at scale $k$.
\begin{lemma}
\label{lem:induction1}
Let $k\geq2$ and $C>C_0$. If $\lambda$ is such that $(\emph{P}_{k-1}^C)$ and \eqref{cask} hold, then \eqref{fund_property} holds.
\end{lemma}
Lemma~\ref{lem:induction1} provides a key step in the inductive argument in the proof of Lemma~\ref{lem:cascading}. One subtlety is that we have not yet proved that $(\text{Conf}_1)$ holds but the proof at higher scales transpires to be completely analogous. To streamline later proofs, we now prove that \eqref{confk} holds assuming other fundamental properties at that scale hold. 

\begin{lemma}
\label{lem:source_ET}
Let $k\geq1$ and $C>C_0$. If $\lambda$ is such that \eqref{progk}, \eqref{F_k} and \eqref{D_k^C} hold, then \eqref{confk} holds.
\end{lemma}

\begin{proof}
Let $\mathcal{A}$ be a poorly confined inner cluster of a good $(k+1)$-box with source $Q_k(s)$. 
Let $Q_k(b)$ be the $k$-box in $\partial\mathcal{A}$ with the smallest entrance time, so that $\tau_{k,b} = \inf\{\tau_{k,b'}:Q_k(b')\in\partial\mathcal{A}\}$. As $\lambda$ is such that \eqref{F_k} holds, then 
\begin{equation}
\label{pf:source_ET_1}
\tau_{k,s} < \tau_{k,b} + 2\sigma r_kL_k,
\end{equation}
as there must exist a $k$-box with negative feedback within $\partial\mathcal{A}$ by time $\tau_{k,b} + 2\sigma r_kL_k$. If there did not exist a $k$-box with negative feedback within $\partial\mathcal{A}$ by time $\tau_{k,b} + 2\sigma r_kL_k$, then every $k$-box in $\partial\mathcal{A}$ would have positive feedback as \eqref{F_k} holds and the number of $k$-boxes in $\partial\mathcal{A}$ is bounded above by $\sigma$. This would give a contradiction as $\mathcal{A}$ is assumed to be poorly confined.

It remains to prove that $Q_k(s^{\text{prog}})\notin\mathcal{A}$. Suppose for contradiction that $Q_k(s^{\text{prog}})\in\mathcal{A}$ and set $Q_k(j)$ to be the bad $k$-box in $\mathcal{A}$ with earliest entrance time, so that
$$
\tau_{k,j} = \inf\{\tau_{k,j'}:Q_k(j')\in\mathcal{A}\text{ and } Q_{k}(j')\text{ is bad}\}.
$$
Clearly $\tau_{k,j}\leq\tau_{k,s^{\text{prog}}}$ as we assume that $Q_k(s^{\text{prog}})$ is a bad $k$-box in $\mathcal{A}$.
By construction we have
$$
\tau_{k,b}<\tau_{k,j}\leq\tau_{k,s^{\text{prog}}}<\tau_{k,s}.
$$
Moreover, as $\lambda$ is such that \eqref{progk} and \eqref{D_k^C} hold, we have
\begin{equation}
\label{pf:source_ET_2}
\tau_{k,s} > \tau_{k,s^{\text{prog}}} + C\omega_kL_k > \tau_{k,b} + C\omega_kL_k.
\end{equation}
If $C$ is large enough so 
\begin{equation}
\label{eq:C_cond_source_ET}
C>\tfrac{2\sigma r_k}{\omega_k},
\end{equation}
then \eqref{pf:source_ET_2} contradicts \eqref{pf:source_ET_1} and thus $Q_k(s^{\text{prog}})\notin\mathcal{A}$. By Remark~\ref{rem:r} and Remark~\ref{rem:omega}, for a large enough choice of $L_1$ we have that $r<2r_1$ and $\omega>1/2$ and so if $C>8\sigma r_1$ then \eqref{eq:C_cond_source_ET} holds for all $k\geq1$. As this follows immediately from the fact that $C>C_0$, the result is established.
\end{proof}

Before turning to the proof of Lemma~\ref{lem:cascading}, we must outline an important consequence that can be deduced from the fundamental properties that provides a lower bound on the entrance time for wonderful boxes that have negative feedback in a good box. The idea is that the poor confinement of an inner cluster in a good box can be traced through the parent structure until reaching the boundary or a boundary cluster of that box. By carefully considering the progenitor structure given by \eqref{progk} and also \eqref{confk}, we deduce a lower bound on entrance times by the number of boxes of negative feedback on this parent structure. We make this idea rigorous below.

Assume $C>C_0$ and $\lambda$ is such that \eqref{fund_property} holds. Let $Q_{k+1}(i)$ be good and $\mathcal{A}_1$ be a poorly confined inner cluster in $Q_{k+1}(i)$ with source $Q_k(s_1)$ that is the witness of the poor confinement of $\mathcal{A}_1$. As $\lambda$ is such that \eqref{confk} holds, we have that $Q_k(s_1^{\text{prog}})$ is not contained in $\mathcal{A}_1$. If $Q_k(s_1^{\text{prog}})$ is contained in another inner cluster of $Q_{k+1}(i)$, say $\mathcal{A}_2$, then $\mathcal{A}_2$ is also poorly confined as $\lambda$ is such that \eqref{progk} holds. The inner cluster $\mathcal{A}_2$ has source $Q_k(s_2)$ say. From this point onwards, we can search for the progenitor of $Q_k(s_2)$ and iterate this procedure until the progenitor of a source is either outside of $Q_{k+1}(i)$ or is contained in an boundary cluster in $Q_{k+1}(i)$.  This procedure is illustrated in Figure~\ref{fig:pathjumps} in the case where it terminates when the progenitor of a source is not contained in $Q_{k+1}(i)$.

\begin{figure}[!htb]
	\center{\includegraphics[scale=0.55]{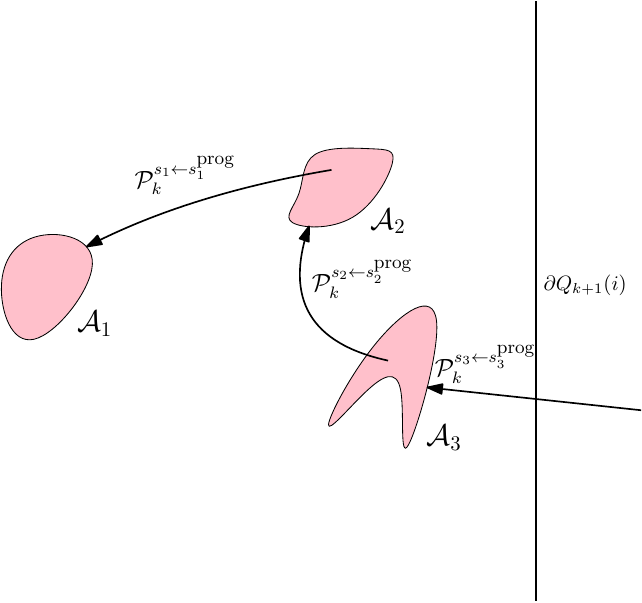}}
        \caption{An illustration of the path of jumps construction. The pink regions represents poorly confined inner clusters in $Q_{k+1}(i)$. 
        The black arrows represent the path of parents from the sources of the poorly confined inner clusters to their progenitor.
        The vertical black line represents a piece of the boundary of $Q_{k+1}(i)$, whose interior is to the left of the vertical line.}
        \label{fig:pathjumps} 
\end{figure}

From this construction we discover a sequence of poorly confined inner clusters $\mathcal{A}_1 \mathcal{A}_2,\ldots,\mathcal{A}_n$ with respective sources $Q_k(s_1),Q_k(s_2),\ldots,Q_k(s_n)$ where $Q_k(s_n^{\text{prog}})$ is contained outside of $Q_{k+1}(i)$ or is contained in a boundary cluster of $Q_{k+1}(i)$. Define the collection of paths of parents
$$
\mathcal{P}_k^{s_1\leftarrow s_1^{\text{prog}}}, \mathcal{P}_k^{s_2\leftarrow s_2^{\text{prog}}},\ldots,\mathcal{P}_k^{s_n\leftarrow s_n^{\text{prog}}}, 
$$
as a \emph{path of jumps}; we refer to a \emph{jump} when we concatenate $s_\iota^{\text{prog}}$ with $s_{\iota+1}$ for each $\iota\in\{1,2,\ldots,n-1\}$.
Hence, for any poorly confined inner cluster in a good box we can construct a path of jumps through the progenitor structure that terminates when it discovers a progenitor not contained in an inner cluster. Utilising the path of jumps construction and \eqref{D_k^C} we may deduce a useful lower bound on the entrance for wonderful boxes of negative feedback in a good box as given in Lemma~\ref{lem:NF_control}.

Before we state Lemma~\ref{lem:NF_control} we introduce the following notation. For $B,B'\subset \mathbb{Z}^d$, we define
$$
\text{dist}(B,B')=\inf_{x\in B,y\in B'}\|x-y\|_1. 
$$
If $Q_k(i)$ and $Q_k(i')$ are $k$-boxes such that $Q_k(i)\subset B$, $Q_k(i')\subset B'$ and the path of parents $\mathcal{P}_k^{i\leftarrow i'}$ is well-defined, then it is easy to verify that
$$
\text{the number of $k$-boxes in }\mathcal{P}_k^{i\leftarrow i'}\geq\tfrac{3}{L_{k}}\text{dist}(B,B').
$$

\begin{lemma}
\label{lem:NF_control}
Let $k\geq1$ and $C>C_0$. Suppose $\lambda$ is such that \eqref{fund_property} holds. If $Q_{k+1}(i)$ is good and $Q_{k}(j_0)\in\mathcal{W}(k+1,i)$ has negative feedback, then
\begin{align}
\tau_{k,j_0} > \tau_{k+1,i} + C\omega_{k}L_{k}\left(\tfrac{3}{L_{k}}\emph{dist}(Q_{k}(j_0), \partial Q_{k+1}(i))-a_4\right), \label{eq:NF_control}
\end{align}
where $a_4>0$ is a constant that only depends on $A$ and $d$.
\end{lemma}

\begin{proof}
Suppose $Q_{k+1}(i)$ is good and $Q_{k}(j_0)\in\mathcal{W}(k+1,i)$ has negative feedback. First consider the case where  $Q_{k}(j_0^{\text{prog}})$ is contained in an inner cluster of $Q_{k+1}(i)$, say $\mathcal{A}_1$. As $\lambda$ is such that \eqref{fund_property} holds, $\mathcal{A}_1$ must be poorly confined and so by the path of jumps construction above, there exists a sequence of poorly confined inner clusters in $Q_{k+1}(i)$ enumerated as $\mathcal{A}_1,\mathcal{A}_2,\ldots,\mathcal{A}_n$
with respective sources $Q_{k}(s_1),Q_{k}(s_2),\ldots,Q_{k}(s_n)$ so that the following properties hold:
\begin{align}
Q_{k}(j_0^{\text{prog}})\in\mathcal{A}_1,\nonumber\\
Q_{k}(s_n^{\text{prog}}) \text{ is not contained in an inner cluster in }Q_{k+1}(i),\nonumber\\
Q_{k}(s_m^{prog})\in\mathcal{A}_{m+1}\quad \text{for }1\leq m<n.\nonumber
\end{align}
By definition of source, no cluster is repeated in this construction and $n\leq A$ as there are at most $A$ clusters in a good $k$-box.

From the path of jumps we outline the following procedure to deduce a lower bound on the entrance time for $Q_{k}(j_0)$. 
Start following the path of parents from $Q_{k}(j_0)$ to $Q_{k}(j_0^{\text{prog}})$ and stop once a $k$-box in $\partial \mathcal{A}_1$ is discovered, 
say $Q_{k}(j_0^*)$ which must have negative feedback as $\lambda$ is such that \eqref{progk} holds. By definition of source and Lemma~\ref{lem:source_ET}, we have that $\tau_{k,s_1}<\tau_{k,j_0^*}$. 
Then jump to $Q_{k}(s_1)$ and follow its path of parents to $Q_{k}(s_1^{\text{prog}})$, stopping once a $k$-box in $\partial \mathcal{A}_2$ is discovered, 
say $Q_{k}(s_1^*)$ which again must have negative feedback with $\tau_{k,s_2}<\tau_{k,s_1^*}$. Next, jump to $Q_{k}(s_2)$ and proceed in this way until we reach $Q_{k}(s_{n})$, 
from which we follow its path of parents until hitting the boundary of $Q_{k+1}(i)$ or $\partial\mathcal{A}_n$, which in this latter case $\mathcal{A}_n$ must be a boundary cluster of $Q_{k+1}(i)$ 
and we set $Q_{k-1}(s_{n}^*)$ as before. In the former case we set $Q_k(s_{n}^*)$ to be the first $k$-box on the path of parents starting from $Q_k(s_{n})$ whose parent is not contained in $Q_{k+1}(i)$. 
In either case it is clear that $Q_{k+1}(i)$ is entered before $Q_k(s_n^*)$. Through this construction we deduce
\begin{align}
\tau_{k,s_1}<\tau_{k,j_0^*},\quad \tau_{k+1,i}\leq\tau_{k,s_{n}^*} \quad\text{and}\quad \tau_{k,s_{m+1}}<\tau_{k,s_{m}^*} \text{ for all }1\leq m< n. \label{eq:jump_times}
\end{align}
Note that in the construction above there are at most $A$ jumps, since each jump occurs in a bad cluster in $Q_{k+1}(i)$ and clusters cannot be repeated due to the ordering of the entrance times of the sources as given in \eqref{eq:jump_times}. Note also that each jump has length at most $\sigma$, since the jump is between two $k$-boxes contained in the outer-boundary of the same cluster. Hence the number of $k$-boxes on the paths constructed above is bounded below by 
\begin{align}
\tfrac{3}{L_{k}}\text{dist}(Q_{k}(j_0),\partial Q_{k+1}(i)) - a, \label{pf:NF_control1}
\end{align}
where $a>0$ is a constant that depends only on $A$ and $d$.

Every $k$-box on these paths is either bad or has negative feedback as $\lambda$ is such that \eqref{progk} holds and the number of such bad boxes is bounded above by a constant that only depends on $A$ and $d$. For every pair of boxes with negative feedback on these paths where one is the parent of the other, we deduce that there is a delay of at least $C\omega_{k}L_{k}$ in the respective entrance times of the boxes as $\lambda$ is such that \eqref{D_k^C} holds. From \eqref{pf:NF_control1}, we deduce the number of such pairs is bounded below by
\begin{align}
\tfrac{3}{L_{k}}\text{dist}(Q_{k}(j_0),\partial Q_{k+1}(i)) - a', \label{pf:NF_control2}
\end{align}
where $a'\geq a$ is a constant that depends only on $A$ and $d$.

Combining \eqref{pf:NF_control2} and \eqref{eq:jump_times} with the delay of at least $C\omega_kL_k$ between boxes with negative feedback that have a negative feedback parent, we deduce
\begin{align*}
\tau_{k,j_0} & > \tau_{k,s_{n}^*} + C\omega_{k}L_{k}\left(\tfrac{3}{L_{k}}\text{dist}(Q_{k}(j_0),\partial Q_{k+1}(i)) - a'\right),\\
& \geq\tau_{k+1,i} + C\omega_{k}L_{k}\left(\tfrac{3}{L_{k}}\text{dist}(Q_{k}(j_0),\partial Q_{k+1}(i)) - a'\right).
\end{align*}
Now consider the case where $Q_k(j_0^{\text{prog}})$ is either not contained in $Q_{k+1}(i)$ or is contained in a boundary cluster of $Q_{k+1}(i)$, $\mathcal{B}$ say. 
In this scenario set $Q_{k}(j_0^*)$ to be the first $k$-box on the path of parents $\mathcal{P}_k^{j_0\leftarrow o}$ whose 
parent is not contained in $Q_{k+1}(i)$ or the first box that is contained in $\mathcal{B}$. 
In this case, the result follow by an analogous argument to the previous case where we instead consider the path of parents $\mathcal{P}_k^{j_0\leftarrow j_0^*}$, noting that no jumps need be considered. The lemma follows by setting $a_4=a'$.
\end{proof}

To conclude the analysis of the fundamental properties, we are left with the task of proving Lemmas~\ref{lem:cascading} and~\ref{lem:induction1}. This will be done in the next two sections.

\subsection{Proof of cascading lemma (Lemma~\ref{lem:cascading})}
\label{sec:cascading_lemma_proof}

In this section, we give the proof for Lemma~\ref{lem:cascading} that provides the cascading argument required in the proof of Theorem~\ref{maintheorem}. 
For this proof we assume that Lemma~\ref{lem:induction1} holds, which will be established in the next section.
The main step of the proof is given in the following claim, from which the proof of Lemma~\ref{lem:cascading} readily follows.

\begin{claim}
\label{claim:induction}
There exists a constant $C_1>0$ that only depends on $p$ and $d$ such that the following holds. For all $k\geq1$ and $C>\max\{C_0,C_1\}$, if $\lambda$ is such that \eqref{fund_property} holds, then $(\emph{CAS}_{k+1})$ holds.
\end{claim}

\begin{proof}[Proof of Lemma~\ref{lem:cascading}]
   Assuming Claim~\ref{claim:induction} holds, let $C^*=\max\{r_1,C_0,C_1\}$, $C>C^*$ and $\lambda\in(0,\lambda_{C})$. 
   Then Lemma~\ref{lem:main_lemma_1}, Lemma~\ref{lem:prog_1}, Lemma~\ref{lem:fast_1}, Lemma~\ref{lem:del_1} and Lemma~\ref{lem:source_ET} give us that $(\text{P}_1^C)$ holds. 
   Claim~\ref{claim:induction} immediately gives us that $(\text{CAS}_2)$ also holds. Lemma~\ref{lem:induction1} then gives us for this choice of $C$ and $\lambda$ that $(\text{P}_2^C)$ also holds. 
   This choice of $C$ and $\lambda$ does not depend on $k$ and so by repeatedly using Claim~\ref{claim:induction} and Lemma~\ref{lem:induction1}, we see that if $C>C^*$ and $\lambda\in(0,\lambda_C)$, 
   then \eqref{cask} holds for all $k\geq1$.
\end{proof}

\begin{proof}[Proof of Claim \ref{claim:induction}]
   Let $k\geq1$ and $C>C_0$. Assume that $\lambda$ is such that \eqref{fund_property} holds and $Q_{k+1}(i)$ has positive feedback. Suppose for contradiction that there exists some $Q_{k}(j_0)\in\mathcal{W}^{\text{inn}}(k+1,i)$ that has negative feedback. Without losing generality, we may assume that $Q_{k}(j_0)$ is the $k$-box in $\mathcal{W}^{\text{inn}}(k+1,i)$ of negative feedback with the smallest entrance time, so that
   $$
   \tau_{k,j_0} = \inf\{\tau_{k,h}:Q_{k}(h)\in\mathcal{W}^{\text{inn}}(k+1,i) \text{ and } Q_{k}(h) \text{ has negative feedback}\}.
   $$
   As we assume $Q_{k+1}(i)$ has positive feedback, there exists $Q_{k}(i_0)\in\mathcal{W}^{\text{inn}}(k+1,i)$ such that 
   $$
   \tau_{k,i_0} < \tau_{k+1,i}+\tfrac{3}{500} r_{k+1}L_{k+1}.
   $$
   By Remark~\ref{rem:completion}, we may fix a path of $k$-boxes $\Xi_{k}$ from $Q_{k}(i_0)$ to $Q_{k}(j_0)$ so that $\Xi_{k}\subset\mathcal{W}^{\text{inn}}(k+1,i)$ and
   $$
   |\Xi_{k}|\leq \tfrac{3dL_{k+1}}{L_{k}} + A\sigma.
   $$
   If all $k$-boxes on $\Xi_{k}$ have positive feedback with the exception of $Q_{k}(j_0)$, then \eqref{F_k} implies that 
   \begin{align*}
   \tau_{k,j_0} & < \tau_{k,i_0} + 2r_{k}L_{k}|\Xi_{k}|,\\
   &\leq \tau_{k,i_0} + 2r_{k}L_{k}\left(\tfrac{3dL_{k+1}}{L_{k}} + A\sigma\right).
   \end{align*}
   Recalling that $Q_{k+1}(i)$ has positive feedback and $\tau_{k,i_0}<\tau_{k+1,i}+\tfrac{3}{500} r_{k+1}L_{k+1}$, from \eqref{def:r_k} and above we deduce that
   \begin{align}
   \tau_{k,j_0} & < \tau_{k+1,i} + \tfrac{3}{500} r_{k+1}L_{k+1} +  6dr_{k}L_{k+1}\left(1+\tfrac{A\sigma}{3d(k+1)^2L_{k}^{d-1}}\right), \nonumber\\
   & < \tau_{k+1,i} + \left(\tfrac{3}{500} + 6d\right)r_{k+1}L_{k+1}. \label{pf:contra_point_1}
   \end{align}
   With an upper bound for $\tau_{k,j_0}$ established in \eqref{pf:contra_point_1}, our aim is to establish a corresponding lower bound for $\tau_{k,j_0}$ that will give a contradiction.
   
   By definition of $\mathcal{W}^{\text{inn}}(k+1,i)$, if $Q_{k}(q)\in\Xi_{k}$, then 
   \begin{align}
   \tfrac{3}{L_{k}}\text{dist}(Q_{k}(q),\partial Q_{k+1}(i)) \geq \tfrac{3L_{k+1}}{1000L_{k}} -\sigma. \label{pf:k_path}
   \end{align}
   By \eqref{pf:k_path}, \eqref{D_k^C} and Lemma~\ref{lem:NF_control}, for all $Q_{k}(q)\in\Xi_{k}$, we have
   \begin{align}
   \tau_{k,q} > \tau_{k+1,i} + C\omega_{k}L_{k}\left(\tfrac{3L_{k+1}}{1000L_{k}} - a_5\right),
   \end{align}
   where $a_5>0$ is a constant that only depends on $A$ and $d$. From the above equation we deduce
   \begin{align}
   \tau_{k,q} & = \tau_{k+1,i} + \tfrac{3}{1000}C\omega_{k}L_{k+1}\left(1 - \tfrac{1000a_5}{3(k+1)^2L_{k}^{d-1}}\right)\nonumber\\
   & > \tau_{k+1,i} + \tfrac{1}{1000} C\omega_{k+1}L_{k+1}, \label{pf:contra_point_2}
   \end{align}
   where the final inequality will hold so long as $L_1$ is sufficiently large and by recalling the definition of $\omega_{k+1}$ in \eqref{def:omega_k}. This requirement on $L_1$ only depends on $a_5$ and thus can be added to the conditions on setting $L_1$ in Section~\ref{sec:setL_1}.
   
   For $k\geq1$, set
   \begin{align*}
   \tilde{C}_{k+1} = 6\left(1+1000d\right)\tfrac{r_{k+1}}{\omega_{k+1}}.
   \end{align*}
   If $C>\max\{C_0,\tilde{C}_{k+1}\}$ then all $k$-boxes on $\Xi_{k}$ necessarily are positive feedback as otherwise \eqref{pf:contra_point_1} contradicts \eqref{pf:contra_point_2}. However this leads to a contradiction, as then $Q_{k-1}(j_0)$ must also have positive feedback. Hence if $C>\max\{C_0,\tilde{C}_{k+1}\}$ and $\lambda$ is such that \eqref{fund_property} holds, then $(\text{CAS}_{k+1})$ also holds. 
   
   To remove the dependence on $k$ in setting $\lambda$, we desire that 
   \begin{align}
   C > 6\left(1+1000d\right)\tfrac{r}{\omega},  \label{eq:pick_C}
   \end{align}
   where $r$ is as given in Remark~\ref{rem:r} and $\omega$ as in Remark~\ref{rem:omega}. 
   Recall that so long as $L_1$ is sufficiently large, we have that $r<2r_1$ and $\omega>1/2$. Thus, if we set
   $$
   C_1 = 24\left(1+1000d\right)r_1,
   $$
   and let $C>C_1$, then \eqref{eq:pick_C} holds. Consequently, if $C>C_1$, then $C>\tilde{C}_{k+1}$ for all $k\geq1$, establishing the claim.
\end{proof}

\subsection{Proofs of fundamental properties at higher scales (Lemma~\ref{lem:induction1})}
\label{sec:pf_lemmmas_higher_scales}

To ease exposition in the proof for Lemma~\ref{lem:induction1}, we split the proof into four parts, one part for each fundamental property to be established.
We also note that for this proof we shall not assume Lemma~\ref{lem:cascading}, as the proof of Lemma~\ref{lem:cascading} uses Lemma~\ref{lem:induction1}.

\begin{proof}[Proof of Lemma~\ref{lem:induction1} for \eqref{progk}]
   Let $k\geq2$ and $C>C_0$. Assume that $\lambda$ is such that $(P_{k-1}^C)$ and \eqref{cask} hold. 
   Suppose that $Q_k(i)$ has negative feedback. Hence, if $Q_{k-1}(\ell)\in\mathcal{W}^{\text{inn}}(k,i)$ has positive feedback, then
   \begin{align}
   \tau_{k-1,\ell} \geq \tau_{k,i} + \tfrac{3}{500} r_kL_k. \label{eq:ell_entrance}
   \end{align}
   Suppose for contradiction that $Q_k(i^{\text{par}})$ has positive feedback. The strategy of this proof is to construct a path of wonderful $(k-1)$-boxes from where $FPP_1$ enters $Q_k(i)$ to $\mathcal{W}^{\text{inn}}(k,i)$ and use the fact that all $(k-1)$-boxes on this path must have positive feedback by the positive feedback of $Q_k(i^{\text{par}})$ and \eqref{cask}. This will imply that $Q_k(i)$ has positive feedback, leading to the desired contradiction. To begin, we make the following claim that we prove later.
   
   \begin{claim}
   \label{claim:progk}
   Suppose the assumptions above hold. There exists $(k-1)$-boxes $Q_{k-1}(i_0^*)$ and $Q_{k-1}(i_1^*)$ such that 
   \begin{enumerate}
   \item $Q_{k-1}(i_0^*)\in\mathcal{W}^{\emph{inn}}(k,i^\emph{par})$,
   \item $\tau_{k-1,i_0^*} \leq \tau_{k,i}$,
   \item $Q_{k-1}(i_1^*) \in \mathcal{W}^{\emph{inn}}(k,i) \cap \mathcal{W}^{\emph{inn}}(k,i^{\emph{par}})$,
   \end{enumerate}
   and there exists a path of $(k-1)$-boxes $\pi_{k-1}$ from $Q_{k-1}(i_0^*)$ to $Q_{k-1}(i_1^*)$ such that
   \begin{align}
   \pi_{k-1}\subset\mathcal{W}^{\emph{inn}}(k,i^{\emph{par}})\quad\text{and}\quad |\pi_{k-1}|\leq \tfrac{3L_k}{1000L_{k-1}} + A\sigma. \label{eq:shortpath}
   \end{align}
   \end{claim}
   
   Assuming Claim~\ref{claim:progk} holds and $Q_k(i^{\text{par}})$ has positive feedback, we deduce that all $(k-1)$-boxes on $\pi_{k-1}$ have positive feedback as $\lambda$ is such that \eqref{cask} holds. Moreover, as $\lambda$ is such that $(\text{P}_{k-1}^C)$ holds, then $(\text{Fast}_{k-1})$ holds. Thus
   \begin{align*}
   \tau_{k-1,i^*_1} <  \tau_{k-1,i_0^*} + 2r_{k-1}L_{k-1}|\pi_{k-1}|.
   \end{align*}
   By \eqref{eq:shortpath} we deduce
   \begin{align*}
   \tau_{k-1,i^*_1} & < \tau_{k-1,i_0^*} + 2r_{k-1}L_{k-1}\left(\tfrac{3L_k}{1000L_{k-1}} + A\sigma\right)\leq\tau_{k,i} + 2r_{k-1}L_{k-1}\left(\tfrac{3L_k}{1000L_{k-1}} + A\sigma\right).
   \end{align*}
   Recalling the definition of $r_k$ in \eqref{def:r_k} where  $a_1=\tfrac{1000A\sigma}{3}$, we have
   \begin{align}
   \tau_{k-1,i^*_1} &< \tau_{k,i} + \tfrac{3}{500}r_{k-1}L_k\left(1+ \tfrac{a_1}{k^2L_{k-1}^{d-1}}\right)\nonumber\\
   & = \tau_{k,i} + \tfrac{3}{500}r_kL_k. \label{eq:progk_contra}
   \end{align}
   As $Q_{k-1}(i_1^*)\in\mathcal{W}^{\text{inn}}(k,i)$, then \eqref{eq:progk_contra} contradicts \eqref{eq:ell_entrance}. Hence it must be the case that $Q_k(i^{\text{par}})$ is either bad or has negative feedback.
\end{proof}

\begin{proof}[Proof of Claim~\ref{claim:progk}]
   Let $Q_{k-1}(i_{\text{ent}})$ be such that $e_{k,i}\in Q_{k-1}^{\text{core}}(i_{\text{ent}})$. If $Q_{k-1}(i_{\text{ent}})\in\mathcal{W}^{\text{inn}}(k,i^{\text{par}})$ then the first two conditions in Claim~\ref{claim:progk} 
   follow immediately by setting $i_0^*=i_{\text{ent}}$. The other possibility is where $Q_{k-1}(i_{\text{ent}})\in\mathcal{A}$, where $\mathcal{A}$ is an inner cluster in $Q_k(i^{\text{par}})$. 
%   As $\lambda$ is such that \eqref{cask} holds, it must be that all $(k-1)$-boxes in $\partial\mathcal{A}$ have positive feedback by Remark~\ref{rem:geo_interp}. 
   In this case we let $Q_{k-1}(i_0^*)$ be the $(k-1)$-box in $\partial\mathcal{A}$ with the smallest entrance time and the first two conditions of Claim~\ref{claim:progk} follow. 
   Note that $Q_{k-1}(i_{\text{ent}})$ cannot be contained in a boundary cluster of $Q_k(i^{\text{par}})$ as by definition of parent, we have $Q_{k-1}^\text{core}(i_{\text{ent}})\subset Q_k^{\text{core}}(i^{\text{par}})$.
   
   The final part of the claim follows by setting $\pi_{k-1}$ as the shortest path of $(k-1)$-boxes from $Q_{k-1}(i_0^*)$ to $\mathcal{W}^{\text{inn}}(k,i)$ that is completely contained in $\mathcal{W}^{\text{inn}}(k,i^{\text{par}})$. Since the distance between $\partial Q_k(i)$ and $\mathcal{W}^{\text{inn}}(k,i)$ is at most $\tfrac{3L_k}{1000}$, by further accounting for possible deviations around clusters of $Q_k(i)$ through their respective outer boundaries, we obtain the desired conditions.
\end{proof}
   
\begin{proof}[Proof of Lemma~\ref{lem:induction1} for \eqref{F_k}]
   Let $k\geq2$ and $C>C_0$. Assume that $\lambda$ is such that $(P_{k-1}^C)$ and \eqref{cask} hold. Suppose $Q_k(i)$ is good and $Q_k(j)$ has positive feedback so that $\|i-j\|_{\infty}=1$. As $Q_k(j)$ has positive feedback, there exists a $(k-1)$-box $Q_{k-1}(j_0)\in\mathcal{W}^{\text{inn}}(k,j)$ that has positive feedback such that 
   $$
   \tau_{k-1,j_0} < \tau_{k,j}+\tfrac{3}{500} r_kL_k.
   $$
   Fix a $(k-1)$-box $Q_{k-1}(j_1)\in\mathcal{W}^{\text{inn}}(k,j)\cap\mathcal{W}(k,i)$ such that there is a path $\pi_{k-1}$ from $Q_{k-1}(j_0)$ to $Q_{k-1}(j_1)$ with
   $$
   \pi_{k-1}\subset\mathcal{W}^{\text{inn}}(k,j)\quad\text{and}\quad|\pi_{k-1}| \leq (1-\tfrac{3}{1000})\tfrac{L_k}{L_{k-1}} + A\sigma.
   $$
   As $\lambda$ is such that $(\text{P}_{k-1}^C)$ holds, then by $(\text{Fast}_{k-1})$ we deduce that
   \begin{align*}
   \tau_{k,i} \leq \tau_{k-1,j_1} & < \tau_{k-1, j_0} + 2r_{k-1}L_{k-1}\left((1-\tfrac{3}{1000})\tfrac{L_k}{L_{k-1}} + A\sigma\right)\\
   & = \tau_{k-1,j_0} + 2r_{k-1}L_{k}\left(1-\tfrac{3}{1000}+\tfrac{A\sigma}{k^2L_{k-1}^{d-1}}\right)\\
   & < \tau_{k,j} + 2r_{k-1}L_{k}\left(1-\tfrac{3}{1000}+\tfrac{A\sigma}{k^2L_{k-1}^{d-1}}\right)+\tfrac{3}{500} r_kL_k.
   \end{align*}
   By taking the final term above inside the bracket, we deduce that
   \begin{align*}
   \tau_{k,i} & < \tau_{k,j} + 2r_{k-1}L_{k}\left(1+\tfrac{3}{1000}\left(\tfrac{r_{k}}{r_{k-1}}-1\right)+\tfrac{A\sigma}{k^2L_{k-1}^{d-1}}\right),\\
   & = \tau_{k,j}+ 2r_{k-1}L_{k}\left(1+\tfrac{2A\sigma}{k^2L_{k-1}^{d-1}}\right)\\
   & < \tau_{k,j} + 2r_{k}L_{k}
   \end{align*}
   wherein the first two lines we recall \eqref{def:r_k} and $a_1=\tfrac{1000A\sigma}{3}$. From the above inequality, we deduce \eqref{F_k} holds.
\end{proof}

\begin{proof}[Proof of Lemma~\ref{lem:induction1} for \eqref{D_k^C}]
   Let $k\geq2$, $C>C_0$ and assume that $\lambda$ is such that $(P_{k-1}^C)$ and \eqref{cask} hold. Suppose that $Q_k(i)$ and $Q_k(i^{\text{par}})$ both have negative feedback.
   
   \begin{claim}
      \label{claim:delk}
      Suppose the assumptions above hold. There exists a $(k-1)$-box $Q_{k-1}(\hat{\imath}_0)\subset Q_k(i^{\emph{par}})$ with negative feedback such that
      \begin{enumerate}
      \item $\emph{dist}(Q_{k-1}(\hat{\imath}_0),\partial Q_k(i^{\emph{par}}))\geq \tfrac{1}{3}\left(L_{k} - a_6 L_{k-1}\right)$,
      \item $\tau_{k-1,\hat{\imath}_0} < \tau_{k,i} + a_7 r_{k-1}L_{k-1},$
      \end{enumerate}
      where $a_6,a_7>0$ are constants that depend only on $A$ and $d$.
   \end{claim}
   Let $Q_{k-1}(\hat{\imath}_0)$ be as given in the claim above. As $\lambda$ is such that $(\text{P}_{k-1}^C)$ holds, then by $(\text{Del}_{k-1}^C)$, Lemma~\ref{lem:NF_control} and Claim~\ref{claim:delk}, we deduce
   \begin{align}
   \tau_{k-1,\hat{\imath}_0} & > \tau_{k,i^{\text{par}}}+ C\omega_{k-1}L_{k-1}\left(\tfrac{3}{L_{k-1}}\text{dist}(Q_{k-1}(\hat{\imath}_0),\partial Q_k(i^{\text{par}})) - a_4\right),\nonumber\\
   & \geq \tau_{k,i^{\text{par}}} + C\omega_{k-1}L_k\left(1 - \tfrac{a_4+a_6}{k^2L_{k-1}^{d-1}}\right). \label{eq:hatent1}
   \end{align}
   Moreover, by Claim~\ref{claim:delk} and \eqref{eq:hatent1} we deduce that
   \begin{align}
   \tau_{k,i} & > \tau_{k-1,\hat{\imath}_0} - a_7 r_{k-1}L_{k-1},\nonumber\\
   & > \tau_{k,i^{\text{par}}} + C\omega_{k-1}L_k\left(1 - \tfrac{a_4+a_6+a_7 r_{k-1}}{k^2L_{k-1}^{d-1}}\right). \label{eq:delay1}
   \end{align}
   If we set
   \begin{align*}
   a_2 = a_4+a_6+2a_7 r_1,
   \end{align*}
   then $a_2>a_4+a_6+a_7 r_{k-1}$ for all $k\geq2$ by Remark~\ref{rem:r} for a sufficiently large choice of $L_1$, and only depends on $A$, $d$ and $p$.
   In particular, by \eqref{eq:delay1} we have
   $$
   \tau_{k,i} > \tau_{k,i^{\text{par}}} + C\omega_{k-1}\left(1-\tfrac{a_2}{k^2L_{k-1}^{d-1}}\right),
   $$
   establishing the result.
\end{proof}

\begin{proof}[Proof of Claim~\ref{claim:delk}]
   Let $Q_{k-1}(i_{\text{ent}})$ be a $(k-1)$-box such that $e_{k,i}\in Q_{k-1}^{\text{core}}(i_{\text{ent}})$. If $Q_{k-1}(i_{\text{ent}})\in\mathcal{W}(k,i^{\text{par}})$ has negative feedback, 
   then set $\hat \imath_0 = i_{\text{ent}}$. Another simple case is when $Q_{k-1}(i_{\text{ent}})$ is contained in some poorly confined inner cluster of $Q_k(i^{\text{par}})$ 
   which must have a negative feedback $(k-1)$-box in its outer-boundary; the one with the smallest entrance time can be chosen to be $Q_{k-1}(\hat\imath_0)$ as $\lambda$ is such that \eqref{confk} holds.
   Note that $Q_{k-1}(i_{\text{ent}})$ cannot be contained in a boundary cluster of $Q_k(i^{\text{par}})$, as by definition of parent we have that $Q_{k-1}^\text{core}(i_{\text{ent}})\subset Q_k^{\text{core}}(i^{\text{par}})$.
   
   The two remaining cases to consider are when $Q_{k-1}(i_{\text{ent}})\in\mathcal{W}(k,i^{\text{par}})$ has positive feedback or is contained in a successfully confined inner cluster of $Q_k(i^{\text{par}})$. The idea of the proof for these cases is to find a $(k-1)$-box $Q_{k-1}(i_0)\in\mathcal{W}(k,i^{\text{par}})$ with positive feedback nearby to $Q_{k-1}(i_{\text{ent}})$ whose entrance time is not much larger than $\tau_{k-1,i_{\text{ent}}}$. From $Q_{k-1}(i_0)$ we fix a path of $(k-1)$-boxes $\pi_{k-1}$ to $\mathcal{W}^{\text{inn}}(k,i)$, so that 
   \begin{equation}
   \label{eq:pi_k-1_cond}
   \pi_{k-1}\subset\mathcal{W}^{\text{core}}(k,i^{\text{par}})\quad\text{and}\quad|\pi_{k-1}|\leq\tfrac{3L_{k}}{1000L_{k-1}}+A\sigma,
   \end{equation}
   where $\mathcal{W}^{\text{core}}(k,i^{\text{par}})$ is the set of wonderful $(k-1)$-boxes in $Q_{k}(i^{\text{par}})$ that are either contained in $Q_{k}^{\text{core}}(i^{\text{par}})$ or in the outer boundary of an inner cluster that intersects $Q_k^{\text{core}}(i^{\text{par}})$. 
   As $\lambda$ is such that $(\text{Fast}_{k-1})$ holds, if all $(k-1)$-boxes on $\pi_{k-1}$ have positive feedback then $Q_k(i)$ must have positive feedback which is a contradiction. Hence there exists some $(k-1)$-box on $\pi_{k-1}$ with negative feedback and we prove that this  box is necessarily nearby to $Q_{k-1}(i_{\text{ent}})$, completing the proof.
   
   Now we implement the idea above rigorously. If $Q_{k-1}(i_{\text{ent}})\in\mathcal{W}(k,i^{\text{par}})$ has positive feedback then we simply set $i_0=i_{\text{ent}}$. 
   Otherwise, if $Q_{k-1}(i_{\text{ent}})$ is contained in a successfully confined inner cluster of $Q_k(i^{\text{par}})$, $\mathcal{A}$ say, then we may assume that all $(k-1)$-boxes within $\partial\mathcal{A}$ have positive feedback as otherwise we may set $\hat{\imath}_0$ to correspond to one of these boxes with negative feedback. Note that $\mathcal{A}$ intersects $Q_k^{\text{core}}(i^{\text{par}})$. Hence we may choose $Q_{k-1}(i_0)$ to be the $(k-1)$-box in $\partial\mathcal{A}$ with earliest entrance time. Through this construction, we have
   \begin{align}
   \tau_{k-1,i_0} \leq \tau_{k,i}. \label{eq:PF_cons}
   \end{align}
   
	Fix a geodesic path of $(k-1)$-boxes from $Q_{k-1}(i_0)$ to $\mathcal{W}_{k}^{\text{inn}}(i)\cap Q_k^{\text{core}}(i^{\text{par}})$, that is, let this path be the shortest path of $(k-1)$-boxes from $Q_{k-1}(i_0)$ to some $(k-1)$-box contained in $\mathcal{W}_{k}^{\text{inn}}(i)\cap Q_k^{\text{core}}(i^{\text{par}})$. 
	Let $\pi_{k-1}$ be the augmentation of this path that deviates through the outer boundary of clusters in $Q_k(i^{\text{par}})$ it intersects, so that $\pi_{k-1}$ ends at some $(k-1)$-box $Q_{k-1}(j)\in\mathcal{W}^{\text{inn}}(k,i)\cap Q_k^{\text{core}}(i^\text{par})$ and $\pi_{k-1}$ satisfies~\eqref{eq:pi_k-1_cond}.
	This construction is possible as $Q_{k-1}(i_0)\in\partial\mathcal{A}\subset\mathcal{W}^{\text{core}}(k,i^{\text{par}})$ and we may always choose the deviations so that such a $Q_{k-1}(j)$ exists.
	
	If $\pi_{k-1}\subset\mathcal{W}(k,i)$, then define $\pi_{k-1}^{\text{end}}=\pi_{k-1}$ and $\pi_{k-1}^{\text{start}}=\emptyset$.
	Otherwise, let $\pi_{k-1}^{\text{end}}$ be the subpath of $\pi_{k-1}$ given by starting from $Q_{k-1}(j)$ and backwards traversing $\pi_{k-1}$ until there is a $(k-1)$-box not contained in $Q_k^{\text{core}}(i^{\text{par}})$, and terminating at the $(k-1)$-box immediately before this. Hence, if $Q_{k-1}\in\pi_{k-1}^{\text{end}}$, then $Q_{k-1}\subset Q_k^{\text{core}}(i^{\text{par}})$ and $Q_{k-1}\in\mathcal{W}(k,i)$. Let $\pi_{k-1}^{\text{start}}=\pi_{k-1}\setminus\pi_{k-1}^{\text{end}}$ and note that the number of $(k-1)$-boxes in $\pi_{k-1}^{\text{start}}$ is bounded above by a constant that depends only on $A$ and $d$. This is because in the construction above, $\pi_{k-1}$ will only need to be outside $Q_{k}^{\text{core}}(i^{\text{par}})$ to deviate around inner clusters in $Q_k(i^{\text{par}})$.
   
   As $Q_k(i)$ has negative feedback, there must exist some $Q_{k-1}(j_0)\in\pi_{k-1}$ that has negative feedback and without losing generality we may assume that it is the $(k-1)$-box on $\pi_{k-1}$ with negative feedback that is first discovered by traversing $\pi_{k-1}$ from $Q_{k-1}(i_0)$. Indeed, if such a $Q_{k-1}(j_0)$ did not exist, then $(\text{Fast}_{k-1})$ and \eqref{eq:pi_k-1_cond} would imply that $Q_k(i)$ has positive feedback. Let $\pi_{k-1}^{j_0}$ be the sub-path of $\pi_{k-1}$ from $Q_{k-1}(i_0)$ to $Q_{k-1}(j_0)$.
   
   If $Q_{k-1}(j_0)\in\pi_{k-1}^{\text{start}}$, then $\pi_{k-1}^{j_0}\subset\pi_{k-1}^{\text{start}}$. In this case, there exists constants $a_8,a_9>0$ that depend only on $A$ and $d$ such that 
   \begin{align}
   \tfrac{3}{L_{k-1}}\text{dist}(Q_{k-1}(j_0), \partial Q_k(i))\leq a_8\quad\text{and}\quad|\pi_{k-1}^{j_0}|\leq a_9. \label{dist_cond1}
   \end{align}
   If $Q_{k-1}(j_0)\in\pi_{k-1}^{\text{end}}$, then $Q_{k-1}(j_0)\subset Q_{k}^{\text{core}}(i^{\text{par}})\subset Q_k(i)$. As $C>C_0$ and $\lambda$ is such that $(\text{P}_{k-1}^C)$ holds, by Lemma~\ref{lem:NF_control}, then
   \begin{align}
   \tau_{k-1,j_0} > \tau_{k,i} + C\omega_{k-1}L_{k-1}\left(\tfrac{3}{L_{k-1}}\text{dist}(Q_{k-1}(j_0), \partial Q_k(i))-a_4\right). \label{C_cont1}
   \end{align}
   In other words, if $Q_{k-1}(j_0)\subset Q_k(i)$ but is not close to the boundary of $Q_k(i)$, we can use Lemma~\ref{lem:NF_control} to deduce a non-trivial lower bound on the difference between $\tau_{k-1,j_0}$ and $\tau_{k,i}$. 
   In this case,
   \begin{align}
   |\pi_{k-1}^{j_0}| \leq \tfrac{3}{L_{k-1}}\text{dist}(Q_{k-1}(j_0), \partial Q_k(i)) + A\sigma. \label{C_pre1}
   \end{align}
   By construction, every $(k-1)$-box on $\pi_{k-1}^{j_0}$ has positive feedback except for $Q_{k-1}(j_0)$. Thus, as $\lambda$ is such that $(\text{Fast}_{k-1})$ holds, then by \eqref{eq:PF_cons} we deduce
   \begin{align}
   \tau_{k-1,j_0} &< \tau_{k-1,i_0} + 2r_{k-1}L_{k-1}|\pi_{k-1}^{j_0}|,\nonumber\\
   & \leq \tau_{k,i} + 2r_{k-1}L_{k-1}\left(\tfrac{3}{L_{k-1}}\text{dist}(Q_{k-1}(j_0), \partial Q_k(i))+ A\sigma\right). \label{C_cont2}
   \end{align}
   If $\tfrac{3}{L_{k-1}}\text{dist}(Q_{k-1}(j_0), \partial Q_k(i))>a_4$ and $C>0$ is large enough so that 
   \begin{align}
   \label{eq:C_ineq_pf_claim}
   C\omega_{k-1}&\left(\tfrac{3}{L_{k-1}}\text{dist}(Q_{k-1}(j_0), \partial Q_k(i))-a_4\right)\nonumber\\
   & >  2r_{k-1}\left(\tfrac{3}{L_{k-1}}\text{dist}(Q_{k-1}(j_0), \partial Q_k(i))+ A\sigma\right),
   \end{align}
   then \eqref{C_cont1} contradicts \eqref{C_cont2}. Through rearranging the above inequality, we deduce that \eqref{eq:C_ineq_pf_claim} holds if 
   \begin{equation}
   C > \tfrac{2\sigma r_{k-1}}{\omega_{k-1}}a_3, \nonumber
   \end{equation}
   where $a_3>1$ is a constant that only depends on $A$ and $d$ which is exactly the choice of $a_3$ in the definition of $C_0$ in \eqref{def:C_0}. Moreover, if $C>C_0$ and $L_1$ is large enough so that $r<2r_1$ and $\omega>1/2$, then the above inequality implies that \eqref{eq:C_ineq_pf_claim} follows for all $k\geq2$. Consequently,
   \begin{align}
   \tfrac{3}{L_{k-1}}\text{dist}(Q_{k-1}(j_0), \partial Q_k(i))\leq a_4. \label{C_case1}
   \end{align}
   The claim follows through \eqref{dist_cond1}, \eqref{C_pre1}, \eqref{C_case1} and \eqref{F_k} by setting $\hat{\imath}_0=j_0$, $a_6=\max\{a_4,a_8\}$ and $a_7=\max\{a_4+A\sigma,a_9\}$.
\end{proof}

\begin{proof}[Proof of Lemma~\ref{lem:induction1} for \eqref{confk}]
   This is an immediate consequence of the proofs for the three fundamental properties above and Lemma~\ref{lem:source_ET}.
\end{proof}

\section{Proofs of Theorem~\ref{maintheorem} and Theorem~\ref{thm:strongsurvival}}
\label{sec:Proof_Thm1}

In this section we prove Theorem~\ref{maintheorem} and Theorem~\ref{thm:strongsurvival}. First we state and prove the following lemma.

\begin{lemma}
\label{lem:pp}
If $L_1$ is large enough and $\lambda$ small enough, then with positive probability $Q_k(o)$ has positive feedback for all $k\geq1$.
\end{lemma} 

\begin{proof}
By Lemma~\ref{lem:flawless} it suffices to prove that if the origin is flawless, then $Q_k(o)$ has positive feedback for all $k\geq1$. We proceed by induction and note that the case $k=1$ is immediate from the definition of positive feedback at scale 1. Now suppose $k>1$ and that the result holds up to scale $k-1$. As a consequence of the origin being flawless, we have
$$
Q_{k-1}(o)\in\mathcal{W}^{\text{inn}}(k,o),
$$
and by the inductive hypothesis $Q_{k-1}(o)$ has positive feedback. Moreover, as $\tau_{k-1,o}=\tau_{k,o}=0$ and $\eta_0(o)=1$, we have that $Q_k(o)$ has positive feedback at scale $k$.
\end{proof}

\begin{proof}[Proof of Theorem~\ref{maintheorem}]
	Fix $L_1$ large enough and then $\lambda$ small enough so that Lemma~\ref{lem:cascading} and Lemma~\ref{lem:pp} hold. Assume that $Q_\kappa(o)$ has positive feedback for all $\kappa\geq1$.
	Fix a large scale $k$ and consider $Q_{k}(o)$. 
	By Lemma~\ref{lem:cascading}, every $(k-1)$-box contained in $\mathcal{W}^{\text{inn}}(k,o)$ has positive feedback. 
	For ease of argument, let us focus on $(k-1)$-cores in $Q_k^{\text{core}}(o)\cap\mathcal{W}^{\text{inn}}(k,o)$. 
	The number of $(k-1)$-cores contained in $Q_k^{\text{core}}(o)\cap\mathcal{W}^{\text{inn}}(k,o)$ is at least
\begin{align*}
\left(\tfrac{L_k}{L_{k-1}}\right)^d - A\sigma = \left(\tfrac{L_k}{L_{k-1}}\right)^d\left(1-\tfrac{A\sigma}{k^{2d}L_{k-1}^{d(d-1)}}\right),
\end{align*}
	where the $A\sigma$ term counts for all $(k-1)$-boxes contained in clusters in $Q_k(o)$. 
	Through analogous reasoning, in each of these $(k-1)$-cores in $Q_k^{\text{core}}(o)\cap\mathcal{W}^{\text{inn}}(k,o)$, the number of $(k-2)$-cores that are not contained in a cluster of bad $(k-2)$-boxes in a good $(k-1)$-box is bounded below by
\begin{align*}
\left(\tfrac{L_{k-1}}{L_{k-2}}\right)^d - A\sigma = \left(\tfrac{L_{k-1}}{L_{k-2}}\right)^d\left(1-\tfrac{A\sigma}{(k-1)^{2d}L_{k-2}^{d(d-1)}}\right).
\end{align*}
Iterating this argument all the way to $k=1$ allows us to deduce that the number of $1$-cores contained in $Q_k^{\text{core}}(o)\cap\mathcal{W}^{\text{inn}}(k,o)$ that are not contained in a cluster of any good box up to scale $k$ is bounded below by
\begin{align*}
\left(\frac{L_k}{L_1}\right)^d\left(\prod_{j=1}^{k-1}\zeta_{j}\right) \quad\text{where}\quad \zeta_j = 1-\tfrac{A\sigma}{(j+1)^{2d}L_{j}^{d(d-1)}}.
\end{align*}
From $\zeta_j$ given above, we deduce $\{\prod_{j=1}^{k-1}\zeta_j\}_{k\geq1}$ is a decreasing sequence that is uniformly bounded away from 0, so that 
\begin{align}
\zeta=\lim_{k\to\infty}\prod_{j=1}^{k-1}\zeta_{j}>0, \label{eq:zeta}
\end{align}
where $\zeta$ depends only on $A$ and $d$.

	If $Q_1(g)\subset Q_k^{\text{core}}(o)\cap\mathcal{W}^{\text{inn}}(k,o)$ and is not contained in a cluster of any good box up to scale $k$, then $Q_1(g)$ has positive feedback as $Q_k(o)$ has positive feedback. Recall $\mathcal{C}^{-}_1(g)$ is the largest component of non-seeds in $Q_1(g)$ with sites in the boundary $\partial Q_1(g)$ removed. By Lemma~\ref{lem:main_lemma_1}, all sites in $\mathcal{C}^{-}_1(g)$ are occupied by $FPP_1$, and in particular, all sites in $\mathcal{C}_1(g)\cap Q_1^{\text{core}}(g)$. As $Q_1(g)$ is good, for $\varepsilon$ we fix in \eqref{def:E_1} we have 
\begin{align}
\left|\mathcal{C}_1(g)\cap Q_1^{\text{core}}(g)\right|\geq (1-\varepsilon)\theta(1-p)\left(\tfrac{L_1}{3}\right)^d, \label{eq:corecount}
\end{align}
where we recall that $\theta(1-p)>0$ as $p<1-p_c^{\text{site}}$. As $1$-cores partition $\mathbb{Z}^d$, we deduce from \eqref{eq:zeta} and \eqref{eq:corecount} that the number of sites that $FPP_1$ occupies in $Q_k(o)$ is bounded below by
\begin{align}
3^{-d}\zeta(1-\varepsilon)\theta(1-p)L_k^d . \label{eq:density}
\end{align}
By dividing \eqref{eq:density} by $L_k^d$, we have that a lower bound of the density of sites occupied by $FPP_1$ in $Q_k(o)$ is a positive constant that does not depend on $k$. The result follows by taking $k\to\infty$.
\end{proof}

Recall the critical probability $p''_c$ defined in~\eqref{p''_c}. To prove Theorem~\ref{thm:strongsurvival}, we augment the definition of a good $1$-box as we will need to exploit the fact that $p<p''_c$. 
Consider the $1$-box $Q_1(i)$. Enumerate the connected components of filled seeds in $Q_1(i)$ as 
$$
\mathcal{S}_1(i),\mathcal{S}_2(i),\ldots,\mathcal{S}_{m_i}(i),
$$ 
so that $|\mathcal{S}_1(i)|\geq|\mathcal{S}_2(i)|\geq\cdots\geq|\mathcal{S}_{m_i}(i)|$, where $m_i$ is the number of connected components of filled seeds in $Q_1(i)$. Define the event
$$
E_6(i) = \left\{|\mathcal{S}_1(i)|\leq L_1/100\right\}.
$$
By setting $M=L_1/100$ in~\eqref{p''_c}, we observe that we can include $E_6(i)$ in the definition of $Q_1(i)$ being good and all the previous results in our multi-scale analysis follow in the same manner for $p<p''_c$. 

The proof of Theorem~\ref{thm:strongsurvival} also relies on extending the notion of parent and progenitor to sites, which one can view as \emph{scale 0} boxes. 
We define a site to be good if it does not host a seed and bad otherwise. A good site has positive feedback if it is occupied by $FPP_1$ and has negative feedback 
if it is occupied by $FPP_{\lambda}$. The parent of a site $v$ is the neighbouring site from which $FPP_1$ or $FPP_{\lambda}$ propagated to that site at time $\tau(v)$. 
Any bad or negative feedback site $v$ (i.e.\ a site occupied by $FPP_{\lambda}$) can be traced through the parent structure to a unique seed that was activated by $FPP_1$, which we call the progenitor in an analogous manner to scales $k\geq1$. Hence both $v^{\text{prog}}$ and $\mathcal{P}_0^{v\leftarrow v^{\text{prog}}}$ are well-defined. With these considerations, we are now in a position to prove Theorem~\ref{thm:strongsurvival}.

\begin{proof}[Proof of Theorem~\ref{thm:strongsurvival}]
Suppose $d\geq2$, $p<\min\{p''_c,1-p_c^{\text{site}}\}$ and $L_1$ is large enough and $\lambda$ is small enough so that Lemma~\ref{lem:cascading}, Lemma~\ref{lem:pp} and Theorem~\ref{maintheorem} hold. Assume that $Q_\kappa(o)$ has positive feedback for all $\kappa\geq1$. In this scenario, $FPP_1$ survives with positive probability and so strong survival will be established if we prove all connected components of $FPP_{\lambda}$ are almost surely finite.

Let $S$ be an arbitrary component of seeds and $S'$ be the set of sites whose progenitor is contained in $S$. Fix a scale $k$ large enough so that $S\subset Q_k^{\text{core}}(o)$. We note that such a $k$ exists as $p<p''_c\leq p_c^{\text{site}}$, and so all components of seeds are almost surely finite.

To set up the proof, we introduce the following definition. For $\ell\geq1$ and $B\subset\mathbb{Z}^d$, we define a finite set of $\ell$-boxes $\mathcal{O}_\ell$ to be an \emph{annulus of $\ell$-boxes separating $B$ from infinity} if any infinite self-avoiding path from $B$ intersects the core of an $\ell$-box from $\mathcal{O}_\ell$ and, moreover, 
we can find two disjoint connected sets $Z_1$ and $Z_2$ such that 
\begin{align}
\mathbb{Z}^d \setminus \bigcup_{Q_\ell\in \mathcal{O}_\ell}Q_\ell^{\text{core}} = Z_1 \cup Z_2, \label{eq:Z}
\end{align}
where $Z_1$ is finite and contains $B$ while $Z_2$ is unbounded.

Let $\mathcal{O}_{k-1}\subset\mathcal{W}^{\text{inn}}(k,o)$ be an annulus of $(k-1)$-boxes separating $Q_k^{\text{core}}(o)$ from infinity. 
As $\lambda$ is such that \eqref{cask} holds, then all $(k-1)$-boxes in $\mathcal{O}_{k-1}$ have positive feedback.
Let $\mathcal{O}_{k-2}$ be an annulus of $(k-2)$-boxes separating $Q_k^{\text{core}}(o)$ from infinity such that 
$$
\mathcal{O}_{k-2} \subset \bigcup_{Q_{k-1}(j)\in\mathcal{O}_{k-1}}\mathcal{W}^{\text{inn}}(k-1,j).
$$
As all $(k-1)$-boxes in $\mathcal{O}_{k-1}$ have positive feedback and $\lambda$ is such that $(\text{CAS}_{k-1})$ holds, then all $(k-2)$-boxes in $\mathcal{O}_{k-2}$ also have positive feedback. We continue this procedure until we construct an annulus of $1$-boxes that separates $Q_k^{\text{core}}(o)$ from infinity, $\mathcal{O}_1$ say, such that all $1$-boxes on $\mathcal{O}_1$ have positive feedback.

Let $Z_2$ be the unbounded connected set induced by $\mathcal{O}_1$ as given in \eqref{eq:Z}, noting that $\mathbb{Z}^d\setminus Z_2$ must be a finite set that contains $S$. 
If $S'$ is infinite, then by the definition of annulus there must exist a box 
$Q_1(j)\in \mathcal{O}_1$ whose core contains a site from $S'$. Hence, $\mathcal{C}_1^-(j)$ must intersect $S'$, which contradicts the fact that 
$Q_1(j)$ is of positive feedback.

% Suppose for contradiction that $\mathcal{S}'\cap Z_2\neq\emptyset$. Let $x\in\mathcal{S}'\cap Z_2$ and $\mathcal{P}_0^{x\leftarrow x^{\text{prog}}}$ be the path of parents from $x$ to $x^{\text{prog}}$, recalling that $x^{\text{prog}}\in\mathcal{S}$. We observe that every site on $\mathcal{P}_0^{x\leftarrow x^{\text{prog}}}$ is occupied by $FPP_{\lambda}$.
% 
% Through this construction and a simple geometric argument, there exists a $1$-box $Q_1(i)\in\mathcal{O}_1$ and a path of sites $\gamma$ so that 
% $$
% \gamma\subset \left(Q_1(i)\setminus\partial Q_1(i)\right)\cap\mathcal{P}_0^{x\leftarrow x^{\text{prog}}} \quad\text{and}\quad |\gamma|\geq L_1/6.
% $$ 
% 
% As $Q_1(i)$ has positive feedback, then all sites in $\mathcal{C}_1^{-}(i)$ are occupied by $FPP_1$ (cf.\ Lemma~\ref{lem:cascading}), and thus $\gamma\cap\mathcal{C}_1^{-}(i)=\emptyset$. Consequently, $\gamma$ is composed entirely of filled seeds in $Q_1(i)$, and moreover, $\gamma$ is contained in a single connected component of filled seeds in $Q_1(i)$. As $Q_1(i)$ is good, then $|\gamma|\leq L_1/100$, giving the desired contradiction. Hence $\mathcal{S}'\subset\mathbb{Z}^d\setminus Z_2$ and thus $\mathcal{S}'$ is a finite set.
% 
\end{proof}

\section{Concluding remarks and open questions}
\label{sec:conclusion}

To establish the coexistence phase in Corollary~\ref{maincorollary}, we require that the $FPP_{\lambda}$ survives almost surely by letting $p\in(p'_c, 1-p_c^{\text{site}})$, 
which is only possible if $d\geq3$. 
This simplifies our analysis as we know that $FPP_\lambda$ survives almost surely in such situations, so by setting $\lambda$ small enough we can prove $FPP_1$ survives with positive probability as outlined 
in Theorem~\ref{maintheorem}. 
% We emphasise that the proof of Theorem \ref{maintheorem} is highly non-trivial, requiring the introduction of an entirely new framework of multi-scale analysis. 
The most natural question is whether or not there is a coexistence phase without the almost sure survival of $FPP_{\lambda}$. 

\begin{question}
For $d\geq3$, does there exist a coexistence phase for all $p\in(0,1-p_c^{\text{site}})$?
\end{question}

The coexistence phase established in Corollary~\ref{maincorollary} leads one to wonder if there could be coexistence when $p<p'_c$, but the case $d=2$ is wide open in the sense that it is not known if there is coexistence for any choice of parameters.

\begin{question}
Does there exist a coexistence phase for FPPHE when $d=2$?
\end{question}

The main difficulty in approaching these two questions is controlling for the survival of both processes simultaneously. By Theorem~\ref{maintheorem}, we have that for all $p<1-p_c^{\text{site}}$, 
if $\lambda$ is small enough then $FPP_1$ survives with positive probability but there is no information about the survival of $FPP_{\lambda}$ unless we assume that the seeds are supercritical or use an enhancement argument with filled seeds. 
For a fixed $p$ it could be the case that for some values of $\lambda$ there is coexistence but once $\lambda$ is sufficiently small, there is strong survival. This issue is compounded as the strong survival  
and coexistence may not be disjoint phases, in that there could be a choice of $\lambda$ and $p$ for which both occur with positive probability. 
This in fact happens when the graph $G$ is composed of two copies of $\mathbb{Z}^2$ 
that are connected by a single edge between their respective origins. 
By Sidoravicius and Stauffer \cite[Theorem 1.3]{sidoravicius2019multi}, there is a phase of strong survival on $\mathbb{Z}^2$. 
If $\lambda$ and $p$ are such that there is strong survival on $\mathbb{Z}^2$, then strong survival and coexistence both have a positive probability of occurring on $G$.
Nonetheless, we expect that, on $\Z^d$, strong survival and coexistence are indeed disjoint phases. 
% Hence establishing the circumstances in which the strong survival and coexistence phases are distinct is an interesting open question in of itself.

\begin{question}
For what graphs are the strong survival and coexistence phases distinct?
\end{question}

Even if we could assume that the strong survival and coexistence phases were distinct, distinguishing between them is still challenging due to the lack of monotonicity in FPPHE. 
As discussed in Section \ref{sec:related_models}, one can help the spread of $FPP_1$ by strategically planting additional seeds or increasing passage times and one can even construct graphs 
where increasing $p$ can increase the probability that $FPP_1$ survives. However, one would expect that if the underlying graph was transitive, then the probability of $FPP_1$ surviving 
would be a non-increasing function of $\lambda$ and $p$. 
%We refer to this property as \emph{weak monotonicity} 
%(since it does not require the existence of a coupling satisfying the properties explained in the case of the two-type Richardson model)
%and establishing for what graphs it holds on is an interesting question.
Establishing for what graphs this property holds on is an interesting question.

\begin{question}
For what graphs is the probability of $FPP_1$ surviving 
a non-increasing function of $\lambda$ and $p$?
\end{question}

In Corollary~\ref{maincorollary}, a coexistence phase is established in which both $FPP_1$ and $FPP_{\lambda}$ occupy a positive density of sites. It is possible that coexistence occurs with $FPP_1$ surviving while only occupying a set of zero density; a regime we refer 
to as \emph{weak coexistence}. It is not difficult to prove there is a phase of weak coexistence on (non-amenable) trees 
but establishing a phase of weak coexistence on $\mathbb{Z}^d$ remains an open problem.

\begin{question}
For $d\geq2$, does there exist a regime of weak coexistence on $\mathbb{Z}^d$?
\end{question}

One may consider an alternative definition of survival of $FPP_{\lambda}$ wherein the spread of $FPP_{\lambda}$ from a single seed gives rise to an infinite connected component of $FPP_{\lambda}$.
More precisely, we construct a infinite directed forest on $\mathbb{Z}^d$ encoding the spread of $FPP_1$ and $FPP_{\lambda}$ as follows.
If $FPP_{1}$ (resp.\ $FPP_{\lambda}$) spreads from $x$ to successfully occupy $y$, we attach a directed edge from $x$ to $y$.
When a seed is activated by the attempted occupation of $FPP_1$,  we attach no directed edge with endpoint at that seed.
We define there to be \emph{infinite geodesic survival} if there exists at least one infinite connected component of $FPP_{\lambda}$ in the directed forest above.
Clearly infinite geodesic survival implies survival  considered in this paper.
A natural question is whether there can be coexistence with this definition.

\begin{question}
For $d\geq2$, can $FPP_1$ and $FPP_{\lambda}$ observe infinite geodesic survival simultaneously with positive probability?
\end{question}

We can use our multi-scale analysis with non-equilibrium feedback to recover the strong survival phase established in \cite{sidoravicius2019multi}, 
in which for $\lambda\in(0,1)$, there is strong survival for sufficiently small $p$. 
We briefly indicate how this can be done now. 

Fix $\lambda\in(0,1)$. The aim is to prove that $FPP_1$ survives while all connected components of $FPP_{\lambda}$ are finite for a small enough choice of $p$. 
The standard multi-scale construction proceeds in an almost identical manner except we alter the definition of good $1$-box. We define a $1$-box to be good if $FPP_1$ can readily spread in the box and there are no seeds. 
By setting $L_1$ large enough and $p$ small enough, the standard multi-scale construction then follows readily, so that good boxes at all scales occur with sufficiently high probability. 

The idea of the proof is that in a good $k$-box, all bad boxes at lower scales (away from the boundary) are well-separated and one can construct an encapsulation argument to control the spread of $FPP_{\lambda}$ from these bad boxes. 
Indeed, as $FPP_{\lambda}$ must spread from bad boxes, it is reasonable to construct a parent and progenitor structure analogous to the one in Section~\ref{sec:Parent_progenitor_def} to identify the good boxes that 
$FPP_{\lambda}$ manages to occupy sites within. For this, we tune the definition of positive and negative feedback so that $FPP_1$ occupies all sites in a good box that are sufficiently far from bad boxes at lower scales. 
For example, a $1$-box has positive feedback if $FPP_1$ occupies a site away from the boundary at a relatively fast time after the box was first entered. 
Through a cascading argument, $FPP_{\lambda}$ is confined to a collection of regions in this $k$-box that can no longer spread while $FPP_1$ occupies every other site. 
By taking the limit $k\to\infty$, $FPP_1$ survives while $FPP_{\lambda}$ fails to observe an infinite connected component. 
The spirit of this proof is close to the one given in \cite{sidoravicius2019multi}, but our multi-scale analysis with non-equilibrium feedback offers a more systematic approach 
through the introduction of positive and negative feedback, leading to a clearer and more streamlined proof.

\section*{Acknowledgements}
A.\ Stauffer is thankful to Ioan Manolescu for useful discussions. Part of this work was completed during a visit by T.\ Finn to the University of Roma Tre, for which he is thankful for their generous hospitality. 

\bibliographystyle{abbrv}

\bibliography{references(2)}

\end{document}